\documentclass[12pt]{article}
 
 \usepackage{amsmath,amssymb,amscd,amsthm,esint}
 
\usepackage{graphics,amsmath,amssymb,amsthm,mathrsfs}

\newtheorem{theorem}{Theorem}[section]
\newtheorem{prop}[theorem]{Proposition}
\newtheorem{lemma}[theorem]{Lemma}

\newtheorem{remark}[theorem]{Remark}

\def\xxint#1#2#3{{\setbox0=\hbox{$#1{#2#3}{\int}$}
  \vcenter{\hbox{$#2#3$}}\kern-.5\wd0}}

\def \rd {{\mathbb R}^d}
\def\td{\mathbb{T}^d}

\def\loc{\text{\rm loc}}

\def\e{\varepsilon}

\def\wn{\widetilde{n}}

\newcommand{\average}{-\!\!\!\!\!\!\int}

\begin{document}

\title
{\bf Boundary Layers  in Periodic Homogenization of Neumann Problems}

\author{
Zhongwei Shen\thanks{Supported in part by NSF grant DMS-1600520.}
\qquad Jinping Zhuge \thanks{Supported in part by NSF grant DMS-1161154.}}

%\thanks{Corresponding author,
% Department of Mathematics, University of Kentucky, Lexington, Kentucky 40506, USA. 
% Tel. 1-859-257-3231. Fax. 1-859-257-4078. Email: zshen2@uky.edu.}}

\date{ }

\maketitle

\begin{abstract}
This paper is concerned with a family of second-order elliptic systems in divergence form with rapidly oscillating 
periodic coefficients.
We initiate the study of homogenization and boundary layers for
 Neumann problems with first-order oscillating boundary data.
 We identify the homogenized system and establish 
the sharp rate of convergence in $L^2$ in dimension three or higher.
Sharp regularity estimates are also obtained for the homogenized boundary data in both Dirichlet and Neumann 
problems.
The results are used to establish  a higher-order convergence rate
 for Neumann problems with non-oscillating data.

\bigskip



\end{abstract}

  % \tableofcontents

%%%%%%%%%%%%%%%%%%%%%%%%%%%%%%%%%%%%%%%%%%%%%%%%%%%

\section{Introduction}
\setcounter{equation}{0}

The main purpose of this paper is to study the homogenization and boundary layers
for  the Neumann problem
with first-order oscillating boundary data,
\begin{equation}\label{NP-0}
\left\{
\aligned
& \mathcal{L}_\e (u_\e)  =0 &\quad & \text{ in } \Omega,\\
& \frac{\partial u_\e}{\partial \nu_\e}
  = T_{ij}\cdot \nabla_x \big\{ g_{ij}(x, x/\e) \big\}  +g_0(x, x/\e)
  & \quad & \text{ on } \partial \Omega,
 \endaligned
 \right.
\end{equation}
where $T_{ij}=(n_i e_j -n_j e_i)$ is a tangential vector field on $\partial\Omega$,
\begin{equation}\label{operator}
\mathcal{L}_\e =-\text{\rm div}(A(x/\e)\nabla)
\end{equation}
a second-order elliptic system in divergence form with oscillating periodic coefficients, and
$\partial u_\e/\partial\nu_\e=n\cdot A(x/\e)\nabla u_\e$ denotes the conormal derivative associated with 
the operator $\mathcal{L}_\e$.

Throughout this paper, unless otherwise stated,
we assume that 
 $\Omega$ is a bounded smooth, strictly  convex domain in $\rd$ and that
 $g_0(x, y), \, g_{ij}(x, y)$ are smooth in $(x, y)\in \rd \times \rd$ and 
 1-periodic in $y$, 
 \begin{equation}\label{periodicity-g}
 g(x, y+z)=g(x,y) \quad \text{ for any } x, y\in \rd \text{ and } z\in \mathbb{Z}^d.
 \end{equation}
We also assume that the coefficient matrix  $A=A(y)=\big( a_{ij}^{\alpha\beta}\big)$,
with $1\le i, j\le d$ and $1\le \alpha, \beta\le m$, is smooth and satisfies the ellipticity condition,
\begin{equation}\label{ellipticity}
\mu |\xi|^2 \le a_{ij}^{\alpha\beta} (y) \xi_i^\alpha \xi_j^\beta \le \frac{1}{\mu} |\xi|^2 \quad 
\text{ for any } \xi=(\xi_i^\alpha)\in \mathbb{R}^{m\times d},
\end{equation}
where $\mu>0$, and the periodicity condition,
\begin{equation}\label{periodicity}
A(y+z)=A(y) \quad \text{ for any } y\in \rd \text{ and } z\in \mathbb{Z}^d.
\end{equation}
Under these conditions we will show that as $\e\to 0$,
the unique solution of (\ref{NP-0}) with
$\int_{\Omega} u_\e =0$ converges strongly in $L^2(\Omega)$ to $u_0$, where
$u_0$ is a solution of
\begin{equation}\label{NP-h}
\left\{
\aligned
& \mathcal{L}_0 (u_0)  =0 &\quad & \text{ in } \Omega,\\
 &\frac{\partial u_0}{\partial \nu_0}
 = T_{ij}\cdot \nabla_x \overline{g}_{ij}  +\langle g_0 \rangle & \quad & \text{ on } \partial \Omega.
 \endaligned
 \right.
\end{equation}
 The operator $\mathcal{L}_0$ is given by
$\mathcal{L}_0 =-\text{\rm div}(\widehat{A}\nabla)$, with $\widehat{A}$ being the usual homogenized 
matrix of $A$, and
\begin{equation}\label{g-bar}
\langle g_0\rangle (x) =\average_{\td} g_0 (x, y)\, dy.
\end{equation}
 The formula for the function $\{ \overline{g}_{ij}\} $ on $\partial\Omega$, which is much more involved and
 will be given explicitly in Section 6, 
shows that  its value at $x\in \partial\Omega$ depends only on $A$, $\{ g_{ij}(x, \cdot)\}$,
 and the outward normal $n$ to $\partial\Omega$
at $x$.
Moreover, we establish a convergence rate in $L^2$, which is optimal  
(up to an arbitrarily  small exponent) for $d\ge 3$.

\begin{theorem}\label{main-theorem-1}
Let $\Omega$ be a bounded smooth, strictly convex domain in $\rd$, 
$d\ge 3$.
Assume that $A(y)$, $g_0 (x, y)$ and $g_{ij}(x, y)$ are smooth and
satisfy conditions (\ref{periodicity-g}), (\ref{ellipticity}) and (\ref{periodicity}).
Let $u_\e$ and $u_0$ be the solutions of (\ref{NP-0}) and (\ref{NP-h}), respectively,
with $\int_{\Omega} u_\e=\int_{ \Omega}  u_0=0$.
Then for any $\sigma \in (0,1/2)$ and $\e\in (0,1)$,
\begin{equation}\label{rate-0}
\| u_\e -u_0 \|_{L^2(\Omega)}
\le C_\sigma\,  \e^{\frac12-\sigma},
\end{equation}
where $C_\sigma$ depends only on  $d$, $m$, $\sigma$, $A$, $\Omega$, and $g=\{g_0, g_{ij}\}$.
Furthermore, the function $\overline{g}=\{ \overline{g}_{ij}\}$ in (\ref{NP-h}) satisfies
\begin{equation}\label{Sobolev}
\| \overline{g} \|_{ W^{1, q} (\partial\Omega)} \le C_q\, \sup_{y\in \td}
\| g(\cdot, y)\|_{C^1(\partial\Omega)} \quad  \text{ for any } q< d-1,
\end{equation}
where $C_q$ depends only on $d$, $m$, $\mu$, $q$ and $\|A\|_{C^k(\td)}$ for some
$k=k(d)\ge 1$.
\end{theorem}

In recent years there has been considerable interest in the homogenization of
boundary value problems with oscillating boundary data
\cite{Masmoudi-2011, Masmoudi-2012, Prange-2013, ASS-2013, KLS4, ASS-2014, Feldman-2014, Kim-2014, Kim-2015, ASS-2015, Armstrong-Prange-2016} (also see related 
earlier work in \cite{Santosa-Vogelius, Santosa-Vogelius-erratum, Moskow-Vogelius-1997, Moskow-Vogelius-2,
Allaire-Amar-1999}.
In the case of Dirichlet problem,
\begin{equation}\label{DP-0}
\mathcal{L}_\e (u_\e) =0 \quad   \text{ in } \Omega \quad \text{ and } \quad
 u_\e =g(x, x/\e) \quad \text{ on } \partial\Omega,
\end{equation}
where $g(x, y)$ is  assumed to be periodic in $y$,
major progress was made in \cite{Masmoudi-2012} and more recently in \cite{Armstrong-Prange-2016}.
Let $u_\e$ be the solution of (\ref{DP-0}).
Under the assumption that $\Omega$ is smooth and strictly convex in $\rd$,
$d\ge 2$, it was proved in \cite{Masmoudi-2012} that
$$
\| u_\e -u_0\|_{L^2(\Omega)} \le C \, \e^{\frac{(d-1)}{3d+5}-\sigma}
$$
for any $\sigma \in  (0,1)$, where $u_0$ is the solution of the homogenized problem,
\begin{equation}\label{DP-h}
\mathcal{L}_0 (u_0) =0 \quad   \text{ in } \Omega \quad \text{ and } \quad
 u_0 =\overline{g} \quad \text{ on } \partial\Omega,
\end{equation}
and the homogenized data $\overline{g}$ at $x$ depends on $g(x, \cdot)$, $A$, and $n(x)$.
A sharper  rate of convergence  in $L^2$ was obtained recently in \cite{Armstrong-Prange-2016} for 
the Dirichlet problem (\ref{DP-0}), with $O(\e^{\frac12-})$ for $d\ge 4$, 
$O(\e^{\frac{1}{3}-})$ for $d=3$, and
$O(\e^{\frac16-})$ for $d=2$.
As demonstrated in \cite{ASS-2015} in the case of elliptic equations with constant coefficients,
the optimal rate would be $O(\e^{1/2})$ for $d\ge 3$ (up to a factor of $\ln \e$ in the case of $d=3$),
and $O(\e^{1/4})$ for $d=2$.
Thus the convergence rates obtained in \cite{Armstrong-Prange-2016}
for (\ref{DP-0}) is optimal, up to an arbitrarily small exponent, for $d\ge 4$.
With the sharp results established in this paper for the homogenized data (see Theorem \ref{main-theorem-2} below),
the method used in \cite{Armstrong-Prange-2016} also leads to the optimal rate of convergence 
for $d=2$ or $3$.

In the case of the Neumann problem with only zero-order oscillating data $g_0(x, x/\e)$, i.e., $g_{ij}(x, y)=0$,
the homogenization of (\ref{NP-0}) is well understood, mostly due to the fact that the Neumann function 
$N_\e(x, y)$ for $\mathcal{L}_\e$ in $\Omega$
converges pointwise to $N_0 (x, y)$, the Neumann function for the homogenized 
operator $\mathcal{L}_0$ in $\Omega$. In fact, it was proved in \cite{KLS4} that if $\Omega$ is a bounded
$C^{1,1}$ domain in $\rd$ and $d\ge 3$, then
\begin{equation}\label{N-function-0}
|N_\e (x, y) -N_0(x, y)|\le \frac{C\,  \e \ln \big[ \e^{-1}|x-y| +2\big]}{|x-y|^{d-2}}
\end{equation}
for any $x, y\in \Omega$.
This effectively reduces the problem to the case of operators with constant coefficients, 
which may be handled by the method of oscillatory integrals \cite{ASS-2013, ASS-2015}.
Thus the real challenge  for the Neumann problem starts with the first-order oscillating boundary
 data that includes terms 
in the form of $\e^{-1} g(x, x/\e)$.
Since the Neumann data involves  the first-order derivatives of the solution,
the problem (\ref{NP-0}) seems quite natural.
We remark that one of the main motivations for studying Dirichlet problem (\ref{DP-0})
is its applications  to the higher-order convergence in the two-scale expansions of solutions to
Dirichlet problems for $\mathcal{L}_\e$ with non-oscillating boundary data.
As we will show in Section 9, 
in the study of the higher-order convergence of solutions to the Neumann problems
for $\mathcal{L}_\e$ with non-oscillating boundary data, 
 one is forced to deal with a Neumann problem in the form of (\ref{NP-0}).

We now describe the main steps in the proof of Theorem \ref{main-theorem-1} 
as well as the organization of the paper.
Our general approach is inspired by the recent work \cite{Armstrong-Prange-2016}.
The starting point is the following 
asymptotic expansion for the Neumann function $N_\e(x, y)$,
\begin{equation}\label{N-function-a}
\Big| \frac{\partial}{\partial y_k} \big\{  N_\e (x, y) \big\} 
-\frac{\partial}{\partial y_k} \big\{ \Psi^*_{\e, \ell} (y) \big\}
\frac{\partial}{ \partial {y_\ell}} \big\{  N_0 (x, y)\big\}\Big|
\le \frac{C_\sigma\,  \e^{1-\sigma}}{|x-y|^{d-\sigma}},
\end{equation}
for any $x, y\in \Omega$ and $\sigma \in (0,1)$.
The function $\Psi_{\e, \ell}^*$ in (\ref{N-function-a})
is the so-called Neumann corrector, defined by
(\ref{N-corrector}),  for the adjoint operator $\mathcal{L}_\e^*$.
The estimate (\ref{N-function-a}) was proved in \cite{KLS4} 
for bounded $C^{2, \alpha}$ domains, using boundary Lipschitz estimates
for solutions with Neumann data \cite{KLS1} (also see \cite{AS}, which extends the estimates to
operators without the symmetry condition $A^*=A$).
Using (\ref{N-function-a}), one reduces the problem to the study of the tangential derivatives of
$\Psi_{\e, \ell}^*$ on the boundary.
This preliminary reduction is carried out in Section 2.

Next, to analyze  the Neumann corrector on the boundary,
we approximate it in a neighborhood of a point $x_0\in \partial\Omega$,
by a solution of a Neumann problem in a half-space,
\begin{equation}\label{half-space-NP-0}
\left\{
\aligned
&\mathcal{L}_\e^* (\phi^*_{\e, k})=0 &\quad & \text{ in } \mathbb{H}_n^d(a),\\
& \frac{\partial}{\partial \nu_\e^*} \big(\phi^*_{\e, k}\big)
=-n_\ell b^*_{\ell k} (x/\e) &\quad & \text{ on } \partial\mathbb{H}_n^d(a),
\endaligned
\right.
\end{equation}
where  $n=n(x_0)$,
$\partial\mathbb{H}_n^d(a)$ is the tangent plane for $\partial\Omega$ at $x_0$, and
$b^*_{\ell k} (y)$ are some 1-periodic functions depending only on $A$.
Under the assumption that $n$ satisfies the Diophantine condition (\ref{D-condition}),
a solution to (\ref{half-space-NP-0})
is constructed in Section 3  by solving a Neumann problem for a degenerated elliptic system
in $\td\times \mathbb{R}_+$,  in a similar fashion as in the case of the Dirichlet problem \cite{Masmoudi-2011, Masmoudi-2012}.
In Section 4 we prove two refined estimates with bounding constants
$C$ independent of the constant $\kappa$ in the Diophantine condition (\ref{D-condition}).
The first estimate shows that the solution constructed for (\ref{half-space-NP-0})
satisfies 
\begin{equation}\label{decay-0}
|\nabla \phi_{\e, k}^* (x)|\le \frac{C \, \e}{|x\cdot n +a|},
\end{equation}
where $C$ depends only on $d$, $m$, $\mu$ and some H\"older norm of $A$.
We remark that the estimate (\ref{decay-0}) plays the same role as the maximum principle 
in the Dirichlet problem.
The other estimates in Section 4 concern with weighted norm inequalities for $\mathcal{L}_1$
in a half-space $H^d_n(a)$.
In particular, we show that if $u$ is  a solution of
$\mathcal{L}_1 (u)=\text{\rm div}(f)$ in $\mathbb{H}_n^d(a)$,
with either Dirichlet condition $u=0$ or  the Neumann condition
$\frac{\partial u}{\partial \nu} =-n \cdot f$ on $\partial\mathbb{H}_n^d(a)$, then
\begin{equation}\label{weighted-0}
\int_{\mathbb{H}_n^d (a)} |\nabla u (x) |^2  \big[\delta (x)\big]^\alpha\, dx
\le C_\alpha \int _{\mathbb{H}_n^d (a)} |f (x)  |^2  \big[\delta (x)\big]^\alpha\, dx,
\end{equation}
where  $-1<\alpha<1$ and 
$$
\delta(x) =\text{dist}(x, \partial \mathbb{H}_n^d (a)) =|x\cdot n +a|.
$$
We point out that it is this weighted estimate with $-1<\alpha<0$
that allows us to prove the optimal 
result (\ref{Sobolev}) for the homogenized data.
We further remark that the weighted estimate (\ref{weighted-0}), which is obtained
by using a real-variable method and the theory of $A_p$ weights in harmonic analysis,
 also leads to the optimal 
result for the oscillating Dirichlet problem (\ref{DP-0}).

\begin{theorem}\label{main-theorem-2}
Let $\Omega$ be a bounded smooth, strictly convex domain in $\rd$, $d\ge 2$.
Assume that $A(y)$ and $g(x, y)$ are smooth and satisfy conditions
(\ref{periodicity-g}), (\ref{ellipticity}) and (\ref{periodicity}). 
Let $\overline{g}$ be the homogenized data in (\ref{DP-h}).
Then
\begin{equation}\label{Sobolev-D}
\| \overline{g}\|_{W^{1, q}(\partial\Omega)} \le C_q\, \sup_{y\in \td} \| g(\cdot, y)\|_{C^1(\partial\Omega)}
 \quad \text{ for any } q<d-1,
\end{equation}
where $C_q$ depends only on $d$, $m$, $\mu$, $q$ and $\|A\|_{C^k(\td)}$ for some
$k=k(d)\ge 1$.
Moreover, 
\begin{equation}\label{rate-DP-low}
\| u_\e -u_0\|_{L^2(\Omega)}
\le C_\sigma \left\{
\aligned
& \e^{\frac12-\sigma} & \quad & \text{ if } d=3,\\
& \e^{\frac14-\sigma} & \quad & \text{ if } d=2,
\endaligned
\right.
\end{equation}
where $u_\e$ and $u_0$ are solutions of Dirichlet problems (\ref{DP-0})
and (\ref{DP-h}), respectively.
\end{theorem}

Estimate (\ref{Sobolev-D}) improves a similar
result in \cite{Armstrong-Prange-2016}, where it was proved that 
$\nabla \overline{g}\in L^{q, \infty}(\partial\Omega)$ for $q=\frac{2(d-1)}{3}$ and $d\ge 3$,
and that $\overline{g}\in W^{1, s}(\partial\Omega)$ for $s<\frac{2}{3}$ and $d=2$.
In an earlier work \cite{Masmoudi-2012}, it was shown that $\nabla\overline{g}\in L^{q, \infty}(\partial\Omega)$
for $q=\frac{d-1}{2}$. We believe that our results (\ref{Sobolev}) and (\ref{Sobolev-D})
are optimal, since the estimate $\overline{g}\in W^{1, q}(\partial\Omega)$
for some $q>d-1$ would imply that $\overline{g}$ is continuous on $\partial\Omega$.
As indicated before,
the convergence rate (\ref{rate-DP-low}), which is optimal,
improves the main results in \cite{Armstrong-Prange-2016} for $d=2$ and $3$.

One of the main technical lemmas in this paper is given in Section 5.
Motivated by  \cite{Armstrong-Prange-2016} for Dirichlet problem, we show that  for any $\sigma\in (0,1)$
and $\e\in (0,1)$,
\begin{equation}\label{main-estimate-0-1}
\aligned
& \| \nabla \Big( \Psi_{\e, k}^* - x_k -\e \chi_k^* (x/\e) -\phi_{\e, k}^* \Big)\|_{L^\infty(B(x_0, r)\cap\Omega)}\\
&\qquad \qquad
 \le C\sqrt{\e} \big\{ 1+ |\ln \e|\big\}
+C \e^{-1-\sigma} r^{2+\sigma},
\endaligned
\end{equation}
where $\e\le r\le  \sqrt{\e}$.
This is done by using boundary Lipschitz estimates and representations by Neumann functions $N_\e(x, y)$.
In contrast to the case of Dirichlet problem (\ref{DP-0}), 
in order to fully utilize the decay estimates for derivatives of Neumann functions,
the key insight  here is to transfer,
through integration by parts on the boundary, the derivatives from the Neumann data 
to the Neumann functions.

In Section 6 we use the weighted estimates in Section 4 to prove some weighted estimates
for solutions of the degenerated elliptic systems in $\td\times \mathbb{R}_+$.
This, together with a partition of unity for
$\partial\Omega$ that adapted to the constant $\kappa$ in the Diophantine condition (\ref{D-condition}),
 leads to the proof of (\ref{Sobolev}) and (\ref{Sobolev-D}) for the homogenized boundary data.
We remark that the partition of unity  for $\partial\Omega$
 we construct in Section 7 is in the same spirit as that in \cite{Armstrong-Prange-2016}.
 Our $L^p$-based approach with a finite $p>d-1$, which is motivated 
 by \cite{Fefferman-1996, Shen-1998},
 seems to make the proof much more transparent.
 Moreover, it gives us a pointwise estimate for
 $|u_\e (x) - u_0 (x)|$ in terms of the Hardy-Littlewood maximal function of
 $\kappa^{-q}$ with $q<d-1$ on the boundary, when $x$ is not in  a boundary layer 
 $\Omega_\e$ of measure  less than $C \e^{1-\sigma}$ (see Remark \ref{remark-8.1}).
 
 Our main result on the convergence rate  in Theorem \ref{main-theorem-1}  is proved 
 in Section 8, using a similar  line of argument as in \cite{Armstrong-Prange-2016}.
 We mention that the convergence rates in $L^p$ for $2<p<\infty$, even in 
 $H^\alpha(\Omega)$ for $0<\alpha<1/2$, follow from (\ref{rate-0})
 by interpolation (see Remark \ref{remark-2.10}).
 Finally, in Section 9, we use  Theorem \ref{main-theorem-1} to establish 
 a higher-order convergence in the two-scale expansions 
 of solutions to  Neumann problems with non-oscillating data.
 Although some related work may be found in \cite{Moskow-Vogelius-2},
 to the best of the authors' knowledge, 
 our $O(\e^{3/2})$ rate in $L^2(\Omega)$ with a homogenized first-order term
  seems to be
   the first result on the higher-order convergence for $\mathcal{L}_\e$ with Neumann conditions.

The summation convention will be used throughout.
We will use $C$ and $c$ to denote constants that depend at most on $d$, $m$, $A$ and $\Omega$
as well as other parameters,
but never on $\e$ or the constant $\kappa$ in (\ref{D-condition}).
We will always assume that $d\ge 3$ (unless otherwise indicated).
Our method, with some modifications, should also apply to the two-dimensional Neumann problem,
which will be handled in a separate work.

\medskip

\noindent{\bf Acknowledgement.}
The first author would like to thank David G\'erard-Varet for bringing the problem of homogenization 
of the Neumann problem (\ref{NP-0}) to his attention.

%%%%%%%%%%%%%%%%%%%%%%%%%%%%%%%%%%%%%%%%%%%%%%%%%%%

\section{Preliminaries}
\setcounter{equation}{0}

Under the conditions (\ref{ellipticity})-(\ref{periodicity}) and $A\in C^\sigma(\td)$ for some $\sigma \in (0,1)$,
 one may construct an $m\times m$ matrix of Neumann functions
$N_\e (x, y)=(N_\e^{\alpha\beta} (x, y))$  in a bounded $C^{1, \alpha}$ domain $\Omega$,
such that
\begin{equation}\label{N-def}
\left\{ 
\aligned
&\mathcal{L}_\e \big\{ N_\e (\cdot, y)\big\} = \delta_ y(x)I 
& \quad & \text{ in } \Omega,\\
& \frac{\partial}{\partial \nu_\e} \big\{ N_\e (\cdot, y) \big\}=- |\partial\Omega|^{-1} I
&\quad & \text{ on } \partial\Omega,\\
& \int_{\partial\Omega} N_\e (x, y)\, dx=0,
\endaligned
\right.
\end{equation}
where $I=I_{m\times m} $ and
the operator $\mathcal{L}_\e$ acts on each column of $N_\e(\cdot, y)$.
Let $u_\e\in H^1(\Omega; \mathbb{R}^m)$ be a solution to $\mathcal{L}_\e (u_\e)=F$ in $\Omega$
with $\frac{\partial u_\e}{\partial \nu_\e}= h$ on $\partial\Omega$, 
then
\begin{equation}\label{N-representation}
u_\e (x) -\average_{\partial\Omega} u_\e
=\int_{\partial\Omega} N_\e (x, y) F(y)\, dy
+\int_{\partial\Omega} N_\e(x, y) h(y)\, d\sigma (y)
\end{equation}
for any $x\in \Omega$.
If $d\ge 3$,
the Neumann functions satisfy the following estimates,
\begin{equation}\label{N-estimate}
\aligned
|N_\e (x, y)| & \le C\, |x-y|^{2-d}\\
|\nabla_x N_\e (x, y)| +|\nabla_y N_\e (x, y)| & \le C |x-y|^{1-d},\\
|\nabla_x \nabla_y N_\e (x, y)| & \le C |x-y|^{-d},
\endaligned
\end{equation}
for any $x, y\in \Omega$. 
This was proved in \cite{KLS1}, using boundary Lipschitz estimates with Neumann conditions,
which require the additional assumption $A^*=A$.
This additional assumption for the boundary Lipschitz estimates was removed later in \cite{AS}.
As a result, the estimates in (\ref{N-estimate}) hold if
$A$ satisfies (\ref{ellipticity})-(\ref{periodicity}) and is H\"older continuous.
Note that if $x, y, z\in \Omega$ and $|x-z|\le (1/2) |x-y|$, it follows from (\ref{N-estimate}) that 
 \begin{equation}\label{N-estimate-1}
 \aligned
 |N_\e (x, y)-N_\e (z, y) | & \le \frac{C|x-z|}{|x-y|^{d-1}},\\
  |\nabla_y \big\{ N_\e (x, y)-N_\e (z, y)\big\} | & \le \frac{C|x-z|}{|x-y|^{d}}.
  \endaligned
  \end{equation}

Let $\chi_j (y) =(\chi_j^{\alpha\beta} (y))$ denote the matrix of correctors for $\mathcal{L}_\e$ in 
$\rd$, defined by
\begin{equation}\label{corrector}
\left\{
\aligned
& \mathcal{L}_1\big(\chi_j^\beta + P_j^\beta\big) =0 \quad \text{ in } \rd,\\
& \chi_j^\beta \text{ is 1-periodic  and }
\int_{\td} \chi_j^\beta =0,
\endaligned
\right.
\end{equation}
where  $\chi_j^\beta =(\chi_j^{1 \beta}, \dots, 
\chi_j^{m\beta})$
and $P_j^\beta (y) =y_j (0, \dots, 1, \dots, 0)$ with $1$ in the $\beta^{th}$ position.
The homogenized operator $\mathcal{L}_0$ for $\mathcal{L}_\e$ is given by
$\mathcal{L}_0=-\text{\rm div} (\widehat{A}\nabla)$, where $\widehat{A}=\big( \widehat{a}_{ij}^{\alpha\beta} \big)$
is the homogenized matrix  with
\begin{equation}\label{homo}
\widehat{a}_{ij}^{\alpha\beta}
=\average_{\td} \Big\{ a_{ij}^{\alpha\beta}
+ a_{ik}^{\alpha\gamma} \frac{\partial}{\partial y_k} \big( \chi_j^{\gamma\beta} \big)\Big\}.
\end{equation}
To study the boundary regularity for solutions of Neumann problems, 
the matrix of Neumann correctors  
$\Psi_{\e, j}^\beta =\big( \Psi_{\e, j}^{\alpha\beta}\big)$ for $\mathcal{L}_\e$ in $\Omega$, defined by
\begin{equation}\label{N-corrector}
\mathcal{L}_\e \big(\Psi_{\e, j}^\beta\big)=0 \quad \text{ in } \Omega
\quad \text{ and } \quad 
\frac{\partial}{\partial \nu_\e} \big( \Psi_{\e, j}^\beta\big)
=\frac{\partial}{\partial \nu_0} \big( P_j^\beta \big) \quad \text{ on } \partial\Omega,
\end{equation}
was introduced in \cite{KLS1}, where $\partial u/\partial\nu_0$ denotes the conormal derivative 
associated with $\mathcal{L}_0$. One of the main estimates in \cite{KLS1} is the following Lipschitz estimate for
$\Psi_{\e, j}^\beta$,
\begin{equation}\label{Lip-Psi}
\|\nabla \Psi_{\e, j}^\beta\|_{L^\infty(\Omega)} \le C.
\end{equation}

Let $N_0(x, y)$ denote the matrix of Neumann functions for $\mathcal{L}_0$ in $\Omega$.
It was proved in  \cite{KLS4} that  if $\Omega$ is $C^{1,1}$, 
\begin{equation}\label{size-diff}
|N_\e (x, y) -N_0(x, y)|\le \frac{C\e \ln \big[ \e^{-1}|x-y| +2\big]}{|x-y|^{d-1}}
\end{equation}
for any $x, y\in \Omega$, and that if $\Omega$ is $C^{2, \alpha}$ for some $\alpha\in (0, 1)$,
\begin{equation}\label{derivative-diff}
\Big|\frac{\partial}{\partial y_i} \big\{ N^{\gamma \alpha}_\e (x, y) \big\}
-\frac{\partial}{\partial y_i} \big\{ \Psi_{\e, j}^{*\alpha\beta} (y) \big\}
\cdot
\frac{\partial}{\partial y_j } \big\{ N_0^{\gamma\beta} (x, y) \big\}\Big|
\le \frac{C_\sigma\,  \e^{1-\sigma}}{|x-y|^{d-\sigma}}
\end{equation}
for any $x, y\in \Omega$ and $\sigma\in (0,1)$.
The functions $(\Psi_{\e, j}^{*\alpha\beta})$ in (\ref{derivative-diff})
are the Neumann correctors, defined as in (\ref{N-corrector}),
 for the adjoint operator $\mathcal{L}_\e^*$
in $\Omega$.
We remark that these estimates as well as (\ref{Lip-Psi})
were proved in \cite{KLS1} under the additional assumption 
$A^*=A$. As in the case of (\ref{N-estimate}), with the results in \cite{AS}, 
they continue to hold without this assumption.

The estimates (\ref{size-diff}) and (\ref{derivative-diff}) mark the starting point of our investigation of
the Neumann problem (\ref{NP-0}) with oscillating data.
Indeed, let $u_\e$ be the solution of (\ref{NP-0}) with $\int_{\partial\Omega} u_\e=0$.
It follows by (\ref{N-representation}) that
\begin{equation}\label{p-1}
u_\e (x) =\int_{\partial\Omega}
N_\e (x, y)  (T_{ij}(y)\cdot \nabla_y) \big\{ g_{ij}(y, y/\e)\big\}\, d\sigma (y)
+\int_{\partial\Omega} N_\e(x, y) g_0(y, y/\e)\, d\sigma (y),
\end{equation}
where $T_{ij}=n_i e_j -n_j e_i$, $n=(n_1, \cdots, n_d)$ is the outward normal to $\partial\Omega$, and
 $e_i =(0, \dots, 1, \dots, 0)$ with $1$ in the $i^{th}$ position.

\begin{lemma}\label{lemma-p1}
Let $\Omega$ be a bounded Lipschitz domain in $\rd$.
Then, for $u, v\in C^1({\partial\Omega})$, 
\begin{equation}\label{parts}
\int_{\partial\Omega} ( (n_i e_j -n_j e_i)\cdot \nabla u ) \, v \, d\sigma
=-\int_{\partial\Omega} u  \, ((n_ie_j -n_j e_i)\cdot \nabla v)\, d\sigma.
\end{equation}
\end{lemma}

\begin{proof}
This may be proved by localizing the functions $u$ and $v$ to a small neighborhood of
$x_0\in\partial\Omega$ and reducing the problem to the case of flat boundary $\mathbb{R}^{d-1}$.
One may also prove (\ref{parts}) by using the divergence theorem.
\end{proof}

It follows from (\ref{p-1}) and (\ref{parts}) that
\begin{equation}\label{p-2}
\aligned
u_\e(x) &
=-\int_{\partial\Omega} \big( T_{ij} (y)\cdot \nabla_y \big)N_\e (x, y) \cdot 
g_{ij}(y, y/\e)\, d\sigma (y)\\
&\qquad + \int_{\partial\Omega}
N_\e (x, y) g_0 (y, y/\e)\, d\sigma (y).
\endaligned
\end{equation}
In view of (\ref{N-estimate}) this implies that
\begin{equation}\label{p-2-0}
\aligned
|u_\e (x)| &\le C \, \| g\|_\infty \int_{\partial\Omega} \frac{d\sigma (y)}{|x-y|^{d-1}}
 +C \, \| g\|_\infty \int_{\partial\Omega} \frac{d\sigma (y)}{|x-y|^{d-2}}\\
 &\le C\,  \| g\|_\infty \big\{ 1 + |\ln \delta(x) | \big\},
 \endaligned
 \end{equation}
 where $g=\{ g_{ij}, g_0\}$ and $\delta (x)=\text{dist}(x, \partial\Omega)$. 
 
 \begin{remark}\label{remark-2.10}
 {\rm
 It follows from (\ref{p-2-0}) that for any $1<q<\infty$,
 $$
 \| u_\e\|_{L^q(\Omega)} \le C _q\,  \| g\|_\infty,
 $$
 where $C_q$ depends on $q$, $A$  and $\Omega$.
 By interpolation, this, together with (\ref{rate-0}), implies that
 \begin{equation}\label{rate-q}
 \| u_\e -u_0\|_{L^q(\Omega)}
 \le C_{q, \sigma} \, \e^{\frac{1}{q}-\sigma}
 \end{equation}
 for any $2<q<\infty$ and $\sigma \in (0, \frac{1}{q})$.
 Moreover, if $A^*=A$, it follows from \cite{KLS2} that
 \begin{equation}\label{square}
 \| u_\e\|_{H^{1/2}(\Omega)} +\left(\int_\Omega |\nabla u_\e (x)|^2 \, \delta(x)\, dx\right)^{1/2}
 \le C \| g\|_{L^2(\partial\Omega)}.
 \end{equation}
 Thus, by interpolation, we may deduce from (\ref{rate-0})  and (\ref{square}) that
 \begin{equation}\label{rate-h}
 \| u_\e -u_0\|_{H^{\alpha} (\Omega)}
 \le C_{\alpha, \sigma} \e^{\frac12 -\alpha -\sigma}
 \end{equation}
 for any $\alpha \in (0, 1/2)$ and $\sigma\in (0, (1/2)-\alpha)$.
 }
 \end{remark}
 
Using (\ref{size-diff}) and (\ref{derivative-diff}), we  obtain
\begin{equation}\label{p-3}
\aligned
& u^\gamma_\e (x)=
-\int_{\partial\Omega}
\big( T_{ij}(y)\cdot \nabla_y \big)\Psi_{\e, k}^{*\alpha\beta} (y)
\cdot \frac{\partial}{\partial y_k} \big\{ N_0^{\gamma \beta} (x, y)\big\} \cdot
 g_{ij}^\alpha (y, y/\e)\, d\sigma (y)\\
&\qquad\qquad +\int_{\partial\Omega}
N^{\gamma \alpha}_0(x, y) g_0^\alpha(y, y/\e)\, d\sigma (y)
+ R_\e^\gamma(x),
\endaligned
\end{equation}
where the remainder $R_\e$ satisfies
\begin{equation}\label{R-1}
|R_\e(x)|
\le C\e^{1-\sigma} \| g\|_\infty \int_{\partial\Omega} \frac{d\sigma (y)}{|x-y|^{d-\sigma}}
+ C\e \| g\|_\infty
\int_{\partial\Omega} \frac{\ln [\e^{-1}|x-y| +2]}{|x-y|^{d-1}} d\sigma (y).
\end{equation}

\begin{lemma}\label{lemma-p2}
Let $\Omega$ be a bounded $C^{2, \alpha}$ domain for some $\alpha\in (0,1)$.
Then the function $R_\e$, given by (\ref{p-3}), satisfies 
\begin{equation}\label{p2-0}
\| R_\e\|_{L^q(\Omega)} \le C \,  \e^{\frac{1}{q}} ( 1+ |\ln \e|) \| g\|_\infty,
\end{equation}
for any $1<q<\infty$,
where $C$ depends only on $q$,  $A$ and $\Omega$.
\end{lemma}

\begin{proof}
Let $x\in \Omega$.
If $\delta (x)=\text{dist}(x, \partial\Omega)\ge \e$, we may use (\ref{R-1}) to show that
\begin{equation}\label{p2-1}
| R_\e (x) |\le C_\sigma \left(\frac{\e}{\delta (x)}\right)^{1-\sigma} \| g\|_\infty
\end{equation}
for any $\sigma \in (0,1)$.
If $\delta(x)\le \e$, the estimates in (\ref{N-estimate}), as in (\ref{p-2-0}), lead to
\begin{equation}\label{p2-2}
|R_\e (x)|
\le C\, \| g\|_\infty ( 1+|\ln \delta(x)|).
\end{equation}
It is not hard to verify  that (\ref{p2-0}) follows from (\ref{p2-1}) and (\ref{p2-2}).
\end{proof}

As $\e \to 0$,  the second term in the RHS of (\ref{p-3}) converges to
\begin{equation}\label{p-4}
w^\gamma_0(x)
=\int_{\partial\Omega}
N^{\gamma \alpha}_0(x, y) \langle g_0^\alpha\rangle (y)\, d\sigma (y),
\end{equation}
where
\begin{equation}\label{g-average}
\langle g_0 \rangle (y)  = \average_{\td} g_0(y, z)\, dz.
\end{equation}
More precisely,  the following results on the convergence rate were obtained in \cite{ASS-2015}.

\begin{lemma}\label{lemma-p3}
Let $w_\e$ denote the second term in the RHS of (\ref{p-3}).
Assume that $\Omega$ is a bounded smooth, uniformly convex domain in $\rd$.
Then, for any $1\le q <\infty$,
\begin{equation}\label{p3-0}
\| w_\e -w_0\|_{L^q(\Omega)}
\le C_q 
\left\{
\aligned
& \e^{\frac{1}{q}} & \quad  &\text{ if  d=3},\\
& \e^{\frac{3}{2q}} & \quad & \text{ if d=4}, \\
& \e^{\frac{2}{q}} (1+|\ln \e|)^{\frac{1}{q}} & \quad & \text{ if } d\ge 5,
\endaligned
\right.
\end{equation}
where $w_0$ is given by (\ref{p-4}).
\end{lemma}

Much of the rest of paper is devoted to the study of the first term in 
the RHS of (\ref{p-3}).
To this end
we first replace the  function $\Psi_{\e, k}^{*\alpha\beta}$ by
\begin{equation}\label{psi}
\psi_{\e, k}^{*\alpha\beta} (x)
=\Psi^{*\alpha\beta}_{\e, k} (x) -P_k^{\alpha\beta} (x) 
-\e \chi_{k}^{*\alpha\beta} (x/\e),
\end{equation}
where $(\chi_{k}^{*\alpha\beta} (y))$ denotes the matrix of correctors for $\mathcal{L}_\e^*$
in $\rd$. Note that
\begin{equation}\label{psi-star}
\mathcal{L}^*_\e (\psi_{\e, k}^{*\beta}) =0 \quad \text{ in } \Omega,
\end{equation}
where $\psi_{\e, k}^{*\beta}=(\psi_{\e, k}^{*1\beta}, \dots, \psi_{\e, k}^{*m\beta})$.
 
We end this section with some observations on its conormal derivatives.

\begin{lemma}\label{lemma-p5}
Let $\psi_{\e, k}^{*\alpha\beta}$ be defined by (\ref{psi}).
Then
\begin{equation}
\left(\frac{\partial}{\partial\nu_\e^*} \big\{ \psi_{\e, k}^{*\beta} \big\} \right)^\alpha (x)
=-n_i (x) b_{ik}^{*\alpha\beta} (x/\e) \quad \text{ for } x\in \partial\Omega,
\end{equation}
where 
\begin{equation}\label{b}
b_{ik}^{*\alpha\beta} (y) =a_{ik}^{*\alpha\beta} (y)+ a_{ij}^{*\alpha\gamma} (y)\frac{\partial}{\partial y_j}
\big( \chi_{k}^{*\gamma \beta}\big) -\widehat{a}_{ik}^{*\alpha\beta}
\end{equation}
and $\widehat{A^*} =(\widehat{a}_{ij}^{*\alpha\beta}) = (\widehat{A})^*$ is the homogenized 
matrix of  $A^*$.
\end{lemma}

\begin{proof}
See e.g. \cite[p.1023]{KLS2}.
\end{proof}

Note that by the definitions of correctors $\chi_k^{*\alpha\beta}$ and the homogenized matrix $\widehat{A^*}$,
\begin{equation}\label{p101}
\frac{\partial}{\partial y_i} \big\{ b_{ik}^{*\alpha\beta}\big\} =0 \quad
\text{ and } \quad \int_{\td} b_{ik}^{*\alpha\beta} =0.
\end{equation}
This implies that there are 1-periodic functions $f_{\ell i k}^{\alpha\beta}$ with mean value zero
such that
\begin{equation}\label{p102}
b_{ik}^{*\alpha\beta} =\frac{\partial}{\partial y_\ell} \big\{ f_{\ell i k}^{\alpha\beta}\big\}
\quad \text{ and } \quad 
f_{\ell i k}^{\alpha\beta} =- f_{i \ell k}^{\alpha\beta}.
\end{equation}
(see  e.g. \cite{KLS2}).
 As a result, we obtain 
 \begin{equation}\label{p103}
 \aligned
 n_i (x) b_{ik}^{* \alpha\beta} (x/\e)
 =\frac12 ( n_i e_j -n_j e_i) \cdot \nabla_x \big\{ \e f_{j i k}^{\alpha\beta} (x/\e)\big\}.
 \endaligned
 \end{equation}
 This shows that $\e^{-1}\psi^{*\beta}_{\e, k} $ is a solutions of the Neumann problem (\ref{NP-0})
 with  $g_{ij} (x, y)=(1/2)f^\beta_{j i k }(y)$ and $g_0=0$.

%%%%%%%%%%%%%%%%%%%%%%%%%%%%%%%%%%%%%%%%%%%%%%%%%%%%%

%%%%%%%%%%%%%%%%%%%%%%%%%%%%%%%%%%%%%%%%%%%%%%%%%%%%%%

\section{Neumann problems in a half-space}
\setcounter{equation}{0}

For $n\in \mathbb{S}^{d-1}$ and $a\in \mathbb{R}$, let
\begin{equation}\label{half}
\mathbb{H}^d_n (a)
=\big\{ x\in \rd: x\cdot n <-a \big\}
\end{equation}
denote a half-space with outward unit normal $n$.
Consider the Neumann problem
\begin{equation}\label{half-NP}
\left\{
\aligned
\text{\rm div} (A\nabla u) & = 0 & \quad & \text{ in } \mathbb{H}_n^d(a),\\
n \cdot A\nabla u &=T\cdot \nabla g  & \quad & \text{ on } \partial \mathbb{H}_n^d(a),
\endaligned
\right.
\end{equation}
where $T\in \mathbb{R}^d$, $|T|\le 1$ and $T\cdot n=0$.
We will assume that $g\in C^\infty(\td)$ with mean value zero and
$n$ satisfies the Diophantine condition 
\begin{equation}\label{D-condition}
| ( I-n \otimes n) \xi|\ge \kappa |\xi|^{-2} \quad \text{ for any } \xi\in \mathbb{Z}^d\setminus \{ 0\},
\end{equation}
where $\kappa>0$ and $n\otimes n=(n_i n_j)_{d\times d}$. 
We emphasize  again that all constants $C$ will be independent of $\kappa$.
Let $M$ be a $d\times d$ orthogonal matrix such that $Me_d=-n$.
Note that the last column of $M$ is $-n$. Let $N$ denote the $d\times (d-1)$
matrix of the first $d-1$ columns of $M$. Since $MM^T=I$, we see that
\begin{equation}\label{2.0}
N N^T + n \otimes n =I,
\end{equation}
where $M^T$ denotes the transpose of $M$.

Suppose that a solution of (\ref{half-NP}) is given by
\begin{equation}\label{V}
u(x)= V (x- (x\cdot n) n, -x\cdot n),
\end{equation}
where $V=V(\theta, t)$ is a function of $(\theta, t)\in \mathbb{T}^d\times [a, \infty)$.
Note that
\begin{equation}\label{2.1}
\nabla_x u = \Big( I - n \otimes n, -n \Big) \left(\begin{array}{c} \nabla_\theta V\\ \partial_t V \end{array}\right)
=M \left( \begin{array} {c}N^T \nabla_\theta \\ \partial_t  \end{array} \right)V,
\end{equation}
where we have used (\ref{2.0}).
It follows from (\ref{half-NP}) and (\ref{2.1}) that $V$ is a solution of
\begin{equation}\label{half-NP-V}
\left\{
\aligned
\left( \begin{array} {c}N^T \nabla_\theta \\ \partial_t  \end{array}\right)
\cdot B \left( \begin{array} {c}N^T \nabla_\theta \\ \partial_t  \end{array} \right)V
& =0  &\quad & \text{ in } \mathbb{T}^d \times (a, \infty),\\
-e_{d+1} \cdot B \left( \begin{array} {c}N^T \nabla_\theta \\ \partial_t  \end{array} \right)V
& =T\cdot \nabla_\theta \widetilde{g} & \quad & \text{ on } \mathbb{T}^d \times \{ a \},
\endaligned
\right.
\end{equation}
where 
\begin{equation}\label{B}
B=B(\theta, t)=M^T A(\theta-tn) M,
\end{equation}
$\widetilde{g}(\theta,t) =g(\theta -tn)$, and we have used the assumption  that $T\cdot n=0$ to obtain 
$T\cdot \nabla_x g =T\cdot \nabla_\theta \widetilde{g}$.
Observe that if $V^0$ is a solution of (\ref{half-NP-V}) with $a=0$ and
$$
V^a (\theta, t)= V^0(\theta-an, t-a) \quad \text{ for } a\in \mathbb{R},
$$
then $V^a$ is a solution of (\ref{half-NP-V}).
This follows from the fact  that 
$$
B(\theta-an, t-a)=B(\theta, t) \quad \text{ and } \quad \widetilde{g}(\theta-an, t-a)=\widetilde{g}(\theta, t).
$$
As a result, it suffices to study the boundary value problem (\ref{half-NP-V}) for $a=0$.
To this end, we shall consider the Neumann problem 
\begin{equation}\label{half-NP-V-G}
\left\{
\aligned
-\left( \begin{array} {c}N^T \nabla_\theta \\ \partial_t  \end{array}\right)
\cdot B \left( \begin{array} {c}N^T \nabla_\theta \\ \partial_t  \end{array} \right)V -\lambda  \Delta_\theta V
& = \left( \begin{array} {c}N^T \nabla_\theta \\ \partial_t  \end{array} \right)G
 &\quad & \text{ in } \mathbb{T}^d \times \mathbb{R}_+,\\
-e_{d+1} \cdot B \left( \begin{array} {c}N^T \nabla_\theta \\ \partial_t  \end{array} \right)V
& =T\cdot \nabla_\theta g +e_{d+1} \cdot G
& \quad & \text{ on } \mathbb{T}^d \times \{ 0 \},
\endaligned
\right.
\end{equation}
where $\lambda>0$ 
 and the term $-\lambda \Delta_\theta V $ is added to regularize the system.

Let 
\begin{equation}\label{H}
\mathcal{H}=\left\{ f\in H^1_{\loc} (\td\times \mathbb{R}_+): 
\int_0^\infty \int_{\td} \left( | \nabla_\theta f|^2 +|\partial_t f|^2 \right)<\infty \right\}.
\end{equation}
We call $V\in \mathcal{H}$ a weak solution of (\ref{half-NP-V-G}) 
with $g\in H^1(\td)$ and $G \in L^2(\td \times \mathbb{R}_+)$, if
\begin{equation}\label{weak}
\aligned
 &\int_0^\infty \int_{\td} \left\{ 
B\left(\begin{array}{c} N^T\nabla_\theta \\ \partial_t  \end{array}\right) V 
\cdot
\left(\begin{array}{c}  N^T\nabla_\theta \\ \partial_t  \end{array}\right)W
+\lambda 
\left(\begin{array}{c} \nabla_\theta \\ 0  \end{array}\right) V 
\cdot
\left(\begin{array}{c}  \nabla_\theta \\ 0  \end{array}\right)W\right\} \, d\theta dt \\
& =-\int_{\td} (T\cdot \nabla_\theta g) \cdot W(\theta, 0)\, d\theta
-\int_0^\infty \int_{\td} G \cdot \left( \begin{array} {c}N^T \nabla_\theta \\ \partial_t  \end{array}\right) W\, d\theta dt
\endaligned
\end{equation}
for any $W\in \mathcal{H}$.

\begin{prop}\label{prop-2.1}
Let $g\in H^1(\td)$ and $G\in L^2(\td\times \mathbb{R}_+)$. 
Then the boundary value problem (\ref{half-NP-V-G})
has a solution, unique up to a constant, in $\mathcal{H}$.
Moreover, the solution $V$ satisfies
\begin{equation}\label{2.1-0}
\int_0^\infty \int_{\td} \left( |N^T \nabla_\theta V|^2 +|\partial_t V|^2 \right)
\le C \Big\{ \|g\|^2_{H^1(\td)} +\| G\|^2_{L^2(\td\times \mathbb{R}_+)} \Big\},
\end{equation}
and
\begin{equation}\label{2.1-000}
\lambda \int_0^\infty \int_{\td}
|\nabla_{\theta} V |^2 \le C \Big\{ \|g\|^2_{H^1(\td)} +\| G\|^2_{L^2(\td\times \mathbb{R}_+)} \Big\},
\end{equation}
where $C$ depends only on $d$, $m$ and $\mu$.
\end{prop}

\begin{proof}
This follows readily from the Lax-Milgram theorem.
One only needs to observe that
\begin{equation}\label{2.1-00}
\Big|\int_{\td} (T\cdot \nabla_\theta g) \cdot W(\theta, 0)\, d\theta \Big|
\le C \| g\|_{H^1(\td)}
\left(\int_0^1 \int_{\td} \left( |N^T \nabla_\theta W|^2 +|\partial_t W|^2 \right)\right)^{1/2}
\end{equation}
for any $W\in \mathcal{H}$.
Indeed, write
\begin{equation}\label{2.1-1}
\aligned
& \int_{\td} (T\cdot \nabla_\theta g) \cdot W(\theta, 0)\, d\theta= \\
& \int_0^1\int_{\td} (T\cdot \nabla_\theta g)  \cdot ( W(\theta, 0) -W(\theta, t) )\, d\theta dt
+\int_0^1\int_{\td} (T\cdot \nabla_\theta g) \cdot W(\theta, t)\, d\theta dt.
\endaligned
\end{equation}
It is easy to see that the first term in the RHS of (\ref{2.1-1}) is bounded by
$$
 C\| \nabla_\theta g \|_{L^2(\td)} \left(\int_0^1 \|\partial_t W\|_{L^2(\td)}^2\, dt\right)^{1/2}.
 $$
To handle the second term in the RHS of (\ref{2.1-1}), we use
$$
\int_0^1\int_{\td} (T\cdot \nabla_\theta g) \cdot W(\theta, t)\, d\theta dt
=-\int_0^1 g \cdot (T\cdot \nabla_\theta W) (\theta, t)\, d\theta dt
$$
and 
$$
T\cdot \nabla_{\theta} W =NN^T T \cdot \nabla_{\theta} W = T\cdot NN^T \nabla _\theta W
$$
to bound it by
$$
 C \| g\|_{L^2(\td)} \left(\int_0^1 \| N^T\nabla_\theta W\|_{L^2(\td)}^2 \, dt \right)^{1/2}.
 $$
 The estimate (\ref{2.1-00}) now follows.
\end{proof}

\begin{prop}\label{prop-2.2}
Let $g\in H^k(\td)$ and $G\in L^2(\mathbb{R}_+, H^{k-1}(\td))$
 for some $k\ge 1$. Then the solution of (\ref{half-NP-V-G}), given by Proposition 
\ref{prop-2.1}, satsfies
\begin{equation}\label{2.2-0}
\aligned
 & \int_0^\infty \left( \|  N^T \nabla_\theta V \|^2_{H^{k-1} (\td)}
 + \| \partial_t V\|_{H^{k-1}(\td)}^2  +\lambda \| V\|^2_{H^k(\td)} \right)\, dt\\
 &
\qquad\qquad
\le C_k \left\{
  \|g\|^2_{H^k(\td)}
 + \int_0^\infty \| G\|^2_{H^{k-1}(\td)} \right\}\, dt,
 \endaligned
\end{equation}
where $C_k$ depends on $d$, $m$, $k$, $\mu$ and $\| A\|_{C^{k-1}(\td)}$.
\end{prop}

\begin{proof}
The proof is standard.
The case $k=1$ is given in Proposition \ref{prop-2.1}.
To prove the estimate for $k=2$, one applies the estimate for $k=1$ to the quotient of difference 
$  \left\{  V(\theta+  s e_j , t) -V (\theta, t)\right\}  s^{-1}$ and lets $s\to 0$. 
The general case follows similarly by an induction argument on $k$.
\end{proof}

\begin{prop}\label{prop-2.3}
Let $g\in H^{k+\ell-1}(\td)$ for some $k, \ell\ge 1$.
Suppose that 
$$
\partial_t^\alpha G \in L^2(\mathbb{R}_+, H^{k+\ell-2-\alpha}(\td))
\quad \text{ for } \ 0\le \alpha \le \ell-1.
$$
Then the solution of (\ref{half-NP-V-G}), given by Proposition 
\ref{prop-2.1}, satsfies
\begin{equation}\label{2.3-0}
  \int_0^\infty  \|  \partial_t^\ell  V \|^2_{H^{k-1} (\td)} \, dt
\le C \left\{
  \|g\|^2_{H^{k+\ell-1} (\td)}
 + \sum_{0\le \alpha\le \ell-1}
 \int_0^\infty \| \partial_t^\alpha  G\|^2_{H^{k+\ell -2-\alpha}(\td)} \right\}\, dt,
\end{equation}
where $C$ depends on $d$, $m$, $k$, $\ell$, $\mu$ and $\| A\|_{C^{k+\ell-2}(\td)}$.
\end{prop}

\begin{proof}
The case $\ell=1$ is contained in Proposition \ref{prop-2.2}.
To see the case $\ell=2$, we observe that the second-order equation in (\ref{half-NP-V-G})
allows us to obtain 
\begin{equation}\label{2.3-1}
\aligned
\partial_t^2 V
 &= \text{\rm a linear combination of }\\
& \qquad
\nabla_\theta (N^T\nabla_\theta)V, N^T\nabla_\theta V,
 \partial_t \nabla_\theta V, \partial_tV, \nabla_\theta G, \lambda\Delta_\theta V,
\partial_t G
\endaligned
\end{equation}
with smooth coefficients. It follows that
$$
\aligned
 \|\partial_t^2 V\|_{H^{k-1}(\td)}
\le C  & \Big\{ \| N^T \nabla_\theta V \|_{H^k(\td)}
+\|\partial_t V \|_{H^k(\td)}
+ \|  G \|_{H^{k}(\td)}\\
&\qquad\qquad
+\| \partial_t G \|_{H^{k-1}(\td)} 
+\lambda \| V\|_{H^{k+1}(\td)}
\Big\}.
\endaligned
$$
This, together with the estimate (\ref{2.2-0}), gives (\ref{2.3-0}) for $\ell=2$.
The general case follows by differentiating (\ref{2.3-1}) in $t$ and using an induction argument on $\ell$.
\end{proof}

\begin{prop}\label{prop-2.4}
Suppose that $n$ satisfies the Diophantine condition (\ref{D-condition}) with constant $\kappa>0$.
Let $V$ be the solution of (\ref{half-NP-V-G}), given by Proposition \ref{prop-2.1}. 
Let 
$$
\widetilde{V}(\theta, t) = V(\theta, t) -\average_{\td} V(\cdot, t).
$$
Then
\begin{equation}\label{2.4-0}
\int_0^\infty \kappa^2 \| \widetilde{V} \|_{H^k(\td)}^2 dt \le C \Big\{
\| g\|_{H^{k+3}(\td)}^2 +\int_0^\infty \| G\|_{H^{k+2}(\td)}^2 \, dt \Big\},
\end{equation}
where $C$ depends on $d$ and $k$ .
\end{prop}

\begin{proof}
It follows from (\ref{D-condition}) and (\ref{2.0}) that
$|N^T \xi|\ge \kappa |\xi|^{-2}$ for any $\xi\in \mathbb{Z}^d\setminus \{ 0\}$.
This implies that
$$
\| N^T \nabla_\theta V \|_{H^{k+2}(\td)} \ge C \kappa \| \widetilde{V}\|_{H^k(\td)},
$$
which, together with (\ref{2.2-0}), gives the estimate (\ref{2.4-0}).
\end{proof}

\begin{remark}\label{remark-2.1}
{\rm
Suppose that $g\in C^\infty(\td)$,
$G\in C^\infty(\td\times \mathbb{R}_+)$ and $\partial_t^k \partial_\theta^\alpha  G 
\in L^2(\td\times \mathbb{R}_+)$ for any $k$ and $\alpha$.
For $\lambda>0$, let
$V_\lambda$ be the solution of (\ref{half-NP-V-G}), given by Proposition \ref{prop-2.1}.
By subtracting a constant we may assume that $\int_{\td} V_\lambda (\theta, 0)d \theta=0$ and
thus
$$
V_\lambda(\theta, t) =\widetilde{V_\lambda} (\theta, t) +\int^t_0 \int_{\td} \partial_s V_\lambda(\theta, s)  d\theta ds.
$$
It follows from Propositions \ref{prop-2.3} and \ref{prop-2.4} that the $L^2(\td \times (0, L))$ norm
of
$\partial_t^k \partial^\alpha_\theta {V_\lambda}$
is uniformly bounded in $\lambda$, for any $k$, $\alpha$ and $L\ge 1$.
  Hence, by Sobolev imbedding, the $C^k(\td\times (0, L))$ norm
  of ${V_\lambda}$ is uniformly bounded in $\lambda$, for any $k\ge 0$ and $L\ge 1$.
By a simple limiting argument this allows us to show that the Neumann problem (\ref{half-NP-V-G}) 
with $\lambda=0$ has a solution $V$, unique up to a constant,
in $C^\infty(\td\times [0, \infty))$.
Furthermore, by passing to the limit, estimates (\ref{2.1-0}), (\ref{2.2-0}),
(\ref{2.3-0}) and (\ref{2.4-0}) continue to hold for this solution.
}
\end{remark}

\begin{prop}\label{prop-2.5}
Suppose that $n$ satisfies the Diophantine condition (\ref{D-condition}) with constant $\kappa>0$.
Let $V$ be the solution of (\ref{half-NP-V-G})
with $\lambda=0$, $g\in C^\infty(\td)$ and $G=0$, given by Remark \ref{remark-2.1}. Then
there exists a constant $V_\infty$ such that for any $\ell\ge 1$,
\begin{equation}\label{2.5-0}
| \partial_\theta^\alpha (V-V _\infty) (\theta, t) |\le \frac{C_{\alpha, \ell}}{ \kappa (1+\kappa t)^\ell},
\end{equation}
for any $\alpha=(\alpha_1, \dots, \alpha_d)$.
Moreover, we have
\begin{equation}\label{2.5-1}
|N^T \nabla_\theta (\partial_\theta^\alpha V) (\theta, t)|
+|\partial_t^k \partial_\theta^\alpha V(\theta, t)|  \le \frac{C_{\alpha, \ell, k}}{(1+\kappa t)^\ell},
\end{equation}
where $k\ge 1$.
\end{prop}

\begin{proof}
It follows from Propositions \ref{prop-2.2} and \ref{prop-2.3} by Sobolev imbedding that
$$
|N^T \nabla_\theta (\partial_\theta^\alpha V) (\theta, t)|
+|\partial_t^k \partial_\theta^\alpha V(\theta, t)|  \le C_{\alpha, k}
$$
for any $\alpha=(\alpha_1, \dots, \alpha_d)$ and $k\ge 1$.
Next we note that the decay estimate in (\ref{2.5-1})
follows by the exact argument as in the case of Dirichlet boundary conditions, given in \cite{Masmoudi-2012}.
Indeed, let
$$
F(s)= \int_{s}^\infty \int_{\td} \Big( |N^T \nabla_\theta (\partial_\theta^\alpha V)|^2
+ |\partial_t^k \partial_\theta^\alpha V|^2\Big)\, d\theta dt.
$$
An inspection of the proof of Proposition 2.6 in \cite{Masmoudi-2012} shows that
$$
F(s)\le \frac{C_\ell}{(\kappa s)^\ell} \quad \text{ for any } \ell\ge 1
$$
(the proof does not use the boundary condition at $t=0$).
By Sobolev imbedding this gives
$$
|N^T \nabla_\theta (\partial_\theta^\alpha V) (\theta, t)|
+|\partial_t^k \partial_\theta^\alpha V(\theta, t)|  \le \frac{C_{\alpha, \ell, k}}{(\kappa t)^\ell}.
$$
Finally, we note that a solution of (\ref{half-NP-V-G}) is also a solution of the same system with Dirichlet condition
$V(\theta, 0)\in C^\infty(\td)$. It follows from \cite{Masmoudi-2011, Masmoudi-2012} (also see \cite{Prange-2013})
that $V(\theta,t)$ has a limit $V_\infty$, as $t\to \infty$.
Moreover,
\begin{equation}\label{estimate-V}
\aligned
|\partial_\theta^\alpha (V-V_\infty) (\theta, t)|
& \le \int_t^\infty |\partial_t \partial_\theta^\alpha V(\theta, s)|\, ds
 \le C \int_t^\infty \frac{ds}{(1+\kappa s)^{\ell+2}}\\
&\le \frac{C}{(1+\kappa t)^\ell} \int_t^\infty \frac{ds}{(1+\kappa s)^2}\\
&\le \frac{C}{\kappa (1+\kappa t)^\ell},
\endaligned
\end{equation}
where we have used (\ref{2.5-1}) for the second inequality.
\end{proof}

We now state and prove the main result  of this section.

\begin{theorem}\label{theorem-2.1}
Let $n\in \mathbb{S}^{d-1}$  and $a\in \mathbb{R}$, where $d\ge 2$.
Let $T\in \rd$ such that $|T| \le 1$ and $T\cdot n=0$.
Suppose that $n\in \mathbb{S}^{d-1}$ satisfies the Diophantine condition (\ref{D-condition})
with constant $\kappa>0$.
Then for any $g\in C^\infty(\td)$, the Neumann problem (\ref{half-NP}) has 
a smooth solution $u$ satisfying 
\begin{equation}\label{2.6-0}
\aligned
| u(x)| & \le \frac{C}{\kappa ( 1+ \kappa |x\cdot n +a|)^\ell},\\
|\partial_x^\alpha u(x)| & \le \frac{C}{( 1+ \kappa |x\cdot n +a|)^\ell},
\endaligned
\end{equation}
for any  $|\alpha|\ge 1$ and $\ell\ge 1$. The constant 
$C$ depends at most on $d$, $m$, $\mu$, $\alpha$, $\ell$  as well as the $C^k(\td)$ norms of $A$ and $g$
for some $k=k(d, \alpha, \ell)$.
\end{theorem}

\begin{proof}
Let $V$ be the solution of (\ref{half-NP-V-G})
with $\lambda=0$, $g\in C^\infty(\td)$ and $G=0$, given by Remark \ref{remark-2.1}. 
Let
$$
u(x)=V(x-(x\cdot n +a)n, -(x\cdot n +a)) -V_\infty.
$$
Then $u$ is a solution of the Neumann problem (\ref{half-NP}).
The first inequality in (\ref{2.6-0}) follows directly from (\ref{2.5-0}).
To see the second inequality, one uses (\ref{2.1}) and (\ref{2.5-1}).
\end{proof}

%%%%%%%%%%%%%%%%%%%%%%%%%%%%%%%%%%%%%%%%%%%%%

%%%%%%%%%%%%%%%%%%%%%%%%%%%%%%%%%%%%%%%%%%%%%

\section{Some refined estimates in a half-space}
\setcounter{equation}{0}

Throughout this section we fix $n\in \mathbb{S}^{d-1}$  and $a\in \mathbb{R}$.
We assume that $n\in \mathbb{S}^{d-1}$ satisfies the Diophantine condition (\ref{D-condition})
with constant $\kappa>0$.
However, we will be only interested in estimates that are independent of $\kappa$.

Our first result plays the same role as the maximum principle in the case of Dirichlet 
problem.

\begin{theorem}\label{theorem-r}
Let $T\in \rd$ such that $|T| \le 1$ and $T\cdot n=0$.
Then for any $g\in C^\infty(\td)$, the solution $u$ of Neumann problem (\ref{half-NP}), given by 
Theorem \ref{theorem-2.1}, satisfies
\begin{equation}\label{r-1-0}
| \nabla u(x)|\le \frac{C\,  \| g\|_\infty }{|x\cdot n +a|},
\end{equation}
for any $x\in \mathbb{H}_n^d(a)$,
where $C$ depends only on $d$, $m$ and $\mu$ as well as some H\"older norm of $A$.
\end{theorem}

\begin{proof}
By translation we may assume that $a=0$.
We choose a bounded smooth domain $D$ such that 
$$
\aligned
& B(0, 1)\cap \mathbb{H}_n^d(0)\subset  D \subset B(0, 2) \cap \mathbb{H}_n^d (0),\\
&\overline{B(0,1)} \cap\partial \mathbb{H}_n^d(0) = \partial D\cap \partial\mathbb{H}_n^d(0).
\endaligned
$$
Let $v_\e (x)=\e u (x/\e)$. 
Since $\mathcal{L}_\e (v_\e)=0$ in $D$,
$$
v_\e (x) -v_\e (z) 
=\int_{\partial D} \big\{ N_\e (x, y) -N_\e (z, y) \big\} \frac{\partial v_\e}{\partial \nu_\e} (y)\, d\sigma (y)
$$
for any $x, z\in D$,
where  $N_\e(x, y)$ denotes the matrix of Neumann functions for $\mathcal{L}_\e$ in $D$.
By a change of variables it follows that
$$
u(x)-u(z)=
\e^{d-2}\int_{\partial D_{1/\e}}
\big\{ N_\e (\e x, \e y)-N_\e (\e z, \e y) \big\} n(\e y) \cdot A(y)\nabla u(y)\, d\sigma (y),
$$
where $D_{1/\e} =\big\{ \e^{-1} y: y\in D\big\}$.

Fix $x, z\in \mathbb{H}_n^d(0)$ such that $|x-z|< (1/2)|x\cdot n|=(1/2)\text{dist}(x, \partial \mathbb{H}_n^d(0))$.
Choose $\eta_\e \in C_0^1(B(0, \e^{-1}))$ such that $0\le \eta_\e \le 1$,
$\eta_\e=1$ on $B(0, \e^{-1}-1)$ and $|\nabla \eta_\e|\le 1$, where $\e< 1/10$.
Let $u(x)-u(z) =I_1 +I_2$,
where
$$
\aligned
I_1 &=\e^{d-2}\int_{\partial D_{1/\e}} \eta_\e (y)
\big\{ N_\e (\e x, \e y)-N_\e (\e z, \e y) \big\} n (\e y) \cdot A(y)\nabla u(y)\, d\sigma (y)\\
&=\e^{d-2}\int_{\partial D_{1/\e}} \eta_\e (y) 
\big\{ N_\e (\e x, \e y)-N_\e (\e z, \e y) \big\}   T\cdot \nabla g (y)\, d\sigma (y)\\
&=-\e^{d-2}\int_{\partial B(0, \e^{-1})\cap \partial \mathbb{H}_n^d (0)}
T\cdot \nabla_y \Big\{ \eta_\e (y) (N_\e (\e x, \e y)-N_\e (\e z, \e y)) \Big\} g(y) \, d\sigma (y),
\endaligned
$$
where we have used the Neumann condition for $u$ as well as an integration by parts on 
the boundary. We now apply the estimates in (\ref{N-estimate-1}).
This gives
$$
\aligned
|I_1| &\le C |x-z| \| g\|_\infty\int_{\partial \mathbb{H}_n^d(0)} \frac{d\sigma (y)}{|x-y|^d}
+ C |x-z|\| g\|_\infty \int_{\frac{1}{\e}-1 \le |y|\le \frac{1}{\e}}
\frac{ d\sigma (y)}{|x-y|^{d-1}}\\
& \le C_0 \|g\|_\infty + C \e\| g\|_\infty |x-z|,
\endaligned
$$
if $\e$, which may depend on $|x|$, is sufficiently small.
 We point out that the constant $C_0$ in the estimate above
  depends only on $d$, $m$, $\mu$ and some H\"older norm of $A$.

Next,  to handle $I_2$, we use the estimate 
$$
|\nabla u(y)|\le \frac{C}{(1+\kappa |y\cdot n|)^2}
$$
from (\ref{2.6-0}). This, together with (\ref{N-estimate-1}), leads to 
$$
\aligned
| I_2|    &=\e^{d-2}\Big| \int_{\partial D_{1/\e}} (1-\eta_\e (y))
\big\{ N_\e (\e x, \e y)-N_\e (\e z, \e y) \big\} n(\e y) \cdot A(y)\nabla u(y)\, d\sigma (y)\Big|\\
&  \le C |x-z| \int_{\partial D_{1/\e}\cap \mathbb{H}_n^d(0)}
\frac{d\sigma (y)}{|x-y|^{d-1} (1+\kappa |y\cdot n|)^2}\\
&\le C_{x, z} \int_{\partial D_{1/\e}\cap \mathbb{H}_n^d(0)}
\frac{d\sigma (y)}{|y|^{d-1} (1+\kappa |y\cdot n|)^2},
\endaligned
$$
which shows that $I_2\to 0$, as $\e\to 0$.
As a result, we have proved that for any $x, z\in \mathbb{H}_n^d(0)$ with
$|x-z|\le \text{dist}(x, \partial\mathbb{H}_n^d(0))$,
$$
|u(x)-u(z)|
=\lim_{\e\to 0} | I_1 +I_2|
\le C_0 \| g\|_\infty.
$$
Since $\mathcal{L}_1 (u)=0$ in $\mathbb{H}_n^d(0)$,
by the interior Lipschitz estimates \cite{AL-1987} for $\mathcal{L}_1$,
we obtain
$$
|\nabla u(x)|\le \frac{C_0 \| g\|_\infty}{|x\cdot n|},
$$
which completes the proof.
\end{proof}

Let $\Omega=\mathbb{H}_n^d(a)$ and $\mathcal{L} =-\text{\rm div} (A(x)\nabla )$ .
In the rest of this section we consider Dirichlet problem,
\begin{equation}\label{DP-w}
\left\{
\aligned
\mathcal{L} (u) &=\text{div} (f) +h &\quad &\text{ in } \Omega,\\
u & =0 &\quad & \text{ on }\partial\Omega,
\endaligned
\right.
\end{equation}
and the Neumann problem,
\begin{equation}\label{NP-w}
\left\{
\aligned
\mathcal{L} (u) &=\text{div} (f) &\quad &\text{ in } \Omega,\\
\frac{\partial u}{\partial \nu} & =-n\cdot f &\quad & \text{ on }\partial\Omega,
\endaligned
\right.
\end{equation}
where $A$ is assumed to satisfy the ellipticity condition (\ref{ellipticity}) and
$A\in C^\sigma (\td)$ for some $\sigma\in (0,1)$.
We shall be interested in the weighted $L^2$ estimate,
\begin{equation}\label{w-estimate-1}
\int_\Omega |\nabla u(x)|^2 \big[\delta(x)\big]^\alpha\, dx
\le C \int_\Omega | f(x) |^2 \big[\delta(x)\big]^\alpha\, dx
+C \int_\Omega |h(x)|^2 \big[\delta(x)\big]^{\alpha+2}\, dx,
\end{equation}
where $-1<\alpha<0$ and
\begin{equation}\label{weight}
\delta(x)=\text{dist} (x, \partial\Omega)=|a + (x\cdot n)|.
\end{equation}

We start with some observations  on the weight $\omega (x) =\big[\delta(x)\big]^\alpha$.

\begin{lemma}\label{lemma-weight}
Let $\omega (x) =\big[\delta(x)\big]^\alpha$, where $-1<\alpha<0$ and $\delta(x)$ is defined by (\ref{weight}).
Then $\omega(x)$ is an $A_1$ weight, i.e., for any ball $B\subset \rd$,
\begin{equation}\label{A-1}
\average_B \omega  \le C \inf_{B} \omega ,
\end{equation}
where $C$ depends only on $d$ and $\alpha$. Moreover, $w$ satisfies the reverse H\"older's inequality, 
\begin{equation}\label{reverse-H}
\left(\average_B \omega ^p\, dx \right)^{1/p}
\le C \average_B \omega\, dx,
\end{equation}
where $1<p< \infty$  and $\alpha p>-1$.
\end{lemma}

\begin{proof}
This is more or less well known and may be verified directly by reducing
the problem to the case $\Omega=\rd_+$ and $\delta (x)=|x_d|$.
\end{proof}

It follows from (\ref{reverse-H}) by H\"older's inequality that if $E\subset B$, then
\begin{equation}\label{measure}
\frac{\omega(E)}{\omega(B)}
\le C \left( \frac{|E|}{|B|}\right)^{1-\frac{1}{p}},
\end{equation}
where $\omega(E)=\int_E \omega\, dx$. 
Since (\ref{measure}) implies that $\omega$ satisfies the doubling condition,
$\omega(2B)\le C \omega(B)$,
it is easy to see that (\ref{measure})
also holds if one replaces ball $B$ by cube $Q$.
In fact, if $\omega$ is an $A_p$ weight in $\rd$,
i.e.,
\begin{equation}\label{A-p}
\average_B \omega \cdot \left(\average_B \omega^{-\frac{1}{p-1}}\right)^{p-1}
\le C,
\end{equation}
 then there exist some $\sigma>0$  and $C>0$
such that
\begin{equation}\label{A-infty}
\frac{\omega(E)}{\omega(Q)}
\le C \left( \frac{|E|}{|Q|}\right)^{\sigma} \quad \text{ for any } E\subset Q.
\end{equation}
Functions satisfying (\ref{A-infty}) are called $A_\infty$ weights.
In the following we will also need the well known fact that if $\omega$ is an $A_p$ weight for some $1<p<\infty$, then
\begin{equation}\label{weighted-max}
\int_{\rd} |\mathcal{M}(f)|^p\, \omega\, dx \le C \int_{\rd} | f|^p\, \omega \, dx,
\end{equation}
where $\mathcal{M}(f)$ denotes the Hardy-Littlewood maximal function of $f$.
This is the defining property of the $A_p$ weights.
Note that if $\omega$ is $A_1$, then $\omega$ is $A_p$ for any $p>1$.
We refer the reader to \cite{Stein-1993} for the theory of weights in harmonic analysis.

The next lemma gives the Lipschitz estimates for $\mathcal{L}$ in $\Omega$ with 
either Dirichet or Neumann boundary conditions.

\begin{lemma}\label{lemma-w-1}
Let $x_0\in \overline{\Omega}$ and $r>0$.
Let $u\in H^1(B(x_0, 2r)\cap\Omega)$ be a weak solution of
$\mathcal{L} (u)=0$ in  $B(x_0, 2r)\cap\Omega$
with either $u=0$ or 
$\frac{\partial u}{\partial \nu}=0$ on $B(x_0, 2r)\cap \partial\Omega$.
Then
\begin{equation}\label{w-1-0}
\| \nabla u\|_{L^\infty(B(x_0, r)\cap \Omega)}
\le C \average_{B(x_0, 2r)\cap\Omega} |\nabla u|,
\end{equation}
where  $C$ depends only on $d$, $m$, $\mu$ and $\| A\|_{C^\sigma(\td)}$.
\end{lemma}

\begin{proof}
By translation and dilation we may reduce (\ref{w-1-0}) to the uniform Lipschitz estimate
 for operator $\mathcal{L}_\e$ with $\e=r^{-1}$
in $B(0, 2)\cap\Omega$.
In the case with Dirichlet 
condition $u=0$ on $B(0, 2)\cap\partial\Omega$, the Lipschitz estimate
was proved in \cite{AL-1987}.
The Lipschitz estimate for solutions  with Neumann condition $\frac{\partial u}{\partial \nu}=0$
on $B(0, 2)\cap\partial\Omega$ was obtained in \cite{KLS1, AS}.
The fact that 
$\Omega$ has a flat boundary is crucial here. 
Otherwise, the constant $C$ will depend on $r$, if $r$ is large.
This is true even for harmonic functions in smooth domains.
\end{proof}

\begin{theorem}\label{theorem-w-2}
Let $\omega$ be an $A_1$ weight in $\rd$.
Let $u\in H^1_{\loc}(\Omega)$ be a weak solution of Dirichlet problem(\ref{DP-w}) with $h=0$.
Assume that 
\begin{equation}\label{w-2-00}
\omega(B(x_0, R)\cap \Omega)
\average_{B(x_0, R)\cap\Omega} |\nabla u|^2 \to 0 \quad \text{ as } R\to \infty,
\end{equation}
for some $x_0\in \partial\Omega$.
Then
\begin{equation}\label{w-2-0}
\int_\Omega |\nabla u|^2 \, \omega \, dx
\le C \int_\Omega |f|^2  \, \omega \, dx,
\end{equation}
where $C$ depends only on $d$, $m$, $\mu$, $\|A\|_{C^\sigma(\td)}$, and the constant in (\ref{A-1}).
\end{theorem}

\begin{proof}
This is essentially proved  in \cite{Shen-2005} by a real-variable method, originated in \cite{CP-1998}.
We provide a proof here for the reader's convenience.
By translation we may assume that $a=0$.
We will also assume that $n=-e_d$ and thus $\Omega=\rd_+$ for simplicity of exposition.
We point out that the periodicity of the coefficient matrix
 is not used particularly  in the proof, only the estimates in Lemma \ref{lemma-w-1}.

Fix $1<p<2$.
Let $\rho\in (0,1)$ be a small constant to be determined and $A=\rho^{-\sigma/2}$, where $\sigma$ is given in
(\ref{A-infty}). Let
$$
\Omega_{R}= (-R, R)\times \cdots \times (-R, R)\times (0, 2R)\subset \rd_+.
$$
W fix $R>1$ and consider the set
\begin{equation}\label{E-lambda}
E(\lambda )= \big\{ x\in \Omega_R: \mathcal{M}_{{R}} (|\nabla u|^p)(x) > \lambda \big\},
\end{equation}
where $\mathcal{M}_R(F)$ is a localized  Hardy-Littlewood maximal function of $F$, defined by
$$
\mathcal{M}_{{R} }(F) (x) =\sup _{x\in Q\subset \Omega_{2R}} \average_Q |F|.
$$
Let 
\begin{equation}\label{lambda-0}
\lambda_0 =\frac{C_0}{|\Omega_{2R}|}\int _{\Omega_{2R}}
 |\nabla u|^p,
\end{equation}
where $C_0$ is a large constant depending on $d$.
For each $\lambda>\lambda_0$,
we perform a Calder\'on-Zygmund decomposition to $E(A\lambda)\subset \Omega_R$.
This produces a sequence of disjoint dyadic subcubes 
$\{ Q_k\}$ of $\Omega_R$ such that
\begin{equation}\label{w-10}
\aligned
& |E(A\lambda)\setminus \cup_k Q_k|=0,\\
& |E(A\lambda)\cap Q_k|>\rho |Q_k|,\ \ \
 |E(A\lambda) \cap \overline{Q}_k|\le \rho |\overline{Q}_k|,
\endaligned
\end{equation}
where $\overline{Q}_k$ denotes the dyadic parent of $Q_k$, i.e., $Q_k$ is obtained by
bisecting $\overline{Q}_k$ once.
We claim  that it is possible to choose $\rho, \gamma\in (0,1)$ so that
\begin{equation}\label{w-11}
\text{ if } \ 
 \big\{ x\in \overline{Q}_k: \mathcal{M}_{R} (|f|^p)(x)\le \gamma \lambda \big\} \neq \emptyset,
\ \text{ then } \ \overline{Q}_k\subset E(\lambda).
\end{equation}

The claim is proved by contraction.
Suppose that there exists some $x_0$ such that
$x_0\in \overline{Q}_k\setminus E(\lambda)$.
Then, if a cube $Q$ contains $\overline{Q}_k$ and $Q\subset \Omega_{2R}$, 
\begin{equation}\label{w-12}
\average_Q |f|^p \le \gamma \lambda \quad \text{ and } \quad \average_{Q} |\nabla u|^p \le \lambda.
\end{equation}
This implies that for any $x\in Q_k$,
\begin{equation}\label{w-13}
\mathcal{M}_{R} (|\nabla u|^p) (x)
\le \max ( \mathcal{M}_{2\overline{Q}_k}( |\nabla u|^p) (x), 5^d \lambda),
\end{equation}
where
$$
\mathcal{M}_{2\overline{Q}_k} (|F|)(x) =\sup_{x\in Q \subset 2\overline{Q}_k\cap\Omega}\average_Q  |F|.
$$
We now write $u=v+w$, where $v$ is a function such that
\begin{equation}\label{w-14}
\mathcal{L} (v)=\text{\rm div} (f)\quad \text{ in } \Omega\cap 5\overline{Q}_k
\quad \text{ and } \quad v=0 \quad  \text{ on } \partial\Omega\cap 5\overline{Q}_k ,
\end{equation}
\begin{equation}\label{w-15}
\int_{\Omega\cap 4\overline{Q}_k} |\nabla v|^p \le C \int_{\Omega \cap 5 \overline{Q}_k} |f|^p,
\end{equation}
and $C$ depends only on $d$, $m$, $p$, $\mu$ and $\|A\|_{C^\sigma(\td)}$.
The existence of such $v$ follows from the boundary $W^{1, p}$ estimates for $\mathcal{L}_\e$ \cite{AL-1987, AL-1991}.
 Note that if $A>5^d$,
\begin{equation}\label{w-16}
\aligned
|Q_k\cap E(A\lambda)|
&\le  |\big\{ x\in Q_k: \mathcal{M}_{2\overline{Q}_k} (|\nabla u|^p)(x)> A\lambda \big\}|\\
&\le |\big\{ x\in Q_k: \mathcal{M}_{2\overline{Q}_k} (|\nabla v|^p)(x)> (1/4)A\lambda \big\}|\\
&\qquad
 + |\big\{ x\in Q_k: \mathcal{M}_{2\overline{Q}_k} (|\nabla w|^p)(x)> (1/4)A\lambda \big\}|.\\
\endaligned
\end{equation}

For the first term in the RHS of (\ref{w-16}), we use the fact that the operator $\mathcal{M}$
is bounded from $L^1$ to weak-$L^1$. This, together with (\ref{w-15}) and (\ref{w-12}), shows that
the term is dominated by
$$
\frac{C}{A\lambda}\int_{\Omega\cap 5 \overline{Q}_k} |f|^p
\le  C\gamma A^{-1} |Q_k|,
$$
where $C$ depends only on $d$, $m$, $\mu$ and $\|A\|_{C^\sigma(\td)}$.
Since $\mathcal{L} (w)=0$ in $\Omega\cap 5 \overline{Q}_k$ and $w=0$ on $\partial\Omega\cap 5\overline{Q}_k$,
in view of Lemma \ref{lemma-w-1},
we obtain 
$$
\aligned
\|\nabla w\|_{L^\infty(\Omega\cap2\overline{Q}_k)}
&\le  C \average_{\Omega\cap 4\overline{Q}_k} |\nabla w| \\
& \le C \left(\average_{\Omega\cap 4\overline{Q}_k} |\nabla u|^p\right)^{1/p} 
+C \left(\average_{\Omega\cap 4\overline{Q}_k} |\nabla v|^p\right)^{1/p} \\
&  \le C\lambda^{1/p}  +
C \left(\average_{\Omega\cap 5\overline{Q}_k} | f|^p\right)^{1/p}\\
& \le C \lambda^{1/p},
\endaligned
$$
where we have also used estimates (\ref{w-12}) and (\ref{w-15}).
It follows  that the second term in the RHS of (\ref{w-16}) is zero, if
$A$ is large. As a result, we have proved that 
$$
|Q_k\cap E(A\lambda)|\le C \gamma A^{-1} |Q_k|
=C\gamma \rho^\sigma |Q_k|,
$$
if $\rho\in (0,1)$ is sufficiently small.
By choosing $\gamma\in (0,1)$  so small  that $C\gamma \rho^\sigma<\rho$,
we obtain $|Q_k\cap E(A\lambda)|< \rho |Q_k|$,
which is in contradiction with (\ref{w-10}).
This proves the claim (\ref{w-11}).
We should point out that the choices of $\rho$ and $\gamma$ are uniform for all $\lambda>\lambda_0$.

To proceed, we use (\ref{A-infty}) and (\ref{w-10}) to obtain 
\begin{equation}\label{w-17}
\omega(E(A\lambda)\cap Q_k) \le C \rho^\sigma \omega(Q_k).
\end{equation}
This, together with  (\ref{w-11}), leads to 
\begin{equation}\label{w-18}
\aligned
\omega(E(A\lambda)) 
&\le \omega \big( E(A\lambda)\cap \big\{x\in \Omega_R: \mathcal{M}_R (|f|^p)(x)\le \gamma \lambda\big\} \big)\\
&\qquad\qquad
+\omega\big\{ x\in \Omega_R: \mathcal{M}_R (|f|^p)(x)>\gamma \lambda \big\}\\
 & \le \sum_{k} \omega \big\{x\in E(A\lambda)\
 \cap Q_k: \  \mathcal{M}_R (|f|^p)(x)\le \gamma \lambda\big\} \\
&\qquad\qquad
+ \omega\big\{ x\in \Omega_R: \mathcal{M}_R (|f|^p)(x)>\gamma \lambda \big\}\\
& \le C\rho^\sigma\sum_{k} \omega (Q_k) 
+ \omega\big\{ x\in \Omega_R: \mathcal{M}_R (|f|^p)(x)>\gamma \lambda \big\},
\endaligned
\end{equation}
for any $\lambda >\lambda_0$,
where the last sum is taken only over those $Q_k$'s such that
$\{ x\in Q_k: \mathcal{M}_R (|f|^p)(x)\le \gamma \lambda \} \neq \emptyset$.
By the claim (\ref{w-11}) this gives
\begin{equation}\label{w-19}
\omega(E(A\lambda)) \le C \rho^\sigma \omega (E(\lambda))
+ \omega\big\{ x\in \Omega_R: \mathcal{M}_R (|f|^p)(x)>\gamma \lambda \big\}
\end{equation}
for any $\lambda >\lambda_0$.

Finally, we multiply both sides of (\ref{w-19}) by $\lambda^t$ with $t=\frac{2}{p}-1\in (0,1)$,
and integrate the resulting inequality  in $\lambda$ over the interval $(\lambda_0, \Lambda)$ to obtain 
$$
A^{-1-t}\int_{A\lambda_0}^{A\Lambda} \lambda^t 
\omega(E(\lambda))\, d\lambda
\le C\rho^\sigma \int_{\lambda_0}^\Lambda \lambda^t  \omega(E(\lambda))\, d\lambda
+ C_\gamma \int_{\Omega_R} \Big\{ \mathcal{M}_{2R} (|f|^p) \Big\}^{\frac{2}{p}}
\omega\, dx.
$$
Since $CA^{1+t}\rho^\sigma =C\rho^{-\frac{\sigma}{2} (1+t)}\rho^{\sigma}<(1/2)$ if $\rho>0$ is small, this gives
$$
\int_0^\Lambda \lambda^t  \omega(E(\lambda))\, d\lambda
\le C \int_0^{A\lambda_0} \lambda^t  \omega(E(\lambda))\, d\lambda
+ C \int_{\Omega_{2R}} |f|^2\omega\, dx,
$$
where we have used the weighted norm inequality (\ref{weighted-max}) as well
as the fact that $2/p>1$.
By letting $\Lambda\to \infty$ we obtain 
\begin{equation}\label{w-20}
\aligned
\int_{\Omega_R} |\nabla u|^2 \omega \, dx
& \le C \lambda_0^{\frac{2}{p}} \omega(\Omega_R) + C \int_{\Omega_{2R}} |f|^2\omega\, dx\\
& = \frac{C\, \omega(\Omega_{R})}{|\Omega_{2R}|}
\int_{\Omega_{2R}} |\nabla u|^2\, dx
+ C \int_{\Omega_{2R}} |f|^2\omega\, dx,
\endaligned
\end{equation}
where we have used the fact $|F|\le \mathcal{M}(F)$.
We complete the proof by letting $R\to \infty$ in (\ref{w-20}) and
using the assumption (\ref{w-2-00}).
\end{proof}

The next theorem treats the Neumann problem (\ref{NP-w}).

\begin{theorem}\label{theorem-w-3}
Let $\omega$ be an $A_1$ weight in $\rd$.
Let $u\in H^1_{\loc}(\Omega)$ be a weak solution of the Neumann problem(\ref{NP-w}).
Assume that $u$ satisfies the condition (\ref{w-2-00}).
Then the estimate (\ref{w-2-0}) holds with constant $C$
depending only on $d$, $m$, $\mu$, $\|A\|_{C^\sigma(\td)}$, and the constant in (\ref{A-1}).
\end{theorem}

\begin{proof}
The proof is similar to that of Theorem \ref{theorem-w-2}.
The only difference is that in the place of (\ref{w-14}), we need to find a function $v$ such that,
\begin{equation}\label{w-3-1}
\mathcal{L}(v) =\text{\rm div}(f\varphi) \quad \text{ in } \Omega\cap 5\overline{Q}_k
\quad \text{ and } \quad \frac{\partial v}{\partial \nu} =n\cdot (\varphi f) \quad \text{ on } \partial\Omega\cap 
5\overline{Q}_k,
\end{equation}
where $\varphi\in C_0^\infty(5\overline{Q}_k)$,
$0\le \varphi\le 1$, and $\varphi=1$ on $4\overline{Q}_k$.
The existence of functions satisfying (\ref{w-3-1}) and
the estimate (\ref{w-15})  follows from the boundary $W^{1, p}$ estimates for $\mathcal{L}_\e$
with Neumann conditions \cite{KLS1, AS}.
We omit the details and refer the reader to \cite{Geng-2012} for 
related $W^{1, p}$ estimates for Neumann problems.
\end{proof}

Finally, we go back to the case where $\omega(x)=\big[\delta(x)\big]^\alpha$.

\begin{theorem}\label{theorem-w-4}
Let  $-1<\alpha<0$, $\Omega=\mathbb{H}^d_n(a)$ and $\delta(x)$ be given by (\ref{weight}).
Let $u\in H^1_{\loc}(\Omega)$ be a weak solution of (\ref{DP-w}) in $\Omega$.
Assume that
\begin{equation}\label{w-4-0}
 R^\alpha \int_{B(x_0, R)\cap\Omega} |\nabla u|^2 \to 0 \quad \text{ as } R \to \infty,
\end{equation}
for some $x_0\in \partial\Omega$.
Then estimate (\ref{w-estimate-1}) holds with a constant $C$ depending only on $d$, $m$, $\mu$,
$\alpha$ and $\|A\|_{C^\sigma(\td)}$.
\end{theorem}

\begin{proof}
By translation we may assume that $a=0$.
Recall that $\omega(x)=\big[\delta(x)\big]^\alpha$ is an $A_1$ weight for $-1<\alpha<0$.
Also observe that in this case the assumption (\ref{w-2-00}) is reduced to (\ref{w-4-0}).
We may assume  that 
$$
\int_\Omega |h(x)|^2 \big[\delta(x)\big]^{\alpha+2}\, dx
=\int_{\partial\mathbb{H}_n^d(0)}
\left\{ \int_0^\infty | h(x^\prime-tn)|^2 t^{\alpha +2} dt \right\} d\sigma (x^\prime)
<\infty.
$$
For otherwise there is nothing to prove. It follows that for a.e. $x^\prime \in \partial\mathbb{H}_n^d(0)$,
$$
\int_s^\infty |h(x^\prime-tn)|dt 
\le \left(\int_s^\infty |h(x^\prime-tn)|^2 t^{\alpha+2}dt\right)^{1/2}
\left(\int_s^\infty t^{-\alpha-2} dt\right)^{1/2}<\infty,
$$
if $s>0$. This allows us to write 
$$
h=\text{\rm div}(H),
$$
where $H=(H_1, \dots, H_d)$ and 
$$
H_i (x)= n_i\int_0^\infty h(x-tn)\, dt.
$$
As a result, we obtain $\mathcal{L}(u) =\text{\rm div} ( f+H)$ in $\Omega$.
It follows by Theorem \ref{theorem-w-2} that
\begin{equation}\label{w-4-10}
\int_\Omega |\nabla  u|^2 \omega\, dx
\le C \int_\Omega |f|^2\omega\, dx +C \int_\Omega |H|^2\omega\, dx.
\end{equation}

Finally, we observe that for $x^\prime\in \partial\mathbb{H}_n^d(0)$,
$$
\aligned
\int_0^\infty |H(x^\prime-tn)|^2 t^\alpha\, dt
&\le \int_0^\infty \Big| \int_0^\infty |h(x^\prime-tn -sn)|\, ds \Big|^2 t^\alpha\, dt \\
 &\le  \int_0^\infty \Big| \int_t^\infty |h(x^\prime-sn)|\, ds \Big|^2 t^\alpha\, dt\\
 & \le \frac{4}{(\alpha+1)^2} \int_0^\infty |h(x^\prime-tn)|^2 t^{\alpha +2}\, dt,
 \endaligned
 $$
 where $\alpha>-1$ and a Hardy's inequality was used for the last step.
 By integrating above inequalities in $x^\prime$ over $\partial\mathbb{H}_n^d(0)$,
 we obtain 
 $$
 \int_\Omega |H(x) |^2 \big[\delta(x)\big]^\alpha \, dx
 \le C \int_\Omega |h (x)|^2 \big[\delta(x)\big]^{\alpha+2} \, dx.
 $$
 This, together with (\ref{w-4-10}), gives the weighted estimate (\ref{w-estimate-1}).
\end{proof}

The next theorem establishes  (\ref{w-estimate-1}) for the Neumann problem (\ref{NP-w}).

\begin{theorem}\label{theorem-w-5}
Let  $-1<\alpha<0$, $\Omega=\mathbb{H}^d_n(a)$ and $\delta(x)$ be given by (\ref{weight}).
Let $u\in H^1_{\loc}(\Omega)$ be a weak solution of Neumann problem (\ref{NP-w}) in $\Omega$.
Suppose that $u$ satisfies  the condition (\ref{w-4-0})
for some $x_0\in \partial\Omega$.
Then 
\begin{equation}\label{w-5-0}
\int_{\Omega} |\nabla u(x)|^2 \, \big[\delta(x)\big]^\alpha\, dx
\le C  \int_{\Omega} |f(x)|^2 \, \big[\delta(x)\big]^\alpha\, dx,
\end{equation}
where  $C$ depending only on $d$, $m$, $\mu$,
$\alpha$ and $\|A\|_{C^\sigma(\td)}$.
\end{theorem}

\begin{proof}
This follows directly from Theorem \ref{theorem-w-3}.
\end{proof}

\begin{remark}\label{remark-w-100}
{\rm 
Let $\Omega=\mathbb{H}^d_n(a)$.
Let $u\in H^1_{\loc}(\Omega)$ be a solution of $-\text{\rm div}(A(x)\nabla u)=\text{\rm div}(f)$
in $\Omega$, with either Dirichlet condition $u=0$ or Neumann condition 
$\frac{\partial u}{\partial\nu}=-n\cdot f$ on $\partial\Omega$.
Let 
\begin{equation}\label{Omega-R}
\Omega_R = \big\{ x\in \Omega: |x-(a+x\cdot n)n|\le R \text{ and } |a+x\cdot n|\le 2R \big\}.
\end{equation}
It follows from (\ref{w-20}) that for $R\ge 1$
\begin{equation}\label{w-101}
\int_{\Omega_R} |\nabla u|^2 \, \big[\delta(x)\big]^\alpha \, dx
\le  C R^\alpha \int_{\Omega_{2R}}|\nabla u|^2\, dx
+C \int_{\Omega_{2R}} |f|^2\, \big[\delta(x)\big]^\alpha \, dx,
\end{equation}
where $-1<\alpha<0$ and
$C$ depends only on $d$, $m$, $\mu$, $\alpha$, and some H\"older norm of $A$.
This will be used in Section 6.
}
\end{remark}

\begin{remark}\label{remark-4.100}
{\rm
The weighted estimates in Theorems \ref{theorem-w-4} and \ref{theorem-w-5}
also hold for the range $0<\alpha<1$, which is not used in the paper.
This may be proved by a duality argument.
}
\end{remark}

%%%%%%%%%%%%%%%%%%%%%%%%%%%%%%%%%%%%%%%%%%%%%%%%%%%%%%%%

%%%%%%%%%%%%%%%%%%%%%%%%%%%%%%%%

\section{Approximation of Neumann correctors}
\setcounter{equation}{0}

Throughout this section we assume that $\Omega$ is a bounded smooth, strictly convex domain
in $\rd$, $d\ge3$, and that $A$ is smooth and satisfies (\ref{ellipticity})-(\ref{periodicity}).
%For $x\in \partial\Omega$ and $1\le i, j\le d$, define
%\begin{equation}\label{T}
%T_{ij} (x) =(0, \dots, n_i(x), \dots, -n_j(x), \dots, 0),
%\end{equation}
%where $n(x)=(n_1(x), \dots, n_d(x))$ denotes the outward unit normal to $\partial\Omega$ at $x$. 
For $g\in C^\infty(\td; \mathbb{R}^m)$, consider the Neumann problem
\begin{equation}\label{NP-M}
\left\{
\aligned
& \mathcal{L}_\e (u_\e) =0 & \quad &\text{ in } \Omega,\\
 &\frac{\partial u_\e}{\partial\nu_\e}  = T(x)  \cdot \nabla g (x/\e) &\quad & \text{ on } \partial\Omega,
\endaligned
\right.
\end{equation}
where $T(x)=T_{ij}(x)=n_i (x) e_j -n_j(x) e_i$ for some $1\le i, j\le d$ is a tangential vector field
 on $\partial\Omega$,
 and $n(x)=(n_1(x), \dots, n_d(x))$ denotes the outward normal to $\partial\Omega$ 
at $x\in \partial\Omega$.
Fix $x_0 \in \partial\Omega$.
Assume that $n=n(x_0)$ satisfies the Diophantine condition (\ref{D-condition}) with constant
$\kappa>0$ (all constants $C$ will be independent of $\kappa$).
To approximate $u_\e$ in a neighborhood of $x_0$, we solve the
Neumann problem in a half-space 
\begin{equation}\label{NP-A-1}
\left\{
\aligned
& \mathcal{L}_\e (v_\e) =0 & \quad &\text{ in } \mathbb{H}_{n}^d (a),\\
 & \frac{\partial v_\e}{\partial\nu_\e}  = T(x_0)  \cdot \nabla g (x/\e) &\quad & \text{ on } \partial \mathbb{H}_{n}^d (a),
\endaligned
\right.
\end{equation}
where $a=-x_0\cdot n$ and $\partial \mathbb{H}_{n}^d (a)$ is the tangent plane of $\partial\Omega$
at $x_0$.
Note that if $v_\e (x)=\e w(x/\e)$, then $w$ is a solution of
\begin{equation}\label{NP-m}
\left\{
\aligned
 &\mathcal{L}_1 (w)  =0 & \quad &\text{ in } \mathbb{H}_{n}^d (a\e^{-1}),\\
 &\frac{\partial w}{\partial\nu_1}  = T(x_0)  \cdot \nabla g(x)  &\quad & \text{ on } \partial \mathbb{H}_{n}^d (a\e^{-1}).
\endaligned
\right.
\end{equation}
It then follows by Theorem \ref{theorem-2.1} that (\ref{NP-A-1}) has a bounded smooth solution $v_\e$ satisfying
\begin{equation}\label{4.10}
\|\partial_x^\alpha v_\e\|_\infty \le C_\alpha\,  \e^{1-|\alpha|}
\quad \text{ for any } |\alpha|\ge 1.
\end{equation}
In particular, $\|\nabla v_\e\|_\infty\le C$.
In view of Theorem \ref{theorem-r}, we also obtain the estimate
\begin{equation}\label{refined-main-1}
|\nabla v_\e (x)|\le \frac{C\, \e}{|x\cdot n +a|}.
\end{equation}

The goal of this section is prove the following.

\begin{theorem}\label{theorem-4.1}
Let $u_\e$ be a solution of (\ref{NP-M}) and $v_\e$ a solution of (\ref{NP-A-1}),
constructed above.
Let $\e \le r\le \sqrt{\e}$.
Then, for any $\sigma\in (0,1)$,
\begin{equation}\label{4.1-0}
\| \nabla (u_\e -v_\e)\|_{L^\infty(B(x_0, r)\cap\Omega)}
\le C \sqrt{\e} \big\{ 1+|\ln \e |\big\}
+C \e^{-1 -\sigma } r^{2+\sigma},
\end{equation}
where $C$ depends on $d$, $m$, $\mu$, $\sigma$, $\Omega$, $\|A\|_{C^k(\td)}$ and
$\| g\|_{C^k(\td)}$ for some $k=k(d)>1$.
\end{theorem}

Let $N_\e (x, y)$ denote the matrix of Neumann functions for the operator $\mathcal{L}_\e$ in $\Omega$.
Since $\Omega\subset \mathbb{H}_{n}^d(a)$ and $\mathcal{L}_\e (u_\e -v_\e)=0$ in $\Omega$,
 we obtain the representation,
\begin{equation}
\aligned
 & u_\e  (x)-v_\e (x) -\big\{ u_\e (z) -v_\e (z)\big\}\\
 &=\int_{\partial\Omega}
 \Big\{ N_\e (x, y)- N_\e (z, y) \Big\}
 \left\{ \frac{\partial u_\e}{\partial\nu_\e}
 -\frac{\partial v_\e}{\partial\nu_\e}  \right\} \, d\sigma (y)
 \endaligned
 \end{equation}
 for any $x, z\in \Omega$.
 Fix a cut-off function $\eta=\eta_\e \in C_0^\infty(B(x_0, 5\sqrt{\e}))$ such that $0\le \eta\le 1$,
 $\eta=1$ in $B(x_0, 4 \sqrt{\e})$ and $|\nabla \eta|\le C \e^{-1/2}$.
 Let
 \begin{equation}\label{I}
 I(x, z)=\int_{\partial\Omega}
 \eta (y) \Big\{ N_\e (x, y)- N_\e (z, y) \Big\}
 \left\{ \frac{\partial u_\e}{\partial\nu_\e}
 -\frac{\partial v_\e}{\partial\nu_\e} \right\} \, d\sigma (y),
 \end{equation}
 \begin{equation}\label{J}
 J(x, z)=\int_{\partial\Omega}
 (1-\eta (y)) \Big\{ N_\e (x, y)- N_\e (z, y) \Big\}
 \left\{  \frac{\partial u_\e}{\partial\nu_\e}
 -\frac{\partial v_\e}{\partial\nu_\e}  \right\} \, d\sigma (y).
 \end{equation}

We begin with the estimate of $I(x, z)$ in (\ref{I}).
Here it is essential to take advantage of the fact that the Neumann data of $v_\e$
agrees with the Neumann data of $u_\e$ at $x_0$.
Furthermore, to fully utilize the decay estimates for the derivatives of Neumann functions,
we need to transfer  the derivative from $\partial (u_\e -v_\e)/\partial \nu_\e$
to the Neumann functions.

For $x\in \mathbb{H}_{n}^d (a)$, we use
\begin{equation}\label{hat}
\widehat{x}=x- ((x-x_0)\cdot n)) n \in \partial \mathbb{H}_n^d(a)
\end{equation}
to denote its projection onto the tangent plane of $\partial\Omega$ at $x_0$.
Observe that if $x\in \partial\Omega$, then $|x-\widehat{x}|\le C|x-x_0|^2$.

\begin{lemma}\label{lemma-4.2}
Suppose that $x\in\partial\Omega$ and  $|x-x_0|\le c_0$. Then
\begin{equation}\label{4.2-0}
\aligned
 & n(x)\cdot A^\e(x)\nabla v_\e (x)
-n(x_0) \cdot A^\e (\widehat{x})\nabla v_\e (\widehat{x})=\\
&
T_{i\ell} (x)\cdot
\nabla_x
\left\{ \int_0^1 a^\e_{ij}(x) \frac{\partial v_\e}{\partial x_j}
\big( sx +(1-s)\widehat{x}\big) \, ds \, ((x-x_0)\cdot n) n_\ell  \right\}
+ R(x),
\endaligned
\end{equation}
where $T_{i\ell}(x)=n_i (x) e_\ell -n_\ell (x) e_i$, $n=n(x_0)=(n_1, \dots, n_d)$,  and 
\begin{equation}\label{4.2-1}
|R(x)|\le C \big\{  |x-x_0| +\e^{-1} |x-x_0|^3 \big\}.
\end{equation}
\end{lemma}

\begin{proof}
In view of (\ref{4.10}) we have $\| \nabla v_\e\|_\infty \le C$ and
$\|\nabla^2 v_\e\|_\infty \le C \e^{-1}$.
It follows that
$$
\aligned
& n(x)\cdot A^\e(x)\nabla v_\e (x) -n(x_0 )\cdot A^\e(\widehat{x})\nabla v_\e (\widehat{x})\\
&=n(x_0)\cdot A^\e(x)\nabla v_\e (x) -n(x_0)\cdot A^\e(\widehat{x})\nabla v_\e (\widehat{x}) +
O(|x-x_0|)\\
&=n_i(x_0) \int_0^1 \frac{\partial}{\partial s}
\left\{ a_{ij}^\e \frac{\partial v_\e}{\partial x_j}
\big(sx +(1-s)\widehat{x}\big) \right\} ds + O(|x-x_0|)\\
& =n_i(x_0) \int_0^1 \frac{\partial}{\partial x_\ell}
\left( a_{ij}^\e \frac{\partial v_\e}{\partial x_j} \right)
\big( sx +(1-s)\widehat{x}\big) ((x-x_0)\cdot n) n_\ell (x_0)\, ds 
+ O(|x-x_0|)\\
&= \int_0^1  \left(n_i(x_0)\frac{\partial}{\partial x_\ell}
-n_\ell (x_0) \frac{\partial}{\partial x_i} \right)
\left( a_{ij}^\e \frac{\partial v_\e}{\partial x_j} \right)
\big( sx +(1-s)\widehat{x}\big) ((x-x_0)\cdot n) n_\ell (x_0)\, ds \\
& \qquad\qquad
+ O(|x-x_0|),
\endaligned
$$
where we have used the equation $\mathcal{L}_\e (v_\e)=0$ in the last step.
Using the observation that
$$
\aligned
 & \left(n_i(x_0)\frac{\partial}{\partial x_\ell}
-n_\ell (x_0) \frac{\partial}{\partial x_i} \right)
\Big( F(sx +(1-s)\widehat{x})\Big)\\
&\qquad
=\left(n_i(x_0)\frac{\partial}{\partial x_\ell}
-n_\ell (x_0) \frac{\partial}{\partial x_i} \right)
 F(sx +(1-s)\widehat{x}),
 \endaligned
 $$
 we then obtain 
 $$
 \aligned
& n(x)\cdot A^\e(x)\nabla v_\e (x) -n(x_0 )\cdot A^\e(\widehat{x})\nabla v_\e (\widehat{x})\\
&=  \left(n_i(x_0)\frac{\partial}{\partial x_\ell}
-n_\ell (x_0) \frac{\partial}{\partial x_i} \right)\left( \int_0^1
\left( a_{ij}^\e \frac{\partial v_\e}{\partial x_j} \right)
\big( sx +(1-s)\widehat{x}\big) ((x-x_0)\cdot n) n_\ell (x_0)\, ds \right)\\
& \qquad\qquad
+ O(|x-x_0|)\\
&=  \left(n_i(x)\frac{\partial}{\partial x_\ell}
-n_\ell (x) \frac{\partial}{\partial x_i} \right)\left( \int_0^1
\left( a_{ij}^\e \frac{\partial v_\e}{\partial x_j} \right)
\big( sx +(1-s)\widehat{x}\big) ((x-x_0)\cdot n) n_\ell (x_0)\, ds \right)\\
& \qquad\qquad
+ O(|x-x_0|) +O(\e^{-1} |x-x_0|^3),
\endaligned
$$
where we have used the fact that $|(x-x_0)\cdot n| \le C |x-x_0|^2$ 
as well as the estimate $|\nabla (A^\e \nabla v_\e)|\le C \e^{-1}$
for the last step.
\end{proof}

Lemma \ref{lemma-4.2} allows us to carry out an integration by parts on the boundary
for $I(x, z)$.

\begin{lemma}\label{lemma-4.3}
Let $I(x, z)$ be given by (\ref{I}). Suppose that $x, z\in B(x_0, 3r)\cap \Omega$
for some ${\e}\le r\le \sqrt{\e}$ and $|x-z|\le (1/2)\delta (x)$, where $\delta (x)=dist(x, \partial\Omega)$.
Then
\begin{equation}\label{4.3-0}
|I(x, z)|\le C\,  r\sqrt{\e}.
\end{equation}
\end{lemma}

\begin{proof}
Let $y\in \partial\Omega$ and $|y-x_0|\le 5  \sqrt{\e}$.
Using the Neumann conditions for $u_\e$, $v_\e$ and Lemma \ref{lemma-4.2},
we see that
$$
\aligned
& \frac{\partial u_\e}{\partial \nu_\e} (y) -\frac{\partial v_\e}{\partial \nu_\e} (y)\\
&= T(y)\cdot \nabla g (y/\e)-T(x_0)\cdot \nabla g (\widehat{y}/\e)
+n(x_0)\cdot A^\e (\widehat{y}) \nabla v_\e (\widehat{y}) 
- n(y)\cdot A^\e (y) \nabla v_\e (y)\\
&=T(y)\cdot \nabla_y \Big\{ \e g(y/\e) -\e g(\widehat{y}/\e) \Big\}
-T_{i\ell} (y) \cdot \nabla_y \big\{ f_{i \ell} (y) \big\}  
+ O(|y-x_0|),
\endaligned
$$
where
$$
f_{i\ell} (y)
=\int_0^1 a^\e_{ij}(y) \frac{\partial v_\e}{\partial y_j}
\big( sy +(1-s)\widehat{y}\big) \, ds \, ((y-x_0)\cdot n) n_\ell 
$$
is given by Lemma \ref{lemma-4.2}. We have also used the observation,
$$
T(x_0)\cdot \nabla_y \big\{ \e g(\widehat{y}/\e)\big\}
=T(x_0) \cdot \nabla g (\widehat{y}/\e),
$$
in the computation above.
This, together with (\ref{parts}), gives
$$
I(x, z)=I_1 +I_2 +I_3,
$$
where
$$
\aligned
 &I_1=-\int_{\partial\Omega}
T(y) \cdot\nabla_y \Big\{ \eta(y) \Big(N_\e x, y)-N_\e (z,y)\Big) \Big\}
\Big\{ \e g(y/\e) - \e g(\widehat{y}/\e) \Big\}\, d\sigma (y),\\
& I_2 = \int_{\partial\Omega}
T_{i\ell} (y) \cdot\nabla_y \Big\{ \eta(y) \Big(N_\e x, y)-N_\e (z,y)\Big) \Big\}
 f_{i\ell} (y)\, d\sigma (y),\\
 & |I_3|
 \le C \int_{\partial\Omega}
 |\eta(y)| |N_\e (x, y)-N_\e (z, y)|  |y-x_0|\, d\sigma (y).
 \endaligned
 $$
 Since $|x-z|\le (1/2)\delta (x)\le (1/2) |y-x|$ for any $y\in \partial\Omega$,
  by (\ref{N-estimate-1}), we obtain 
 \begin{equation}\label{4.3-10}
 \aligned
 |I_1|   &\le C \int_{\partial\Omega}
 |\nabla_y \big\{ \eta(y) \big(N_\e (x, y)-N_\e (z,y)\big) \big\}| |y-x_0|^2\,d\sigma (y)\\ 
 &\le C \sqrt{\e} \delta(x) \int_{4\sqrt{\e}\le |y-x_0|\le 5\sqrt{\e}} \frac{d\sigma(y)}{|y-x|^{d-1}}
  + C \, \delta (x)  \int_{|y-x_0|\le C \sqrt{\e}}
 \frac{ |y-x_0|^2}{|y-x|^d} \, d\sigma (y),
 \endaligned
\end{equation}
where we have used the fact $|y-\widehat{y}|\le C |y-x_0|^2$ and $|\nabla\eta(y)|\le C \e^{-1/2}$.
For the first term in the RHS of (\ref{4.3-10}), we note that
if $|y-x_0|\ge 4\sqrt{\e}$, then
$$
|y-x|\ge |y-x_0|-|x-x_0|\ge 4\sqrt{\e} -3r\ge \sqrt{\e}.
$$
For the second term, we use  $|y-x_0|\le |y-x| +r$.
This leads to
$$
\aligned
|I_1| & \le  C \sqrt{\e} \, \delta(x) +C \delta (x) \int_{|y-x_0|\le C\sqrt{\e}} \frac{d\sigma (y)}{|y-x|^{d-2}}
+Cr^2  \delta (x)\int_{\partial\Omega}  \frac{d\sigma (y)}{|y-x|^{d}}\\
&\le C \sqrt{\e} \, \delta(x) + C r^2\\
&\le C r\sqrt{\e}.
\endaligned
$$
Since $|f_{i\ell}|\le C |y-x_0|^2$, the estimate of $I_2$ is the exactly same as that of $I_1$.

Finally, to handle $I_3$,   we use (\ref{N-estimate-1}) as well as $|y-x_0|\le |y-x| +r$
again to obtain 
$$
\aligned
|I_3|
&\le C \int_{|y-x_0|\le C \sqrt{\e}}
\frac{|x-z|}{|y-x|^{d-1}}
\big\{  |y-x| + r \big\} \, d\sigma (y)\\
&\le C\, r \sqrt{\e}.
\endaligned
$$
This completes the proof.
\end{proof}

To estimate $J(x,z)$ in (\ref{J}), we split it as $J(x,z)=J_1 -J_2$,
where
\begin{equation}\label{J-1-2}
\aligned
J_1 (x, z) & =\int_{\partial\Omega}
 (1-\eta (y)) \Big\{ N_\e (x, y)- N_\e (z, y) \Big\}
  \frac{\partial u_\e}{\partial\nu_\e} \, d\sigma (y),\\
  J_2 (x, z) & =\int_{\partial\Omega}
 (1-\eta (y)) \Big\{ N_\e (x, y)- N_\e (z, y) \Big\}
  \frac{\partial v_\e}{\partial\nu_\e} \, d\sigma (y).
  \endaligned
 \end{equation}

\begin{lemma}\label{lemma-4.4}
Let $J_1 (x, z)$ be given by (\ref{J-1-2}). Suppose that $x, z\in B(x_0,3 r)\cap \Omega$
for some ${\e}\le r\le \sqrt{\e}$ and $|x-z|\le (1/2)\delta (x)$, where $\delta (x)=dist(x, \partial\Omega)$.
Then
\begin{equation}\label{4.4-0}
|J_1 (x, z)|\le C\,  r\sqrt{\e}.
\end{equation}
\end{lemma}

\begin{proof}
Using the Neumann condition for $u_\e$ and (\ref{parts}), we see that
$$
J_1(x, z)=-\int_{\partial\Omega}
T\cdot \nabla_y \Big\{ (1-\eta(y)) \big( N_\e (x, y)-N_\e (z, y) \big) \Big\} \e g(y/\e)\, d\sigma (y).
$$
It follows that
$$
\aligned
|J_1(x, z)| & \le 
C  \e \| g\|_\infty \int_{\partial\Omega}
|\nabla_y \Big\{ (1-\eta(y)) \big( N_\e (x, y)-N_\e (z, y) \big) \Big\}|\, d\sigma (y)\\
&\le C \sqrt{\e} |x-z|
+ C\e |x-z| \int_{ |y-x_0|\ge  5\sqrt{\e}}
\frac{d\sigma (y)}{|x-y|^d}\\
& \le C\, r\sqrt{\e},
\endaligned
$$
where we have used (\ref{N-estimate-1}) for the second inequality.
\end{proof}

It remains to estimate $J_2 (x, z)$.

\begin{lemma}\label{lemma-4.6}
Let $J_2 (x, z)$ be given by (\ref{J-1-2}). Suppose that $x, z\in B(x_0,3 r)\cap \Omega$
for some ${\e}\le r\le \sqrt{\e}$ and $|x-z|\le (1/2)\delta (x)$, where $\delta (x)=dist(x, \partial\Omega)$.
Then
\begin{equation}\label{4.6-0}
|J_2 (x, z)|\le C\,  r\sqrt{\e}\big\{ 1+|\ln \e|\big\}.
\end{equation}
\end{lemma}

\begin{proof}
It follows by the divergence theorem that
$$
\aligned
J_2(x,z)
&=-\int_{\Omega} (1-\eta (y))\nabla_y
\big\{ N_\e (x, y) -N_\e (z, y) \big\}  \cdot A (y/\e) \nabla v_\e (y)\, dy\\
& \qquad -\int_{\Omega} \big\{ N_\e (x, y)-N_\e (z, y) \big\} \nabla_y (1-\eta (y)) \cdot A (y/\e) 
 \nabla v_\e (y)\, dy\\
 &= \int_{\Omega}
 A^{*} (y/\e) \nabla_y \big\{ N_\e (x, y) -N_\e (z, y) \big\} \cdot \nabla_y (1-\eta (y))  \, (v_\e (y)-E)\, dy\\
 & \qquad -\int_{\Omega} \big\{ N_\e (x, y)-N_\e (z, y) \big\} \nabla_y (1-\eta (y)) \cdot A (y/\e) 
 \nabla v_\e (y)\, dy,
\endaligned
$$
where $E\in \mathbb{R}^m$ is a constant to be chosen. Here
we have used  $\mathcal{L}_\e (v_\e)=0$ in $\Omega$ for the first equality and
$$
\aligned
\mathcal{L}_\e^* \big\{ N_\e (x, \cdot) -N_\e (z, \cdot) \big\}  &=0  \quad \text{ in } \Omega \setminus B(x_0, 3\sqrt{\e}),\\
\frac{\partial}{\partial \nu_\e^*}  \big\{ N_\e (x, \cdot) -N_\e (z, \cdot) \big\}  &=0  \quad \text{ on } \partial\Omega
\endaligned
$$
for the second.
As before, we apply the estimates in (\ref{N-estimate-1}) to obtain 
\begin{equation}\label{4.6-10}
\aligned
|J_2 (x, z)|
&\le  \frac{C |x-z|}{(\sqrt{\e})^{d+1}} \int_{B(x_0, 5\sqrt{\e})\cap \mathbb{H}_n^d(a)}
|v_\e -E|
+ \frac{C |x-z|}{(\sqrt{\e})^{d}} \int_{B(x_0, 5\sqrt{\e})\cap \mathbb{H}_n^d(a)}
| \nabla v_\e|\\
& \le \frac{C r}{(\sqrt{\e})^{d}} \int_{B(x_0, 5\sqrt{\e})\cap \mathbb{H}_n^d(a)}
| \nabla v_\e|,
\endaligned
\end{equation}
where we have chosen $E$ to be the average of $v_\e$ over $B(x_0, 5\sqrt{\e})\cap \mathbb{H}_n^d(a)$
and used a Poincar\'e type inequality for the last step.

Finally, to estimate the integral in the RHS of (\ref{4.6-10}), we split the region 
$B(x_0, 5\sqrt{\e})\cap \mathbb{H}_n^d (a)$ into two parts.
If $|x\cdot n+a|\le \e$, we use the estimate $\|\nabla v_\e\|_\infty \le C$.
If $|x\cdot n+a|\ge \e$, we apply the refined estimate (\ref{refined-main-1}). 
This yields that
$$
|J_2 (x, z)|\le C r \sqrt{\e} \big\{ 1+|\ln \e|\big\},
$$
which completes the proof.
\end{proof}

We are now ready to give the proof of Theorem \ref{theorem-4.1}.

\begin{proof}[\bf Proof of Theorem \ref{theorem-4.1}]
Let $\e\le r\le \sqrt{\e}$.
In view of Lemmas \ref{lemma-4.3}, \ref{lemma-4.4} and \ref{lemma-4.6},
we have proved that if $x, z\in \Omega\cap B(x_0, 3r)$
and $|x-z|<(1/2)\delta(x)$, where $\delta(x)= \text{dist}(x, \partial\Omega)$, then
\begin{equation}\label{4.7-1}
|u_\e (x)-v_\e(x) - \{ u_\e (z)-v_\e (z)\}|
\le C r \sqrt{\e} \big\{ 1+|\ln \e |\big\}.
\end{equation}
Since $\mathcal{L}_\e (u_\e -v_\e)=0$ in $\Omega$,
by the interior Lipschitz estimate for $\mathcal{L}_\e$ \cite{AL-1987},
it follows that for any $x\in B(x_0, 2r)$,
\begin{equation}\label{4.7-2}
|\nabla u_\e (x) -\nabla v_\e (x)|\le 
C r\sqrt{\e} \big\{ 1+|\ln \e |\big\} \big[\delta(x)\big]^{-1}.
\end{equation}
Thus, if $0<p<1$, 
\begin{equation}\label{4.7-3}
\left(\average_{B(x_0, 2r)\cap\Omega} |\nabla (u_\e -v_\e)|^p \right)^{1/p}
\le  C \sqrt{\e} \big\{ 1+|\ln \e |\big\}.
\end{equation}

Next, we estimate the $C^\sigma(B(x_0, 2r)\cap\partial \Omega)$ norm of
$$
F(y)=\frac{\partial u_\e}{\partial \nu_\e} -\frac{\partial v_\e}{\partial \nu_\e}.
$$
As in the proof of Lemma \ref{lemma-4.3}, we write
$$
\aligned
F(y) &=\Big\{ T(y)\cdot \nabla g (y/\e)-T(x_0)\cdot \nabla g (\widehat{y}/\e)\Big\}\\
& \qquad+\Big\{ n(x_0)\cdot A (\widehat{y}/\e) \nabla w (\widehat{y}/\e) 
- n(y)\cdot A (y/\e) \nabla w (y/\e)\Big\}\\
&=F_1 (y) + F_2(y),
\endaligned
$$
where we have used the fact $v_\e(x)=\e w(x/\e)$ and $w$ is a solution of (\ref{NP-m}).
Using $|y-\widehat{y}|\le C |y-x_0|^2$ and $\| \nabla w\|_\infty +\|\nabla^2 w\|_\infty
\le C$,
it is easy to see that if $y\in B(x_0, 2r)\cap\partial\Omega$,
\begin{equation}\label{4.7-10}
|F_1(y)| +|F_2(y)|
\le C |y-x_0| + C \e^{-1} |y-x_0|^2\le C\,  \e^{-1} r^2,
\end{equation}
where we also used the assumption $\e\le r$ for the last step.
By extending $n(y)$ smoothly to a neighborhood of $\partial\Omega$,
we may assume that $F(y)$ is defined in $B(x_0, c_0)\cap \mathbb{H}_n^d$.
A computation shows that
\begin{equation}\label{4.7-11}
|\nabla_y F(y)|\le C \big\{ 1 +\e^{-1} |y-x_0| +\e^{-2} |y -x_0|^2 \big\}
\le C \e^{-2} r^2,
\end{equation}
where we have used the estimate $\|\nabla^3 w\|_\infty\le C$.
By interpolation it follows from (\ref{4.7-10}) and (\ref{4.7-11}) that
\begin{equation}\label{4.7-12}
\| F\|_{C^{0, \sigma}(B(x_0, 2r)\cap\partial\Omega)}
\le C\,  (\e^{-1} r^2)^{1-\sigma} (\e^{-2} r^2)^{\sigma}
=C \, \e^{-1-\sigma} r^2
\end{equation}
for any $\sigma \in (0, 1)$.

Finally, since $\mathcal{L_\e}(u_\e -v_\e)=0$
in $\Omega \cap B(x_0, 2r)$,
we apply the boundary Lipschitz estimate for solutions with Neumann data  \cite{KLS1, AS}
to obtain
$$
\aligned
& \| \nabla (u_\e -v_\e)\|_{L^\infty(\Omega\cap B(x_0, r))}\\
 &\le  C \left(\average_{B(x_0, 2r)\cap\Omega} |\nabla (u_\e -v_\e)|^p \right)^{1/p}
+ C \|F \|_{L^\infty(B(x_0, 2r)\cap\partial\Omega)}
+ C r^\sigma \| F\|_{C^{0, \sigma} (B(x_0, 2r)\cap \partial\Omega)}\\
&\le C \sqrt{\e} \big\{ 1+|\ln \e |\big\}
+C \e^{-1} r^2
+C \e^{-1 -\sigma } r^{2+\sigma}\\
&\le  C \sqrt{\e} \big\{ 1+|\ln \e |\big\}
+C \e^{-1 -\sigma } r^{2+\sigma}.\\
\endaligned
$$
This completes the proof.
\end{proof}

Recall that the function $\psi_{\e, k}^{*\beta}
=(\psi_{\e, k}^{* 1\beta} (y), \dots, \psi_{\e, k}^{* m \beta}(y)$ in (\ref{psi}) is a solution of the Neumann 
problem 
\begin{equation}\label{NP-E}
\left\{
\aligned
&\mathcal{L}_\e^* (\psi_{\e, k}^{*\beta})=0 &\quad & \text{ in }\Omega,\\
& \frac{\partial}{\partial \nu_\e^*} 
\big\{ \psi_{\e, k}^{*\beta} \big\} = - \frac12 (n_i(x) e_\ell -n_\ell (x) e_i)\cdot 
\nabla  f_{\ell i k}^\beta (x/\e)  &\quad &\text{ on }\partial\Omega,
\endaligned
\right.
\end{equation}
where the 1-periodic functions  $f_{\ell i k}^\beta\in C^\infty(\td; \mathbb{R}^m)$ are  given by (\ref{p102}).
For each $x_0\in \partial\Omega$ fixed and satisfying (\ref{D-condition}),
in view of Theorem \ref{theorem-4.1},
 we may approximate this function in a small neighborhood of $x_0$
by  a solution of
\begin{equation}\label{NP-A}
\left\{
\aligned
&\mathcal{L}_\e^* (\phi_{\e, k }^{*\beta, x_0 })=0 &\quad & \text{ in } \mathbb{H}_n^d(a),\\
& \frac{\partial}{\partial \nu_\e^*} 
\big\{ \phi_{\e, k}^{*\beta, x_0} \big\} = - \frac12 (n_i e_\ell -n_\ell e_i)\cdot 
\nabla  f_{\ell i k}^\beta (x/\e)  &\quad &\text{ on }\partial\mathbb{H}_n^d(a),
\endaligned
\right.
\end{equation}
where $n=n(x_0)$ and $\mathbb{H}_n^d(a)$ is the tangent plane of $\partial\Omega$ at $x_0$.
Recall that by a change of variables, a solution of (\ref{NP-A}) is given by
\begin{equation}\label{E-1}
\phi_{\e, k}^{\alpha\beta, x_0}
(x)=\e V_k^{*\alpha\beta, n}\left(\frac{x-(x\cdot n+a)n}{\e}, -\frac{x\cdot n +a}{\e} \right),
\end{equation}
where
$V^*=V_{k}^{*\beta, n} (\theta, t)=(V_k^{* 1\beta, n} (\theta, t), \dots, V_k^{*m\beta, n}(\theta, t))$ is the smooth solution of
\begin{equation}\label{E-2}
\left\{
\aligned
& \left( \begin{array} {c}N^T \nabla_\theta \\ \partial_t  \end{array}\right)
\cdot M^TA^*(\theta -tn)M \left( \begin{array} {c}N^T \nabla_\theta \\ \partial_t  \end{array} \right)V^*
 =0  &\quad & \text{in } \mathbb{T}^d \times \mathbb{R}_+,\\
&-e_{d+1} \cdot M^T A^*(\theta)M \left( \begin{array} {c}N^T \nabla_\theta \\ \partial_t  \end{array} \right)V^*
 =-\frac12 (n_i e_\ell -n_\ell e_i)\cdot \nabla_\theta f_{lik}^\beta & \quad & \text{on } \mathbb{T}^d \times \{ 0\},
\endaligned
\right.
\end{equation}
given by Remark \ref{remark-2.1}. As a result, we may deduce the following from Theorem \ref{theorem-4.1}

\begin{theorem}\label{main-lemma}
Let $\e\le r\le \sqrt{\e}$ and   $\sigma\in (0,1/2)$.
Then for any $x\in B(x_0, r)\cap\Omega$,
\begin{equation}\label{main-estimate-4}
\aligned
& \Big| \nabla \left(\Psi_{\e, k}^{*\alpha\beta} (x) - P_k^{\alpha\beta} (x)
-\e \chi_k^{*\alpha\beta} \left(\frac{x}{\e}\right)
-\e V_k^{*\alpha\beta, n}\left(\frac{x-(x\cdot n+a)n}{\e}, -\frac{x\cdot n +a}{\e} \right)\right) \Big|\\
& \le C\sqrt{\e} \big\{ 1+|\ln \e|\big\}
+C \e^{-1-\sigma} r^{2+\sigma},
\endaligned
\end{equation}
where $C$ depends only on $d$, $m$, $\sigma$, $\mu$, $\Omega$ and
$\|A\|_{C^k(\td)}$  for some $k=k(d)>1$.
\end{theorem}

%%%%%%%%%%%%%%%%%%%%%%%%%%%%%%%%%%%%%%%%%%%%%%%

%%%%%%%%%%%%%%%%%%%%%%%%%%%%%%%%%%%%%%%%%%%%%%%%%%

\section{Estimates of the homogenized data}
\setcounter{equation}{0}

Observe that by (\ref{E-1}),
\begin{equation}\label{E-3}
\aligned
T_{ij}(x_0) \cdot \nabla_x \phi_{\e, k}^{*\alpha\beta, x_0} (x)
&=T(x_0) \cdot (I-n\otimes n, -n) \left( \begin{array}{c} \nabla_\theta \\ \partial_t\end{array}\right)
V_k^{*\alpha\beta, n} \left(\frac{x}{\e}, 0\right)\\
& =T_{ij}(x_0) \cdot  \nabla_\theta V_k^{*\alpha\beta, n} \left(\frac{x}{\e}, 0\right),
\endaligned
\end{equation}
where $x\in \partial\mathbb{H}_n^d(a)$ and
we have used the fact $T_{ij}(x_0)\cdot n(x_0)=0$.
For $x\in \partial\Omega$, define
\begin{equation}\label{mean}
\aligned
\widetilde{g}_{k}^\beta (x)
 &= T_{ij}(x)\cdot\left\langle \nabla_\theta (V_k^{*\alpha\beta,n} +\chi_k^{*\alpha\beta})
 (\cdot, 0) g_{ij}^\alpha(x, \cdot)\right\rangle
 +(T_{ij}\cdot \nabla_x) x_k \langle g_{ij}^\beta (x, \cdot) \rangle\\
& = T_{ij}(x)\cdot \int_{\td}  (I-n\otimes n)\nabla_\theta (V_k^{*\alpha\beta, n}+\chi_k^{*\alpha\beta})
 (\theta, 0) g_{ij}^\alpha(x, \theta)\, d\theta\\
&\qquad\qquad
 +(n_i \delta_{jk}-n_j \delta_{ik}) \int_{\td} g_{ij}^\beta (x, \theta) d\theta\\
 &=T_{ij}(x)\cdot
 \int_{\td}
 \Big( e_k \delta^{\alpha\beta}
 +\nabla_\theta \chi_k^{*\alpha\beta} (\theta)
 +\nabla_\theta V_k^{*\alpha\beta} (\theta, 0) \Big) g_{ij}^\alpha (x, \theta)\, d\theta,
\endaligned
\end{equation}
where  $n=n(x)$.
Using the estimate $\|(I-n\otimes n)\nabla_\theta V^*\|\le C$ in Proposition \ref{prop-2.5}, we obtain 
$\|\widetilde{g}\|_\infty \le C \| g\|_\infty$.

Let  
\begin{equation}\label{v}
v^\gamma_\e (x)
=-\int_{\partial\Omega}
\big( T_{ij}(y)\cdot \nabla_y \big)\Psi_{\e, k}^{*\alpha\beta} (y)
\cdot \frac{\partial}{\partial y_k} \big\{ N_0^{\gamma \beta} (x, y)\big\} \cdot
 g_{ij}^\alpha (y, y/\e)\, d\sigma (y)
 \end{equation}
 be the first term in the RHS of (\ref{p-3}).
In the next section we will show that as $\e \to 0$, 
\begin{equation}\label{p-200}
v_\e^\gamma (x)
\to v_0^\gamma (x)=-\int_{\partial\Omega}
\frac{\partial}{\partial y_k} \big\{ N_0 ^{\gamma \beta} (x, y)\big\}
 \widetilde{g}_{k}^\beta (y)\, d\sigma (y).
\end{equation}
Fix $1\le \gamma\le m$ and $1\le k\le d$.
Using
\begin{equation}\label{p-201}
\aligned
& n_in_j\widehat{a}_{ij}^{*\alpha\beta} \frac{\partial}{\partial y_k} \big\{ N_0^{\gamma\beta} (x, y) \big\}\\
&=n_i\widehat{a}_{ij}^{*\alpha\beta}
\Big( (n_j e_k-n_k e_j)\cdot \nabla_y \Big) N_0^{\gamma \beta} (x, y)
+n_k \left(\frac{\partial}{\partial \nu_0^* (y)} \big\{ N^\gamma_0(x, y) \big\}\right)^\alpha,
\endaligned
\end{equation}
where $N_0^\gamma (x, y)=(N_0^{\gamma 1} (x, y), \cdots, N_0^{\gamma m}(x,y))$,
we may write
\begin{equation}\label{p-202}
\frac{\partial}{\partial y_k} \big\{ N_0^{\gamma\beta} (x, y)\big\}
=h^{*\beta\alpha} n_i \widehat{a}^{*\alpha t}_{ij} 
\Big( (n_j e_k-n_k e_j)\cdot \nabla_y \Big) N_0^{ \gamma t} (x, y)
- h^{*\beta \gamma } n_k |\partial\Omega|^{-1},
\end{equation}
where $h^*(y)=(h^{*\alpha\beta}(y))$ is the inverse of  the $m\times m$ matrix
$(\widehat{a}_{ij}^{*\alpha\beta} n_i n_j)$ and we have used the fact that
the conormal derivative of the matrix of Neumann functions is $-|\partial\Omega|^{-1} I_{m\times m}$.
It follows that 
\begin{equation}\label{v-0}
\aligned
v_0^\gamma (x)
= &-\int_{\partial\Omega}
\big[ (n_j e_k-n_k e_j)\cdot \nabla_y\big] N_0^{\gamma t} (x, y)
\cdot h^{*\beta \alpha} n_i \widehat{a}^{*\alpha t}_{ij} \widetilde{g}_k^\beta (y)\, d\sigma (y)\\
&\qquad\qquad
+\average_{\partial\Omega} h^{* \beta\gamma} n_k (y) \widetilde{g}^\beta_k (y)\, d\sigma (y)\\
=&\int_{\partial\Omega} N_0^{\gamma t} (x,y)
\Big[ (n_j e_k-n_k e_j)\cdot \nabla_y\Big] 
\Big( n_i \widehat{a}^{t\alpha}_{ij} h^{\alpha\beta} \widetilde{g}^\beta_k (y)\Big) \, d\sigma (y)
+ \text{constant},
\endaligned
\end{equation}
where $h=(h^*)^*=(h^{\alpha\beta})$ is the inverse of the matrix $(\widehat{a}_{ij}^{\alpha\beta}n_in_j)$.
This shows that $v_0$ is a solution of the following Neumann problem,
\begin{equation}\label{NP-v}
\left\{
\aligned
&\mathcal{L}_0 (v_0)=0 &\quad & \text{ in } \Omega,\\
&\left(\frac{\partial v_0}{\partial \nu_0}\right)^\gamma
=\big( T_{jk}\cdot \nabla\big)
\Big( n_i  \widehat{a}_{ij}^{\gamma \alpha}  h^{\alpha\beta}  \widetilde{g}^\beta_k (y)\Big)
&\quad  & \text{ on }  \partial\Omega.
\endaligned
\right.
\end{equation}
Thus the homogenized data  $\overline{g}_{jk}^\gamma$ in (\ref{NP-h}) is
given by 
\begin{equation}\label{h-data-ij}
\overline{g}_{jk}^\gamma  
=n_i\widehat{a}_{ij}^{\gamma\alpha} h^{\alpha\beta} \widetilde{g}_k^\beta,
\end{equation}
where $\widetilde{g}_k^\beta$ is given by (\ref{mean})
and $(h^{\alpha\beta})$ is the inverse of the matrix $(\widehat{a}_{ij}^{\alpha\beta}n_in_j)$.

The rest of this section is devoted to the proof of the following.

\begin{theorem}\label{theorem-E}
Let $x, y\in \partial\Omega$ and $|x-y|\le c_0$.
Suppose that $n(x)$ and $n(y)$ satisfy the Diophantine condition (\ref{D-condition})
with constants $\kappa (x)$ and $\kappa (y)$, respectively. 
Let $\overline{g}= (\overline{g}_k^\beta)$ be defined by (\ref{h-data-ij}). 
Then, for any $\sigma\in (0,1)$,
\begin{equation}\label{E-1-0}
|\overline{g}(x)-\overline{g}(y)|\le \frac{C_\sigma |x-y|}{ \kappa^{1+\sigma}}
\left(\frac{|x-y|}{\kappa} + 1\right) \sup_{z\in \td} \| g(\cdot, z)\|_{C^1(\partial\Omega)},
\end{equation}
where $\kappa=\max(\kappa(x), \kappa(y))$ and
$C_\sigma$ depends only on $d$, $m$, $\sigma$, $\mu$,
and $\|A\|_{C^k(\td)}$ for some $k=k(d)>1$.
\end{theorem}

Assume that  $n, \wn\in \mathbb{S}^{n-1}$ satisfy the condition (\ref{D-condition}).
Choose two orthogonal matrices $M_n$, $M_{\wn}$ such that $M_n (e_d)=-n$,
$M_{\wn} e_d=-\wn$ and $|M_n -M_{\wn}|\le C |n-\wn|$.
Let $N_n$ and $N_{\wn}$ denote the $d\times (d-1)$ matrices of the first $d-1$ columns of
$M_n$ and $M_{\wn}$, respectively.
Let $V_n^* (\theta, t) $ and $V_{\wn}^*(\theta, t)$ be the corresponding solutions of (\ref{E-2}).
We will show that
\begin{equation}\label{E-1-1}
\int_{\td} |N_n^T \nabla_\theta \big(V_n^*(\theta, 0) -V_{\wn}^*( \theta, 0)\big)|\, d\theta
\le  \frac{C_\sigma | n-\wn|}{ \kappa^{1+\sigma}}
\left(\frac{|n-\wn|}{\kappa} + 1\right),
\end{equation}
where $\kappa>0$ is the constant in the Diophantine condition (\ref{D-condition}) for $\wn$.
Using $N_nN_n^T=I -n\otimes n$,
$N_{\wn} N_{\wn}^T=I -\wn\otimes \wn$, and the estimate
$|\nabla_\theta V^*_{\wn} |\le C\kappa^{-1}$ from Proposition  \ref{prop-2.5},
it is not hard to see that (\ref{E-1-0}) follows from (\ref{E-1-1}).
Furthermore, let 
\begin{equation}\label{W}
W(\theta, t)= V_n^*(\theta, t)-V_{\wn}^*(\theta, t).
\end{equation}
By Sobolev imbedding it suffices to show that
\begin{equation}\label{E-1-1-0}
\int_0^1 \int_{\td}
\Big\{ |N_n^T \nabla_\theta W|^2 
+| \nabla_\theta \partial_t W|^2 \Big\} d\theta dt
\le C_\sigma \left\{ \frac{|n-\wn|^2}{\kappa^{2+\sigma}}
+\frac{|n-\wn|^4}{\kappa^{4+\sigma}}\right\},
\end{equation}
for any $\sigma \in (0,1)$.

Let
$$
B^*_n(\theta, t)=M_n^T A^*(\theta-tn) M_n\quad \text{ and } \quad
B^*_{\wn}(\theta, t)=M_{\wn}^T A^*(\theta-t\wn) M_{\wn}.
$$
To prove (\ref{E-1-1-0}), as in the case of Dirichlet condition \cite{Masmoudi-2012, Armstrong-Prange-2016}, we first note that
$W$ is a solution of the Neumann problem, 
\begin{equation}\label{E-1-2}
\left\{
\aligned
& -\left( \begin{array} {c}N_{n}^T \nabla_\theta \\ \partial_t  \end{array}\right)
\cdot B^*_n \left( \begin{array} {c}N_{n}^T \nabla_\theta \\ \partial_t  \end{array} \right) W
 =  \left( \begin{array} {c}N_{n}^T \nabla_\theta \\ \partial_t  \end{array} \right) G
 +H
 &\quad & \text{ in } \mathbb{T}^d \times \mathbb{R}_+,\\
&-e_{d+1} \cdot B_n^*\left( \begin{array} {c}N_{n}^T \nabla_\theta \\ \partial_t  \end{array} \right) W
 =h  + e_{d+1} \cdot G
 & \quad & \text{ on } \mathbb{T}^d \times \{ 0\},
\endaligned
\right.
\end{equation}
where $G=G_1 +G_2 +G_3$ , $H$, and $h$ are given by
\begin{equation}
\aligned
G_1 &= \big(B_n -B_{\wn} \big)
\left( \begin{array} {c}(N_{n}-N_{\wn})^T \nabla_\theta \\ 0 \end{array} \right)V_{\wn},\\
G_2&=
\big(B_n-B_{\wn} \big)
\left( \begin{array} {c}N_{\wn}^T \nabla_\theta \\ \partial_t \end{array} \right)V_{\wn},\\
G_3&=B_{\wn} 
\left( \begin{array} {c}(N_{n}-N_{\wn})^T \nabla_\theta \\ 0 \end{array} \right)V_{\wn},\\
H&=\left( \begin{array} {c}(N_{n}-N_{\wn})^T \nabla_\theta \\ 0  \end{array} \right)
B_{\wn} 
\left( \begin{array} {c}N_{\wn}^T \nabla_\theta \\ \partial_t \end{array} \right)V_{\wn},\\
h & =-\frac12 \big[ (n_i -\wn_i)e_\ell -(n_\ell -\wn_\ell)e_i\big] \cdot \nabla_\theta f_{i\ell}.
\endaligned
\end{equation}
Note that $|\partial_t^k \partial_\theta^\alpha (B_n -B_{\wn} )|\le
C (1+t)|n-\wn|$. This, together with Proposition  \ref{prop-2.5}, gives
\begin{equation}\label{G-1}
|\partial_t^k \partial_\theta^\alpha G_1(\theta, t)|
\le \frac{C( t +1) |n-\wn|^2}{\kappa (1+\kappa t)^\ell},
\end{equation}
\begin{equation}\label{G-2}
|\partial_t^k \partial_\theta^\alpha G_2(\theta, t)|
\le \frac{C(t+1) |n-\wn|}{(1+\kappa t)^\ell},
\end{equation}
\begin{equation}\label{G-3}
|\partial_t^k \partial_\theta^\alpha G_3(\theta, t)|
\le \frac{C |n-\wn|}{ (1+\kappa t)^\ell},
\end{equation}
\begin{equation}\label{H-1}
|\partial_t^k \partial_\theta^\alpha H (\theta, t)|
\le \frac{C |n-\wn|}{ (1+\kappa t)^\ell},
\end{equation}
for any $\alpha$, $k$, $\ell$,
where $C$ depends on $d$, $m$, $\alpha$, $k$, $\ell$ and $A$.

To deal with the growth factor $t+1$ in (\ref{G-1}), (\ref{G-2}) as well as the term $H$,
we rely  on the following weighted estimates.

\begin{lemma}\label{lemma-E1}
Suppose that $n\in \mathbb{S}^{n-1}$ satisfies the Diophantine condition (\ref{D-condition}).
Let $U$ be a smooth solution of
\begin{equation}\label{E-1-100}
\left\{
\aligned
& -\left( \begin{array} {c}N_{n}^T \nabla_\theta \\ \partial_t  \end{array}\right)
\cdot B^*_n \left( \begin{array} {c}N_{n}^T \nabla_\theta \\ \partial_t  \end{array} \right) U
 =  \left( \begin{array} {c}N_{n}^T \nabla_\theta \\ \partial_t  \end{array} \right) F
 &\quad & \text{ in } \mathbb{T}^d \times \mathbb{R}_+,\\
&-e_{d+1} \cdot B_n^*\left( \begin{array} {c}N_{n}^T \nabla_\theta \\ \partial_t  \end{array} \right) U
 =e_{d+1} \cdot F
 & \quad & \text{ on } \mathbb{T}^d \times \{ 0\}.
\endaligned
\right.
\end{equation}
Assume that
\begin{equation}\label{E-1-A}
(1+t)\| \nabla_{\theta, t} U(\cdot, t)\|_{L^\infty(\td)}
  +(1+t)\| F(\cdot, t)\|_{L^\infty(\td)} <\infty.
\end{equation}
Then, for any $-1<\alpha<0$,
\begin{equation}\label{E-1-101}
\int_0^\infty \int_{\td}
\Big\{ |N^T_n \nabla_\theta U|^2 +|\partial_t U|^2 \Big\}\,  t^\alpha \, d\theta dt
\le C_\alpha \int_0^\infty \int_{\td} |F|^2\,  t^\alpha\, d\theta dt,
\end{equation}
where $C_\alpha$ depends only on $d$, $m$, $\mu$, $\alpha$ as well as some H\"older norm of $A$.
\end{lemma}

\begin{proof}
We will reduce the weighted estimate (\ref{E-1-101})
to the analogous estimates we proved in Section 4 in a half-space.
Let
$$
u(x)=U(x-(x\cdot n)n , -x\cdot n) \quad \text{ and } \quad f(x)=F(x-(x\cdot n)n, -x\cdot n).
$$
Then $u$ is a solution of the Neumann problem,
$$
-\text{\rm div}(A(x)\nabla u)=\text{div}(f) \quad \text{ in } \mathbb{H}_n^d(0)
\quad \text{ and } \quad n\cdot A(x)\nabla u=-n\cdot f \quad \text{ on } \partial \mathbb{H}_n^d(0).
$$
It follows from the estimate (\ref{w-101}) that
\begin{equation}\label{E-1-102}
\aligned
& \frac{1}{R^{d-1}}\int_{\Omega_R} |N_n^T \nabla_\theta U (x-(x\cdot n) n, -x\cdot n) |^2 \, |x\cdot n|^\alpha \, dx\\
&\qquad +\frac{1}{R^{d-1}}\int_{\Omega_R} |\partial_t U (x-(x\cdot n) n, -x\cdot n) |^2 \, |x\cdot n|^\alpha \, dx\\
& \le  C R^{\alpha+1-d} \int_{\Omega_{2R}}|\nabla u|^2\, dx
+\frac{C}{R^{d-1}} \int_{\Omega_{2R} }|F (x-(x\cdot n) n, -x\cdot n) |^2 \, |x\cdot n|^\alpha \, dx,
\endaligned
\end{equation}
where 
$$
\Omega_R=\big\{ x\in \mathbb{H}_n^d (0): |x-(x\cdot n)n|\le R \text{ and } |x\cdot n|\le 2R \big\}.
$$

Next, we compute the limit for each term in (\ref{E-1-102}), as $R\to \infty$.
In view of (\ref{E-1-A}) it is clear that the first term in the RHS of (\ref{E-1-102})
goes to zero.
For the second term in the RHS of (\ref{E-1-102}), we write it as
\begin{equation}\label{E-1-1020}
 C\int_0^{2R} t^\alpha \left\{ \frac{1}{R^{d-1}}\int_{\substack{ x\cdot n=t\\ |x+tn| <R}}
|F(x+tn, t)|^2 \, d\sigma (x) \right\} dt.
\end{equation}
Since $F(\theta, t)$ is 1-periodic in $\theta$ and $n$ satisfies the Diophantine condition,
\begin{equation}\label{E-1-103}
\frac{1}{R^{d-1}}\int_{\substack{ x\cdot n=t\\ |x+tn| <R}}
|F(x+tn, t)|^2 \, d\sigma (x) 
\to C_d \int_{\td} |F(\theta, t)|^2\, d\theta
\end{equation}
for each $t>0$, as $R\to \infty$.
With the assumption (\ref{E-1-A}) at our disposal, we apply the Dominated Convergence Theorem
to deduce that  the last integral in (\ref{E-1-102}) converges to 
$$
C_d \int_0^\infty \int_{\td} |F(\theta, t)|^2\, t^\alpha\, d\theta dt.
$$
A similar argument also shows that the LHS of (\ref{E-1-102}) converges to
$$
C_d \int_0^\infty \int_{\td} 
\Big\{ |N_n^T \nabla_\theta U|^2 +|\partial_t U|^2 \Big\}\, t^\alpha\, d\theta dt.
$$
As a result, we have proved the estimate (\ref{E-1-101}).
\end{proof}

\begin{remark}\label{D-remark}
{\rm 
The same argument as in the proof of Lemma \ref{lemma-E1} also gives a weighted estimate for 
Dirichlet problem.
More precisely, 
suppose that $n\in \mathbb{S}^{n-1}$ satisfies the Diophantine condition (\ref{D-condition}).
Let $U$ be a smooth solution of
\begin{equation}\label{E-1-100-D}
\left\{
\aligned
 -\left( \begin{array} {c}N_{n}^T \nabla_\theta \\ \partial_t  \end{array}\right)
\cdot B^*_n \left( \begin{array} {c}N_{n}^T \nabla_\theta \\ \partial_t  \end{array} \right) U
 & =  \left( \begin{array} {c}N_{n}^T \nabla_\theta \\ \partial_t  \end{array} \right) F +H
 &\quad & \text{ in } \mathbb{T}^d \times \mathbb{R}_+,\\
U
& =0
 & \quad & \text{ on } \mathbb{T}^d \times \{ 0\}.
\endaligned
\right.
\end{equation}
Assume that
$$
(1+t)\| \nabla_{\theta, t} U(\cdot, t)\|_{L^\infty(td)}
  + (1+t)\|  F(\cdot, t)\|_{L^\infty(\td)}
   + (1+t)^2 \| H(\cdot, t)\|_{L^\infty(td)} <\infty.
$$
Then, for any $-1<\alpha<0$,
\begin{equation}\label{E-1-1010}
\int_0^\infty \int_{\td}
\Big\{ |N^T_n \nabla_\theta U|^2 +|\partial_t U|^2 \Big\}\,  t^\alpha \, d\theta dt
\le C_\alpha \int_0^\infty \int_{\td} \Big\{ |F|^2 +|H|^2 t^2 \Big\}\,  t^\alpha\, d\theta dt,
\end{equation}
where $C_\alpha$ depends only on $d$, $m$, $\mu$, $\alpha$ as well as some H\"older norm of $A$.
}
\end{remark}

We are now in a position to give the proof of Theorem \ref{theorem-E}.

\begin{proof}[\bf Proof of Theorem \ref{theorem-E}]
Recall that it suffices to prove (\ref{E-1-1-0}) with $W$ given by (\ref{W}).
To do this we split $W$ as
$W=W_1 +W_2 +W_3$,
where $W_1$ is the solution of (\ref{E-1-2}) with $G=0$ and $H=0$,
$W_2$ the solution with $H=0$ and $h=0$, and
$W_3$ the solution with $G=0$ and $h=0$.
To estimate $W_1$, we use the energy estimate in Proposition \ref{prop-2.1}.
This gives
\begin{equation}\label{E-10-1}
 \int_0^1 \int_{\td} \Big\{ |N_n^T \nabla_\theta W_1|^2 +|\partial_t W_1|^2 \Big\} d\theta dt\\
\le C\, \| h\|^2_{H^1(\td)}
 \le C |n-\wn|^2.
\end{equation}

To handle the function $W_2$, we use the weighted estimate in Lemma \ref{lemma-E1}
with $\alpha =\sigma -1$.
This, together with estimates (\ref{G-1})-(\ref{H-1}), leads to
\begin{equation}\label{E-10-2}
\aligned
& \int_0^1 \int_{\td} \Big\{ |N_n^T \nabla_\theta W_1|^2 +|\partial_t W_1|^2 \Big\} d\theta dt\\
&\le C \int_0^\infty\int_{\td}  |G|^2 t^{\sigma-1}\, d\theta dt\\
&\le C \int_0^\infty 
\left\{ \frac{(t+1)^2 |n-\wn|^4}{\kappa^2 (1+\kappa t)^4}
+\frac{(t+1)^2 |n-\wn|^2}{(1+\kappa t)^4}
+\frac{|n-\wn|^2}{(1+\kappa t)^4} \right\} t^{\sigma-1}\, dt\\
&\le  C\left\{ \frac{ |n-\wn|^4}{\kappa^{4+\sigma}}
+\frac{|n-\wn|^2}{\kappa^{2+\sigma}} \right\}.
\endaligned
\end{equation}

Finally, we note that by writing $H(\theta, t)=\partial_t \widetilde{H}(\theta, t)$, where
$$
\widetilde{H}(\theta, t)=- \int_t^\infty H(\theta, s)ds,
$$
we may reduce the estimate of $W_3$ to the previous two cases.
Indeed, we split $W_3$ as $W_{31} +W_{32}$, where
$W_{31}$ is a solution of (\ref{E-1-2}) with $G=0$, $H=0$,  $h=e_d \cdot \widetilde{H}(\theta, 0)$,
and $W_{32}$ a solution of (\ref{E-1-2}) with $G=(0, \widetilde{H})$, $H=0$, and $h=0$.
Observe that by (\ref{H-1}),
$$
\| \widetilde{H}(\cdot, 0)\|_{H^1(\td)}^2 \le \frac{C |n-\wn|^2}{\kappa^2}
$$
and
$$ 
|\partial_t^k \partial_\theta^\alpha \widetilde{H}(\theta, t)|\le \frac{C |n -\wn|}{\kappa (1+\kappa t)^\ell},
$$
for any $\alpha$, $k$ and $\ell$.
As in the cases of $W_1$ and $W_2$, by Proposition \ref{prop-2.1} and Lemma \ref{lemma-E1}, we obtain
$$ 
 \int_0^1 \int_{\td} \Big\{ |N_n^T \nabla_\theta W_2|^2 +|\partial_t W_2|^2 \Big\} \, d\theta dt
 \le \frac{C |n-\wn|^2}{\kappa^{2+\sigma}}.
 $$
Consequently,  we have proved that
\begin{equation}\label{E-10-10}
 \int_0^1 \int_{\td} \Big\{ |N_n^T \nabla_\theta W|^2 +|\partial_t W|^2 \Big\}\,  d\theta dt
 \le  C\left\{ \frac{ |n-\wn|^4}{\kappa^{4+\sigma}}
+\frac{|n-\wn|^2}{\kappa^{2+\sigma}} \right\}.
\end{equation}

Finally, we note that by differentiating the system (\ref{E-1-1}),
the function $\nabla_\theta W$ is a smooth solution to a boundary value problem of same type
as $W$.
As a result, we may replace $W$ in (\ref{E-10-10}) by $\nabla_\theta W$.
This gives the desired estimate for $\nabla_\theta \partial_t W$
and completes the proof of (\ref{E-1-1-0}).
\end{proof}

%%%%%%%%%%%%%%%%%%%%%%%%%%%%%%%%%%%%%%%%%%%%%

%%%%%%%%%%%%%%%%%%%%%%%%%%%%%%%%%%%%%%%%%%%%

\section{A partition of unity}
\setcounter{equation}{0}

For $x\in \partial\Omega$, let
\begin{equation}\label{k-x}
\aligned
\kappa (x)=\sup
\big\{ \kappa\in [0,1]: \text{the Diophantine condition (\ref{D-condition}) holds}\\
\text{ for $n(x)$ with constant $\kappa$} \big\}.
\endaligned
\end{equation}
It is known that if $\Omega$ is a bounded smooth, strictly convex domain in $\rd$,
then $\kappa(x)>0$ for a.e. $x\in \partial\Omega$. Moreover,
$1/\kappa(x)$ belongs to the space $ L^{d-1, \infty}(\partial\Omega)$ \cite{Masmoudi-2012}.
This means that there exists $C>0$ such that
\begin{equation}\label{weak-L}
|\big\{ x\in \partial\Omega: \big[ \kappa (x)\big]^{-1} >\lambda \big\}| \le C\, \lambda^{1-d} 
\quad \text{ for any } \lambda>0.
\end{equation}

\begin{prop}\label{prop-7.0}
Let $0<q<d-1$. Then for any $x\in \partial\Omega$ and $0<r<\text{\rm diam}(\Omega)$,
\begin{equation}\label{kappa-7}
\left(\average_{B(x, r)\cap\partial\Omega} \big( \kappa (y)\big)^{-q}\, d\sigma (y) \right)^{1/q}
\le \frac{C}{r},
\end{equation}
where $C$ depends only on $d$, $q$ and $\Omega$.
\end{prop}

\begin{proof}
Note that
$$
\aligned
\int_{B(x, r)\cap\partial\Omega} 
\big( \kappa (y)\big)^{-q}\, d\sigma (y)
&=q \int_0^\infty \lambda^{q-1} |\big\{ x\in B(x, r)\cap\partial\Omega: \big[ \kappa (x)\big]^{-1} >\lambda \big\}|\, d\lambda\\
& \le C \int_0^\Lambda \lambda^{q-1} \cdot r^{d-1}\, d\lambda
+C \int_\Lambda^\infty
\lambda^{q-d}\, d\sigma\\
&\le C r^{d-1}\cdot \Lambda^q + C \Lambda^{q-d +1},
\endaligned
$$
where we have used (\ref{weak-L}) for the first inequality.
The proof is finished by optimizing $\Lambda$ with $\Lambda=r^{-1}$.
\end{proof}

In this section we construct a partition of unity for $\partial\Omega$, which is adapted to the function
$\kappa(x)$. We mention that a similar partition of unity, which plays an important role in the analysis of 
the oscillating Dirichlet problem,
was given in \cite{Armstrong-Prange-2016}. Here we provide a more direct $L^p$-based approach.

We first describe such construction in the flat space.

\begin{lemma}\label{lemma-7.1}
Let $Q_0$ be a cube in $\mathbb{R}^{d-1}$
and $F\in L^p(12Q_0)$ for some $p>d-1$. Let
 $\tau>0$ be a number such that
\begin{equation}\label{7.1-0}
\left(\int_{12Q_0}|F|^p\right)^{1/p}> \frac{\tau}{\big[\ell(Q_0)\big]^{1-\frac{d-1}{p}}},
\end{equation}
where $\ell(Q_0)$ denotes the side length of $Q_0$.
Then there exists a finite sequence $\{ Q_j\}$ of dyadic sub-cubes of $Q_0$ such that
the interiors of $Q_j$'s are mutually disjoint,
\begin{equation}\label{7.1-1}
Q_0=\cup Q_j,
\end{equation}
\begin{equation}\label{7.1-2}
\left(\int_{12Q_j}|F|^p\right)^{1/p} \le \frac{\tau}{\big[\ell(Q_j)\big]^{1-\frac{d-1}{p}}},
\end{equation}
\begin{equation}\label{7.1-3}
\left(\int_{12Q^+_j}|F|^p\right)^{1/p} > \frac{\tau}{\big[\ell(Q^+_j)\big]^{1-\frac{d-1}{p}}},
\end{equation}
where $Q_j^+$ denotes the dyadic parent of $Q_j$, i.e., $Q_j$ is obtained by bisecting $Q_j^+$ once. Moreover,
if $4Q_j\cap 4Q_k \neq \emptyset$, then
\begin{equation}\label{7.1-4}
(1/2) \ell (Q_k)\le \ell (Q_j) \le 2 \ell (Q_k).
\end{equation}
\end{lemma}

\begin{proof}
The lemma is proved by using  a stoping time argument (Calder\'on-Zygmund decomposition).
We begin by bisecting the sides of $Q_0$ and obtaining $2^{d-1}$ dyadic sub-cubes $\{ Q^\prime\}$.
If a cube $Q^\prime$ satisfies 
\begin{equation}\label{7.1-5}
\left(\int_{Q^\prime} |F|^p\right)^{1/p} 
\le \frac{\tau}{\big[\ell(Q^\prime)\big]^{1-\frac{d-1}{p}}}
\end{equation}
we stop and collect the cube. Otherwise, we repeat the same procedure on $Q^\prime$.
 Since the RHS of (\ref{7.1-5}) goes to $\infty$ as $\ell(Q^\prime)\to 0$, the procedure is stopped 
 in a finite time. As a result, we obtain a finite number of sub-cubes with mutually disjoint interiors 
 satisfying (\ref{7.1-1})-(\ref{7.1-3}).
 We point out this decomposition was performed in the whole space in \cite{Shen-1998}
 in the study of negative eigenvalues for the Pauli operator.
The inequalities in (\ref{7.1-4}) were proved in \cite{Shen-1998}
by adapting an argument found in \cite{Fefferman-1996}.
The same argument works equally well in the setting of a finite cube $Q_0$.
We omit the details.
\end{proof}

Fix $x_0\in \partial\Omega$.
Let $c_0>0$ be sufficiently small so that
$B(x_0, 10c_0\sqrt{d})\cap\partial\Omega$ is given by the graph of a smooth function in a coordinate system,
obtained from the standard system through rotation and translation.
Let $\mathbb{H}_n^d(a)$ denote the tangent plane for $\partial\Omega$ at $x_0$, where $n=n(x_0)$ and
$a=x_0\cdot n$.
For $x\in B(x_0, 5c_0\sqrt{d})\cap \partial\Omega$, let
\begin{equation}\label{proj}
P(x; x_0)= x-x_0 -((x-x_0)\cdot n)n
\end{equation}
denote its projection to $\mathbb{H}_n^d(a)$.
The projection $P$ is one-to-one in $B(x_0, 10c_0\sqrt{d})\cap\partial\Omega$.
To construct a partition of unity for $B(x_0, c_0)\cap\partial\Omega$, adapted to the function $\kappa(x)$,
we use the inverse map $P^{-1}$ to lift a partition on the tangent plane, given in Lemma \ref{lemma-7.1},
to $\partial\Omega$.
 More precisely,
fix a cube $Q_0$ on the tangent plane $H_n^d(a)$ such that 
$$
B(x_0, 5c_0\sqrt{d})\cap \partial\Omega \subset P^{-1}(Q_0) \subset B(x_, 10c_0\sqrt{d})\cap \partial\Omega.
$$
We apply Lemma \ref{lemma-7.1} to $Q_0$ with the bounded function $F(x)=\kappa(P^{-1}(x))$ and
some $p>d-1$.
For each $0<\tau<c_1$,
this generates a finite sequence of sub-cubes $\{Q_j\}$ with the 
properties (\ref{7.1-1})-(\ref{7.1-4}).

Let $x_j$ denote the center of $Q_j$ and $r_j$ the side length.
Let $\widetilde{x}_j=P^{-1}(x_j)$ and $\widetilde{Q}_j =P^{-1}(Q_j)$. 
We will use the notation $t\widetilde{Q}_j =P^{-1} (t Q_j)$ for $t>0$
and call $\widetilde{Q}_j$ a cube on $\partial\Omega$.
Then
$$
\widetilde{Q}_0=P^{-1}(Q_0)=\cup_{j= 1}^N  \widetilde{Q}_j.
$$
For each $\widetilde{Q}_j$ with $j\ge 1$, we choose $\eta_j\in C_0^\infty(\rd)$ such that
$0\le \eta_j \le 1$, $\eta_j=1$ on $\widetilde{Q}_j$,
$\eta_j =0$ on $\partial\Omega \setminus 2\widetilde{Q}_j$,
and $|\nabla\eta_j|\le C r_j^{-1}$. Note that by the property (\ref{7.1-4}),
$1\le \sum_j \eta_j \le C_0$ on $\widetilde{Q}_0$, where $C_0$ is a constant depending only
on $d$ and $\Omega$.
Let
$$
\varphi_j (x) =\frac{\eta_j(x)}{\sum_{k=1}^N \eta_k (x)}.
$$
Then 
$$
\sum_{j=1}^N\varphi_j (x)=1\quad  \text{ for any } x\in \widetilde{Q}_0.
$$
Observe that $0\le \varphi_j\le 1$, $\varphi_j \ge C_0^{-1}$ on $\widetilde{Q}_j$,
$\varphi_j=0$ on $\partial\Omega\setminus 2\widetilde{Q}_j$, and
$|\nabla \varphi_j|\le Cr_j^{-1}$.
Furthermore, by the property (\ref{7.1-5}),
there are positive constants $c_1$ and $c_2$,
depending only on $d$, $p$ and $\Omega$, such that
\begin{equation}\label{7.2-10}
\left(\average_{12\widetilde{Q}_j}(\kappa(x))^p d\sigma (x)\right)^{1/p} \le 
\frac{c_1\tau}{r_j},
\end{equation}
\begin{equation}\label{7.2-11}
\left(\average_{36\widetilde{Q}_j} (\kappa (x))^p\, d\sigma(x)  \right)^{1/p} \ge
\frac{c_2\tau}{r_j}.
\end{equation}
In particular, since $\kappa\le 1$,
it follows from (\ref{7.2-11}) that 
\begin{equation}
r_j \ge c_2\, \tau.
\end{equation}
Also,  by (\ref{7.2-11}), there exists some $z_j\in 36\widetilde{Q}_j$
such that
\begin{equation}\label{z-j}
\kappa (z_j) \ge \frac{c_2 \tau}{r_j}.
\end{equation}

\begin{prop}\label{prop-7.2}
There exists  a constant $C$, depending only on $d$, $p$ and $\Omega$, such that
\begin{equation}\label{7.2-0}
r_j\le C \sqrt{\tau}.
\end{equation}
\end{prop}

\begin{proof}
 By H\"older's inequality, 
\begin{equation}\label{7.2-20}
\aligned
1 &\le \left(\average_{12 \widetilde{Q}_j } \kappa^{-p^\prime}\, d\sigma \right)^{1/p^\prime}
\left(\average_{12 \widetilde{Q}_j} \kappa^{p}\, d\sigma \right)^{1/p}\\
&\le C\,  r_j^{-1}
\left(\average_{12 \widetilde{Q}_j } \kappa^{p}\, d\sigma \right)^{1/p},
\endaligned
\end{equation}
where we have used Proposition \ref{prop-7.0} for the last step.
Note that  the condition $p>d-1$ is equivalent to $p^\prime<\frac{d-1}{d-2}$, 
which is less or equal to ${d-1}$ if $d\ge 3$.
In view of (\ref{7.2-10}) and (\ref{7.2-20}) we obtain (\ref{7.2-0}).
\end{proof}

\begin{prop}\label{prop-7.3}
Let $0\le \alpha<d-1$.
Then
\begin{equation}\label{7.3-0}
\sum_j r_j^{\alpha +d-1}
\le C_\alpha\,  \tau^\alpha,
\end{equation}
where  $C_\alpha$ depends only on $d$, $p$, $\alpha$ and $\Omega$.
\end{prop}

\begin{proof}
It follows from the first inequality in (\ref{7.2-20}) and (\ref{7.2-10}) that
\begin{equation}\label{7.3-1}
r_j \le C \tau \left(\average_{12\widetilde{Q}_j} 
\kappa^{-p^\prime} d\sigma \right)^{1/p^\prime}.
\end{equation}
Let $\mathcal{M}_{\partial\Omega}(f)$ denote the Hardy-Littlewood maximal function of $f$ on $\partial\Omega$,
 defined by
\begin{equation}\label{HL}
\mathcal{M}(f)(x)
=\sup\left\{ \average_{B(x, r)\cap\partial\Omega} |f|\, d\sigma:  0<r<\text{diam}(\Omega) \right\}
\end{equation}
for $x\in \partial\Omega$.
By (\ref{7.3-1}) we obtain 
\begin{equation}\label{7.3-2}
r_j^\alpha \le C \tau^\alpha \left(\inf_{x\in 12\widetilde{Q}_j}
\mathcal{M}_{\partial\Omega} ( \kappa^{-p^\prime})(x) \right)^{\alpha/p^\prime}
\end{equation}
Hence,
$$
\aligned
\sum_j r_j^{\alpha +d-1}
&\le C \tau^\alpha
\sum_j \int_{\widetilde{Q}_j} 
\left[ \mathcal{M}_{\partial\Omega} ( \kappa^{-p^\prime}) \right]^{\alpha/p^\prime}\\
&\le C \tau^\alpha\int_{\partial\Omega}
\left[ \mathcal{M}_{\partial\Omega} ( \kappa^{-p^\prime}) \right]^{\alpha/p^\prime}.
\endaligned
$$
Finally,
recall that $p^\prime<\frac{d-1}{d-2}$ and $0\le \alpha<d-1$.
Choose $t>1$ so that $p^\prime<t\alpha<d-1$.
Then
$$
\aligned
\int_{\partial\Omega}
\left[ \mathcal{M}_{\partial\Omega} ( \kappa^{-p^\prime}) \right]^{\alpha/p^\prime} d\sigma
 &\le C \left(\int_{\partial\Omega} 
\left[ \mathcal{M}_{\partial\Omega} ( \kappa^{-p^\prime}) \right]^{\alpha t/p^\prime} d\sigma \right)^{1/t}\\
& \le C\left(\int_{\partial\Omega}
(\kappa^{-1})^{\alpha t} d\sigma \right)^{1/t}
<\infty,
\endaligned
$$
where we have used the fact that the operator $\mathcal{M}_{\partial\Omega}$
is bounded on $L^q(\partial\Omega)$ for $q>1$.
This completes the proof.
\end{proof}

\begin{theorem}\label{h-data-theorem}
Let $ \overline{g}=(\overline{g}_k^\beta)$, where $\overline{g}_k^\beta$ is defined by 
(\ref{h-data-ij}).
Then  $\overline{g} \in W^{1, q}(\partial\Omega)$ for any $1<q< d-1$ and
\begin{equation}\label{h-data-1}
\| \overline{g}\|_{W^{1, q}(\partial\Omega)}
\le C_q\, \sup_{y\in \td} \| g(\cdot, y)\|_{C^1(\partial\Omega)},
\end{equation}
where $C_q$ depends only on $d$, $m$, $\mu$, $q$, and
$\|A\|_{C^k(\td)}$ for some $k=k(d)\ge 1$.
\end{theorem}

\begin{proof}
With Theorem \ref{theorem-E} at our disposal, the line of argument is similar to that in \cite{Armstrong-Prange-2016}.
Without loss of generality we assume that 
$\sup_{y\in \td}\| g(\cdot, y)\|_{C^1(\partial\Omega)}=1$ and that for any $y\in \rd$,
supp$(g(\cdot, y))\subset B(x_0, c_0)\cap \partial\Omega$
for some $x_0\in \partial\Omega$ and $c_0>0$. Fix $\tau\in (0,1/2)$.
Let 
$$
\overline{g}_\tau (x)=\sum_j \overline{g}(z_j) \varphi_j (x),
$$
where $\{ \varphi_j\}$ is a finite sequence of smooth functions constructed in the partition of unity for 
$B(x_0, c_0)\cap\partial\Omega$, associated with $\kappa (x)$ at the level $\tau$,
and $z_j$ is given by (\ref{z-j}).
Note that for $x\in \widetilde{Q}_\ell$,
$$
\nabla \overline{g}_\tau (x)=\sum_j \overline{g}(z_j) \nabla \varphi_j (x)
 =\sum_j \big( \overline{g} (z_j) -\overline{g}(y_\ell)\big) \nabla \varphi_j (x),
 $$
 where $y_\ell $ is a point in $\widetilde{Q}_\ell$ that satisfying the condition (\ref{D-condition}).
In view of  Theorem \ref{theorem-E} it  follows that if $x\in \widetilde{Q}_\ell$,
$$
\aligned
|\nabla \overline{g}_\tau (x)|
& \le \sum_j |\overline{g} (z_j) -\overline{g}(y_\ell)|\, | \nabla \varphi_j (x)|\\
& \le C \left(\frac{r_\ell}{\tau}\right)^{1+\sigma} \left\{ \frac{r_\ell^2}{\tau} +1\right\}\\
&\le C \left(\frac{r_\ell}{\tau}\right)^{1+\sigma},
\endaligned
$$
where we have also used the estimates $|\nabla \varphi_j|\le Cr_j^{-1}$,
$\kappa(z_j)\ge c\, \tau r_j^{-1}$, and $r_\ell \le C \sqrt{\tau}$.
Hence, if we choose $\sigma\in (0,1)$ so small that $(1+\sigma)q<d-1$,
\begin{equation}
\int_{B(x_0, c_0)\cap\partial\Omega} |\nabla \overline{g}_\tau|^q \, d\sigma
\le C \sum_\ell \left(\frac{r_\ell}{\tau}\right)^{(1+\sigma)q} r_\ell ^{d-1}
\le C,
\end{equation}
where we have used (\ref{7.3-0}) for the last step.
This shows that the set
$\{ \overline{g}_\tau: 0<\tau<1\} $ is uniformly bounded in $W^{1, q}(B(x_0, c_0)\cap\partial\Omega)$
for any $q<d-1$.

Finally, we observe that a similar argument also gives
$$
\aligned
\int_{B(x_0, c_0)\partial\Omega}
|\overline{g}_\tau (x) -\overline{g}(x)|^q\, d\sigma 
& \le 
\sum_\ell \left(\frac{r_\ell^{2+\sigma}}{\tau^{1+\sigma}}  \right)^q r_\ell^{d-1}\\
&\le C\tau^{\frac{q}{2}} \sum_\ell \left(\frac{r_\ell}{\tau}\right)^{(1+\sigma)q} r_\ell^{d-1}\\
&\le C\tau^{\frac{q}{2}}.
\endaligned
$$
Hence, $\| \overline{g}_\tau -\overline{g}\|_{L^q(B(x_0, c_0)\cap\partial\Omega)} \to 0$ as $\tau\to 0$,
for any  $1<q<d-1$.
As a result, we have proved that $\overline{g}\in W^{1, q}(\partial\Omega)$ for any $1<q<d-1$.
\end{proof}

\begin{proof}[\bf Proof of Theorem \ref{main-theorem-2}]
With the weighted norm inequality (\ref{E-1-1010}) for Dirichlet problem (\ref{E-1-100-D}) at our disposal,
the estimate in Theorem \ref{theorem-E} continues to hold 
for the homogenized Dirichlet boundary data for $d\ge 2$.
The rest of argument for (\ref{Sobolev-D}) and (\ref{rate-DP-low})
is exactly the same as in \cite{Armstrong-Prange-2016}.
We omit the details.
\end{proof}

%%%%%%%%%%%%%%%%%%%%%%%%%%%%%%%%%%%%%%%

%%%%%%%%%%%%%%%%%%%%%%%%%%%%%%%%%%%%%%

\section{Proof of Theorem \ref{main-theorem-1}}
\setcounter{equation}{0}

With the estimates in Sections 5, 6 and 7,
the line of argument is similar to that in \cite{Armstrong-Prange-2016}
for the oscillating Dirichlet problem.
Recall that
$$
 v^\gamma_\e (x)=  -\int_{\partial\Omega}
 \frac{\partial}{\partial y_k} \big\{ N_0^{\gamma \beta} (x, y)\big\}\cdot
\big(T_{ij}(y)\cdot \nabla_y \big)\Psi_{\e, k}^{*\alpha\beta} (y)
\cdot  
 g_{ij}^\alpha (y, y/\e)\, d\sigma (y)
$$
and
\begin{equation}
 v_0^\gamma (x)=-\int_{\partial\Omega}
\frac{\partial}{\partial y_k} \big\{ N_0 ^{\gamma \beta} (x, y)\big\}
 \widetilde{g}_{k}^\beta (y)\, d\sigma (y), 
\end{equation}
where the function $\widetilde{g}_k^\beta$ is given by (\ref{mean}).
We will show that for any $\sigma \in (0,1)$,
\begin{equation}\label{final-goal}
\int_\Omega |v_\e -v_0|^2\, dx \le C_\sigma \, \e^{1-\sigma}.
\end{equation}
This would imply that if $u_\e$ and $u_\e$ are solutions of (\ref{NP-0})
and (\ref{NP-h}) respectively, then there exists some constant $E$ such that
$$
\| u_\e -u_0-E\|_{L^2(\Omega)} \le C_\sigma \, \e^{\frac12 -\sigma}.
$$
It then follows that
$$
\| u_\e - u_0 -\average_\Omega (u_\e-u_0)\|_{L^2(\Omega)}
\le C_\sigma \, \e^{\frac12 -\sigma},
$$
which gives (\ref{rate-0}) in the case $\int_\Omega u_\e  =\int_\Omega u_0=0$.
Recall that the estimate (\ref{Sobolev}) is already proved in Section 7 (see Theorem \ref{h-data-theorem}).

To prove (\ref{final-goal})
 we first note that by using a partition of unity for $\partial\Omega$,
without loss of generality, we may assume that there exists some $x_0\in \partial\Omega$
such that  for any $y\in \td$,
supp$(g(\cdot, y))\subset B(x_0, c_0)$, where $c_0>0$ is sufficiently small so that
$B(x_0, 10c_0\sqrt{d})\cap\partial\Omega$ is given by the graph of a smooth function in a coordinate system,
obtained from the standard system by rotation and translation.
We construct another partition of unity for $B(x_0, 5c_0\sqrt{d})\cap \partial\Omega$, as described in Section 7,
with 
\begin{equation}\label{tau}
\tau = \e^{1-\sigma},
\end{equation}
adapted to the function $\kappa(x)$.
Thus there exist a finite sequence $\{ \varphi_j\}$ of $C_0^\infty$ functions in $\rd$
and a finite sequence $\{\widetilde{Q}_j \}$ of "cubes" on $\partial\Omega$,
such that $\sum_j \varphi_j =1$ on $B(x_0, 5c_0\sqrt{d})\cap\partial\Omega$.

Next, observe that by the estimate $|\nabla_y N_0(x, y)|+|\nabla_y N_\e (x, y)| \le C |x-y|^{1-d}$,
$$
|v_\e(x)| +|v_0(x)| \le C \big\{ 1+ |\ln \delta(x) |\big\},
$$
where $\delta(x)=\text{dist}(x, \partial\Omega)$.
This implies that
\begin{equation}\label{D-out}
\aligned
\sum_j \int_{B(\widetilde{x}_j, C r_j)\cap \Omega}
|v_\e -v_0|^2\, dx
&\le C \sum_j \int_{B(\widetilde{x}_j, C r_j)\cap\Omega} 
\big( 1+ |\ln \delta(x) |\big)^2\, dx\\
& \le C \sum_j r_j^d (1+|\ln r_j|)^2\\
&\le C \e^{1-\sigma} (1+|\ln \e|)^2,
\endaligned
\end{equation}
where we have used Propositions \ref{prop-7.2} and \ref{prop-7.3}
(see Section 7 for the definitions of $\widetilde{x}_j$ and $r_j$).
Also note that
\begin{equation}\label{volume}
| \cup_j B(\widetilde{x}_j, Cr_j)|
  \le C \sum_j r_j^{d}
\le C \tau,
\end{equation}
where we have used Proposition \ref{prop-7.3}.
To estimate the $L^2$ norm of $v_\e -v_0$ on the set
\begin{equation}\label{D-8}
D= D_\e=\Omega \setminus \cup_j B(\widetilde{x}_j , C r_j),
\end{equation}
 we introduce a function 
 \begin{equation}\label{Theta}
 \Theta_t (x)
 =\sum_j \frac{ r_j^{d-1+t}}{|x-\widetilde{x}_j |^{d-1}},
 \end{equation}
  where $0\le t<d-1$.
  
 \begin{lemma}\label{lemma-8.1}
 Let $\Theta_t (x)$ be defined by (\ref{Theta}).
 Then, if $ q\ge 1$ and $0\le q t<d-1$,
 \begin{equation}\label{8.1-00}
\int_D  \left( \Theta_t (x)\right)^q \, dx \le C \, \tau^{q t}.
\end{equation}
 \end{lemma}
 
 \begin{proof}
 Observe that if
  $x\notin B(\widetilde{x}_j , C r_j)$, then
 $$
 \frac{r_j^{d-1}}{|x-\widetilde{x}_j|^{d-1}}
 \le C \int_{\widetilde{Q}_j} \frac{d\sigma (y)}{|x-y|^{d-1}}.
 $$
 Hence, for $x\in D$,
 $$
 \aligned
 \Theta_t (x)
 &\le C \int_{\partial\Omega} \frac{ f_t (y)}{|x-y|^{d-1}} \, d\sigma (y)\\
 &\le C \left(\int_{\partial\Omega} \frac{|f_t(y)|^q}{|x-y|^{d-1}}\, d\sigma (y) \right)^{1/q}
 ( 1+|\ln \delta(x)|)^{1/q^\prime},
 \endaligned
 $$
 where $f_t(y)=\sum_j r_j^t \varphi_j (y)$, $\delta (x)=\text{dist}(x, \partial\Omega)$, and
 we have used H\"older's inequality for the last step.
 It follows that
 $$
 \aligned
 \int_D |\Theta_t (x)|^q \, dx
 &\le C \int_{\partial\Omega} |f_t (y)|^q \, d\sigma\\
 & \le C \sum_j r_j^{qt} r_j^{d-1}\\
 &\le C \tau^{q t},
 \endaligned
 $$
 where have used Proposition \ref{prop-7.3}.
  \end{proof}
 
As in the case of  Dirichlet problem in \cite{Armstrong-Prange-2016}, we split $v_\e -v_0$ into several 
parts,
\begin{equation}
\aligned
& -\big(v^\gamma_\e (x) -v^\gamma_0(x)\big) \\
& =
\sum_j \int_{\partial\Omega}
\partial_{y_k} N_0^{\gamma \beta} (x, y) 
\left\{ \big(T_{i\ell}(y)\cdot \nabla_y \big)\Psi_{\e, k}^{*\alpha\beta} (y)
\cdot   g_{i\ell}^\alpha (y, y/\e)
-\widetilde{g}^\beta_k (y) \right\}  \varphi_j (y) \, d\sigma  (y)\\
& =
I_1+I_2+I_3+I_4+I_5,
\endaligned
\end{equation}
where $I_1, I_2, \dots, I_5$ are defined below and  handled separately.
We will show that for $k=1,2, \dots , 5$,
\begin{equation}\label{estimate-I}
\int_D | I_k(x)|^2\, dx \le C_\sigma\,  \e^{1-4\sigma},
\end{equation}
which, together with (\ref{D-out}), gives (\ref{final-goal}), as $\sigma\in (0,1/4)$ is arbitrary.
\bigskip

\noindent{\bf Estimate of $I_1$}, where

\begin{equation}\label{I-1}
\aligned
 &I_1=\sum_j  \int_{\partial\Omega}
\partial_{y_k} N_0^{\gamma \beta} (x, y)
 \big(T_{i\ell}(y)\cdot \nabla_y \big)\big( \Psi_{\e, k}^{*\alpha\beta} 
-\Phi_{\e, k}^{* \alpha\beta, z_j} \big) \cdot g_{i\ell}^\alpha (y, y/\e)
  \varphi_j (y) \, d\sigma  (y),\\
  &\Phi_{\e, k}^{*\alpha\beta, z_j} (y)
  =y_k \delta^{\alpha\beta}+\e \chi_{k}^{*\alpha\beta}(y/\e) +\phi_{\e, k}^{*\alpha\beta, z_j}(y),
  \endaligned
 \end{equation}
and $z_j$ is given in (\ref{z-j}).
Here we use Theorem \ref{main-lemma} to obtain that for any $\rho\in (0,1/2)$,
\begin{equation}\label{I-10}
\Big|\nabla \Big(\Psi_{\e, k}^{*\alpha\beta} - \Phi_{\e, k}^{*\alpha\beta, z_j}\Big)\Big|
\le C \sqrt{\e} \big\{ 1+ |\ln \e|\big\}
+C \e^{-1-\rho} r_j^{2+\rho},
\end{equation}
for $y\in 2\widetilde{Q}_j$.
It follows from (\ref{I-10}) that for $x\in D$,
\begin{equation}\label{I-11}
|I_1 (x)|
\le C \sqrt{\e} (1+|\ln \e|)\sum_j \frac{r_j^{d-1}}{|x-\widetilde{x}_j|^{d-1}}
+ C\e^{-1-\rho}
\sum_j \frac{r_j^{2+\rho +d-1}}{|x-\widetilde{x}_j|^{d-1}}.
\end{equation}
We now use Lemma \ref{lemma-8.1} to estimate the $L^2$ norm of $I_1$ on $D$.
The first term in the RHS of (\ref{I-11})
is harmless. 
For the second term we use the fact $r_j \le C \sqrt{\tau}$ to bound it by
$$
C \e^{-1-\rho} \tau^{\frac12 +\rho} \sum_j \frac{r_j^{1-\rho +d-1}}{|x-\widetilde{x}_j|^{d-1}}.
$$
Since $2(1-\rho)<d-1$ for $d\ge 3$, we obtain 
$$
\aligned
\int_D |I_1(x)|^2\, dx  & \le C \e (1+|\ln \e|)^2 +C \e^{-2-2\rho} \tau^{1+2\rho}
\tau^{2(1-\rho)}\\
& \le C\, \e^{1-4\sigma},
\endaligned
$$
if $\rho$ is sufficiently small.

\bigskip

\noindent{\bf Estimate of $I_2$}, where

\begin{equation}\label{I-2}
\aligned
& I_2  =  \sum_j  \int_{\partial\Omega}
\partial_{y_k} N_0^{\gamma \beta} (x, y) 
 \big(T_{ij}(y)\cdot \nabla_y \big)\big( \Phi_{\e, k}^{* \alpha\beta, z_j} \big) \cdot g_{ij}^\alpha (y, y/\e)
  \varphi_j (y) \, d\sigma  (y)\\
  & 
 - \sum_j  \int_{\partial\mathbb{H}_j^d}
\partial_{y_k} N_0^{\gamma \beta} (x, P^{-1}_j (y)) 
 \big(T_{i\ell}(P_j^{-1}(y))\cdot \nabla_y \big)\big( \Phi_{\e, k}^{* \alpha\beta, z_j} (y)\big) \cdot g_{i\ell}^\alpha (y, y/\e)
  \varphi_j (y) \, d\sigma  (y),
  \endaligned
 \end{equation}
where $\partial\mathbb{H}^d_j$ denotes the tangent plane for $\partial\Omega$ at $z_j$ and
$P_j^{-1}$ is the inverse of the projection map from $B(z_j, C r_j)\cap\partial\Omega$ to $\partial \mathbb{H}^d_j$.
Here we rely on the estimates
\begin{equation}\label{I-21}
\aligned
|\nabla_y^2 N_0 (x, y)| & \le C |x-y|^{-d},\\
|\nabla^2 \Phi_{\e, k}^{*\alpha\beta, z_j} | &\le C \e^{-1},
\endaligned
\end{equation}
as well as the observation that $|y- P^{-1}_j (y)|\le C r_j^2$ for $y\in B(\widetilde{x}_j, Cr_j)\cap\partial\Omega$.
 It is not hard to see that for $x\in D$,
\begin{equation}\label{I-22}
|I_2 (x)|
 \le  C\, \e^{-1} \sum_j \frac{r_j^{2+d-1}}{|x- \widetilde{x}_j|^{d-1}}
 \le C \e^{-1} \tau^{\frac{1+\rho}{2}} \sum_j \frac{r_j^{1-\rho+d-1}}{|x- \widetilde{x}_j|^{d-1}},
\end{equation}
which, by Lemma \ref{lemma-8.1},
 leads to (\ref{estimate-I}) for $k=2$.

\bigskip

\noindent{\bf Estimate of $I_3$}, where

\begin{equation}\label{I-3}
\aligned
 & I_3  =  \sum_j  \int_{\partial\mathbb{H}_j^d}
\partial_{y_k} N_0^{\gamma \beta} (x, P^{-1}_j (y)) 
 \big(T_{i\ell} (y)\cdot \nabla_y \big)\big( \Phi_{\e, k}^{* \alpha\beta, z_j} \big) \cdot g_{i\ell}^\alpha (y, y/\e)
  \varphi_j (y) \, d\sigma  (y)\\
  & - \sum_j  \int_{\partial\mathbb{H}_j^d}
\partial_{y_k} N_0^{\gamma \beta} (x, P^{-1}_j (y)) 
 \big(T_{i\ell} (z_j )\cdot \nabla_y \big)\big( \Phi_{\e, k}^{* \alpha\beta, z_j} \big) \cdot g_{i\ell}^\alpha (z_j, y/\e)
  \varphi_j (y) \, d\sigma  (y).
  \endaligned
 \end{equation}

It is easy to see that for $x\in D$,
$$
|I_3(x)|\le C \sum_j \frac{r_j^{1+d-1}}{|x-\widetilde{x}_j|^{d-1}}
\le C\,  \tau^{\frac{\rho}{2}} \sum_j \frac{r_j^{1-\rho +d-1}}{|x-\widetilde{x}_j|^{d-1}},
$$
which may be handled by Lemma \ref{lemma-8.1}.

\bigskip

\noindent{\bf Estimate of $I_4$}, where

\begin{equation}\label{I-4}
\aligned
I_4  = & \sum_j  \int_{\partial\mathbb{H}_j^d}
\partial_{y_k} N_0^{\gamma \beta} (x, P^{-1}_j (y)) 
 \big(T_{i\ell} (z_j)\cdot \nabla_y \big)\big( \Phi_{\e, k}^{* \alpha\beta, z_j} \big) \cdot g_{i\ell}^\alpha (z_j, y/\e)
  \varphi_j (y) \, d\sigma  (y)\\
  & -\sum_j  \int_{\partial\mathbb{H}_j^d}
\partial_{y_k} N_0^{\gamma \beta} (x, P^{-1}_j (y)) 
 \widetilde{g}_k^\beta (z_j)
  \varphi_j (y) \, d\sigma  (y).
  \endaligned
 \end{equation}

The estimate of $I_4$ uses the fact that  for each $j$, the function
$$
\big(T_{i\ell} (z_j)\cdot \nabla_y \big)\big( \Phi_{\e, k}^{* \alpha\beta, z_j} \big) \cdot g_{i\ell}^\alpha (z_j, y/\e)
$$
is of form $U(y/\e)$, where $U(x)$ is a smooth 1-periodic function whose mean value is
given by $\widetilde{g}_k^\beta (z_j)$.
Furthermore, the normal to the hyperplane $\partial \mathbb{H}_j^d$ is $n(z_j)$, which satisfies the
Diophantine condition (\ref{D-condition}) with constant 
$$
\kappa (z_j)\ge \frac{c\, \tau}{r_j}.
$$
It then follows from Proposition 2.1 in \cite{Armstrong-Prange-2016}  that if $x\in D$,
$$
|I_4(x)|\le C_N (\tau^{-1} \e)^{-N} \sum_j \frac{r_j^{d-1}}{|x-\widetilde{x}_j|^{d-1}},
$$
 for any $N\ge 1$.
 Since $\tau =\e^{1-\sigma}$, this implies that 
$$
\int_D |I_ 4(x)|^2\, dx \le C_N \e^{ N\sigma}.
$$

\bigskip

\noindent{\bf Estimate of $I_5$}, where

\begin{equation}\label{I-5}
\aligned
I_5=   & \sum_j  \int_{\partial\mathbb{H}_j^d}
\partial_{y_k} N_0^{\gamma \beta} (x, P^{-1}_j (y)) 
 \overline{g}_k^\beta (z_j)
  \varphi_j (y) \, d\sigma  (y)\\
  &-
  \sum_j  \int_{\partial\Omega}
\partial_{y_k} N_0^{\gamma \beta} (x, y) 
 \overline{g}_k^\beta (y)
  \varphi_j (y) \, d\sigma  (y).
  \endaligned
 \end{equation}

Finally, to estimate $I_5$, we use the regularity estimate for $\overline{g}$ in 
Theorem \ref{theorem-E} to obtain 
$$
\aligned
|\overline{g}(y) -\overline{g}(z_j)|
&\le C \left\{ \frac{r_j^2}{\big[\kappa(z_j)\big]^{2+\rho}} +\frac{r_j}{\big[\kappa(z_j)\big]^{1+\rho} }\right\}\\
&\le C \left\{ \frac{r_j^{4+\rho}}{\tau^{2+\rho}} 
+\frac{r_j^{2+\rho}}{\tau^{1+\rho}} \right\}\\
&\le \frac{C r_j^{2+\rho}}{\tau^{1+\rho}},
\endaligned
$$
for any $y\in B(\widetilde{x}_j, Cr_j)\cap\partial\Omega$, where $\rho\in (0, 1/2)$
and we also used the fact $r_j\le C\sqrt{\tau}$.
It follows that for any $x\in D$,
$$
\aligned
|I_5(x)|
 &\le C \sum_j \frac{r_j^d}{|x-\widetilde{x}_j|^{d-1}}
+C\tau^{-1-\rho} \sum_j \frac{r_j^{2+\rho +d-1}}{|x-\widetilde{x}_j|^{d-1}}\\
& \le C \sum_j \frac{r_j^d}{|x-\widetilde{x}_j|^{d-1}}
+C
\tau^{-\frac12} 
\sum_j \frac{r_j^{1-\rho +d-1}}{|x-\widetilde{x}_j|^{d-1}}.
\endaligned 
$$
As before, by applying  Lemma \ref{lemma-8.1} and choosing  $\rho>0$ sufficiently small,
we obtain the desired estimate  for $I_5$.
This completes the proof of (\ref{final-goal}) and thus of Theorem \ref{main-theorem-1}.

\begin{remark}\label{remark-8.1}
{\rm
Let $\Theta_t (x)$ be defined by (\ref{Theta}).
It follows from  the proof of Proposition \ref{prop-7.3} that for $x\in D$,
$$
\Theta_t (x)\le 
C \tau^t \int_{\partial\Omega} 
\frac{\big[ \mathcal{M}_{\partial\Omega} (\kappa^{-q} )\big]^{t/q} (y)}
{|x-y|^{d-1}} d\sigma (y),
$$
where $q=p^\prime<\frac{d-1}{d-2}$ and $t\ge 0$.
Let $u_\e$ and $u_0$ be solutions of (\ref{NP-0}) and (\ref{NP-h}), respectively.
An inspection of our proof of Theorem \ref{main-theorem-1} shows that for any $\sigma\in (0,1/2)$,
there exists a neighborhood $\Omega_\e$ of $\partial\Omega$ in $\Omega$ 
 such that
\begin{equation}\label{error-8.0}
|\Omega_\e|\le C\,  \e^{1-\sigma},
\end{equation}
and for $x\in \Omega \setminus \Omega_\e$,
\begin{equation}\label{error-8.1}
|u_\e (x) -u_0 (x)-E |\le
C \e^{\frac12 -4\sigma} 
\int_{\partial\Omega}
\frac{\big[ \mathcal{M}_{\partial\Omega} (\kappa^{-q} )\big]^{\frac{1-\rho}{q}} (y)}
{|x-y|^{d-1}} d\sigma (y),
\end{equation}
where $1<q<d-1$, $\rho=\rho(\sigma)>0$ is small, and $E$ is a constant.
The boundary layer $\Omega_\e$, which is given locally by the union of $B(\widetilde{x}_j, Cr_j)\cap\Omega$,
depends  only on the function $\kappa$ and $\Omega$.
Furthermore, if  $F(x)$ denotes  the integral in (\ref{error-8.1}),
then
$$
|F(x)|\le C(1+|\ln \delta(x)|)^{1/s^\prime} \left(\int_{\partial\Omega} 
\frac{\big[ \mathcal{M}_{\partial\Omega} (\kappa^{-q} )\big]^{\frac{s(1-\rho)}{q}} (y)}
{|x-y|^{d-1}} d\sigma (y)\right)^{1/s},
$$
where $q<s(1-\rho)<d-1$. Since $\kappa^{-1}\in L^s(\partial\Omega)$ for any $1<s<d-1$,
$\mathcal{M}_{\partial\Omega} (\kappa^{-q})\in L^{s/q}(\partial\Omega)$
for any $q<s<d-1$.
It follows that $F\in L^s(\Omega)$ for any $q<s\le d-1$.
This, together with (\ref{error-8.1}),
\begin{equation}\label{error-8.3}
\| u_\e -u_0-E\|_{L^s(\Omega\setminus \Omega_\e)}
\le C\, \e^{\frac12-4\sigma} \quad \text{ for any } 1<s\le d-1.
\end{equation}

Finally, assume that $\int_\omega u_\e =\int_\Omega u_0=0$.
Since $\| u_\e -u_0\|_{L^2(\Omega)}\le C_\sigma\, \e^{\frac12-\sigma}$ by Theorem \ref{main-theorem-1},
it follows from (\ref{error-8.3}) that
$|E|\le C\, \e^{\frac12-4\sigma}$.
As a result, estimates (\ref{error-8.1}) and (\ref{error-8.3}) hold with $E=0$.
}
\end{remark}

%%%%%%%%%%%%%%%%%%%%%%%%%%%%%%%%%%%%%%%%%%%%%%%%%%%%%%

%%%%%%%%%%%%%%%%%%%%%%%%%%%%%%%%%%%%%%%%%%%%%%%%%%%%%%%

\section{Higher-order convergence}
\setcounter{equation}{0}

In this section we use Theorem \ref{main-theorem-1} to establish a higher-order rate of convergence 
in the two-scale expansion for the Neumann problem,
\begin{equation}\label{N-High}
\left\{
\aligned
\mathcal{L}_\e (u_\e) &=F &\quad &\text{ in } \Omega,\\
\frac{\partial u_\e}{\partial \nu_\e} &= g &\quad &\text{ on } \partial\Omega,
\endaligned
\right.
\end{equation}
where $F$ and $g$ are smooth functions.
Our goal is to prove the following.

\begin{theorem}\label{theorem-9.1}
Suppose that $A$ and $\Omega$ satisfy the same conditions as in Theorem \ref{main-theorem-1}.
Let $u_\e$ be the solution of (\ref{N-High}) with $\int_{\Omega} u_\e=0$, and
$u_0$ the solution of the homogenized problem. Then there exists a function $v^{bl}$, independent of $\e$,
such that
\begin{equation}\label{9.1}
\| u_\e - u_0 -\e \chi_k (x/\e)\frac{\partial u_0}{\partial x_k}
-\e v^{bl} \|_{L^2(\Omega)}
\le C_\sigma \e^{\frac32-\sigma} \| u_0\|_{W^{3, \infty}(\Omega)},
\end{equation}
for any $\sigma\in (0,1/2)$,
where $C_\sigma$ depends only on $d$, $m$, $\sigma$, $A$ and $\Omega$.
Moreover, the function $v^{bl}$ is a solution to the Neumann problem
\begin{equation}\label{N-bl}
\mathcal{L}_0(v^{bl}) = F_* \quad \text{ in } \Omega \quad \text{ and } \quad
\frac{\partial v^{bl}}{\partial \nu_0} =g_*\quad \text{ on }\partial\Omega,
\end{equation}
where $F_*= -\overline{c}_{ki\ell}\frac{\partial^3 u_0}{\partial x_k \partial x_i \partial x_\ell}$ 
for some constants $\overline{c}_{ki\ell}$, and
$g_*$ satisfies
\begin{equation}\label{9.2}
\| g_*\|_{L^q(\partial\Omega)} \le C_q \| u_0\|_{W^{2, \infty}(\Omega)},
\end{equation}
for any $1<q<d-1$.
\end{theorem}

For simplicity of exposition we will drop the superscripts in this section.
Let 
\begin{equation}\label{b-ij}
b_{ij} =a_{ij} +a_{ik}\frac{\partial \chi_j}{\partial y_k} -\widehat{a}_{ij},
\end{equation}
where 
$(\chi_k$) are the (first-order) correctors for $\mathcal{L}_\e$ in $\rd$.
Since 
$$
\int_{\td} b_{ij}=0 \quad \text{ and } \quad \frac{\partial }{\partial y_i} \big( b_{ij}\big)=0,
$$
there exist 1-periodic functions $(\phi_{kij})$ such that
\begin{equation}\label{9.3}
b_{ij}=\frac{\partial }{\partial y_k} \big( \phi_{kij} \big) \quad 
\text{ and } \quad \phi_{kij}=-\phi_{ikj}.
\end{equation}
The  second-order correctors $(\chi_{k\ell})$ with $1\le k, \ell \le d$ are defined by
\begin{equation}\label{second}
\left\{
\aligned
 & -\frac{\partial}{\partial y_i}
\left\{ a_{ij}\frac{\partial \chi_{k\ell}}{\partial y_j} \right\}
=b_{k\ell} +\frac{\partial}{\partial y_i}
\big( a_{i\ell} \chi_k\big) \quad \text{ in } \rd,\\
& \chi_{k\ell} \text{ is  1-periodic and } \int_{\td} \chi_{k\ell}=0.
\endaligned
\right.
\end{equation}
Let
\begin{equation}\label{w-9}
w_\e =u_\e -u_0 -\e \chi^\e_k\frac{\partial u_0}{\partial x_k}
-\e^2 \chi^\e_{k\ell} \frac{\partial^2 u_0}{\partial x_k \partial x_\ell},
\end{equation}
where we have used the notation $f^\e(x)= f(x/\e)$.
A direct computation shows that
\begin{equation}\label{L-w}
\aligned
\mathcal{L}_\e (w_\e)
= &-\e \Big( \phi_{kij}^\e \delta_{j\ell}
-a_{ij}^\e \chi_k^\e \delta_{j\ell}
-a_{ij}^\e \big(\frac{\partial \chi_{k\ell}}{\partial x_j}\big)^\e \Big)
\frac{\partial^3 u_0}{\partial x_i\partial x_k \partial x_\ell}\\
&\qquad
+\e^2 \frac{\partial}{\partial x_i}
\left\{ a_{ij}^\e \chi_{k\ell}^\e \frac{\partial^3 u_0}{\partial x_j \partial x_k \partial x_\ell} \right\}.
\endaligned
\end{equation}

Let
\begin{equation}\label{c-k}
\left\{
\aligned
& c_{ki\ell}=
 \phi_{kij} \delta_{j\ell}
-a_{ij} \chi_k \delta_{j\ell}
-a_{ij} \frac{\partial \chi_{k\ell}}{\partial x_j}, \\
& \overline{c}_{ki\ell}
=\average_{\td} c_{ki\ell}.
\endaligned
\right.
\end{equation}
Note that by the definition of $\chi_{k\ell}$,
$$
\frac{\partial}{\partial x_i} \big( c_{ki\ell} \big)=0.
$$
It follows that there exist 1-periodic functions $f_{mki\ell}$ with
$1\le m, k, i, \ell \le d$ such that
\begin{equation}\label{f-m}
c_{ki\ell}-\overline{c}_{ki\ell} =\frac{\partial}{\partial y_m} \big( f_{m k i \ell}\big)
\quad \text{ and } \quad 
f_{mki\ell}=-f_{ikm\ell}.
\end{equation}
This allows us to rewrite (\ref{L-w}) as
\begin{equation}\label{L-w-1}
\aligned
\mathcal{L}_\e (w_\e) & =-\e \overline{c}_{ki\ell}\frac{\partial^3 u_0}
{\partial x_i\partial x_k \partial x_\ell}
-\e \left(\frac{\partial f_{mki\ell}}{\partial x_m}\right)^\e 
\frac{\partial^3 u_0}
{\partial x_i\partial x_k \partial x_\ell}\\
&\qquad
+\e^2 \frac{\partial}{\partial x_i}
\left\{ a_{ij}^\e \chi_{k\ell}^\e \frac{\partial^3 u_0}{\partial x_j \partial x_k \partial x_\ell} \right\}.
\endaligned
\end{equation}

Next we compute the conormal derivative of $w_\e$.
Again, a direct computation gives
\begin{equation}\label{c-normal}
\aligned
\frac{\partial w_\e}{\partial \nu_\e}
&=-n_i b_{ij}^\e \frac{\partial u_0}{\partial x_j}
-\e n_i a_{ij}^\e \chi_k^\e \frac{\partial^2 u_0}{\partial x_j\partial x_k}
-\e n_i a_{ij}^\e \left(\frac{\partial \chi_{k\ell}}{\partial x_j}\right)^\e \frac{\partial^2 u_0}{\partial x_k\partial x_\ell}\\
& \qquad\qquad
-\e^2 n_i a_{ij}^\e \chi_{k\ell}^\e \frac{\partial^3 u_0}{\partial x_j\partial x_k\partial x_\ell}.
\endaligned
\end{equation}
Using (\ref{9.3}) and (\ref{c-k}),
we further obtain
\begin{equation}\label{c-normal-1}
\frac{\partial w_\e}{\partial \nu_\e}
=-\e n_i \frac{\partial}{\partial x_k} \left( \phi_{kij}^\e \frac{\partial u_0}{\partial x_j} \right)
+\e n_i c^\e_{ki\ell} \frac{\partial^2 u_0}{\partial x_k\partial x_\ell}
-\e^2 n_i a_{ij}^\e \chi_{k\ell}^\e \frac{\partial^3 u_0}{\partial x_j\partial x_k\partial x_\ell}
\end{equation}

In view of (\ref{L-w-1}) and (\ref{c-normal-1}),
we split $w_\e-\average_{\Omega} w_\e$ as $w_\e^{(1)} +w_\e^{(2)} +w_\e^{(3)} +w_\e^{(4)}$,
where
\begin{equation}\label{w-9-1}
\left\{
\aligned
& \mathcal{L}_\e (w^{(1)}_\e) =0
&\quad & \text{ in } \Omega,\\
&\frac{\partial}{\partial \nu_\e}
\big(w_\e^{(1)}\big)
=-\e n_i \frac{\partial}{\partial x_k}
\left(\phi_{kij}^\e \frac{\partial u_0}{\partial x_j}\right)
 & \quad & \text{ on } \partial\Omega,
\endaligned
\right.
\end{equation}
\begin{equation}\label{w-9-2}
\left\{
\aligned
& \mathcal{L}_\e (w^{(2)}_\e) =-\e \overline{c}_{ki\ell} \frac{\partial^3 u_0}{\partial x_i \partial x_k \partial x_\ell }
&\quad & \text{ in } \Omega,\\
&\frac{\partial}{\partial \nu_\e}
\big(w_\e^{(2)}\big)
=\e n_i \overline{c}_{ki\ell} \frac{\partial^2 u_0}
{\partial x_k\partial x_\ell}  & \quad & \text{ on } \partial\Omega,
\endaligned
\right.
\end{equation}
\begin{equation}\label{w-9-3}
\left\{
\aligned
& \mathcal{L}_\e (w^{(3)}_\e) =
-\e \left(\frac{\partial f_{mki\ell}}{\partial x_m}\right)^\e 
\frac{\partial^3 u_0}
{\partial x_i\partial x_k \partial x_\ell} &\quad & \text{ in }\Omega,\\
& \frac{\partial}{\partial \nu_\e}
\big(w_\e^{(3)}\big)
=\e n_i \big( c^\e_{ki\ell} -\overline{c}_{ki\ell} \big)
\frac{\partial^2 u_0}{\partial x_k\partial x_\ell}& \quad & \text{ on } \partial\Omega,
\endaligned
\right.
\end{equation}
and
\begin{equation}\label{w-9-4}
\left\{
\aligned
& \mathcal{L}_\e (w^{(4)}_\e) =
\e^2 \frac{\partial}{\partial x_i}
\left\{ a_{ij}^\e \chi_{k\ell}^\e \frac{\partial^3 u_0}{\partial x_j \partial x_k \partial x_\ell} \right\}
 &\quad & \text{ in }\Omega,\\
& \frac{\partial}{\partial \nu_\e}
\big(w_\e^{(4)}\big)
=-\e^2 n_i a_{ij}^\e \chi_{k\ell}^\e \frac{\partial^3 u_0}{\partial x_j\partial x_k\partial x_\ell}
&\quad &\text{ on }\partial\Omega.
\endaligned
\right.
\end{equation}
We further require that
\begin{equation}\label{mean-w}
\int_{\Omega} w_\e^{(1)}=
\int_{\Omega} w_\e^{(2)} =\int_{\Omega} w_\e^{(3)}
=\int_{\Omega} w_\e^{(4)}=0.
\end{equation}

To proceed, we first note that by Poincar\'e inequality, (\ref{mean-w}) and energy estimates, 
\begin{equation}\label{w-4-estimate}
\|w_\e^{(4)}\|_{L^2(\Omega)}
\le C\, \|\nabla w^{(4)}_\e\|_{L^2(\Omega)}
\le C\, \e^2 \| \nabla^3 u_0\|_{L^2(\Omega)}.
\end{equation}
The solution $w_\e^{(3)}$ may be handled in a similar  manner.
To see this we use the skew-symmetry of $f_{mki\ell}$ in $m$ and $i$ to write
the RHS of the equation as
$$
-\e^2\frac{\partial}{\partial x_m}
\left( f_{m k i \ell}^\e \frac{\partial^3 u_0}{\partial x_i \partial x_k \partial x_\ell}\right),
$$
while the Neumann data for $w_\e^{(3)}$ may be written as
$$
\frac {\e^2}{2} \left(n_i \frac{\partial }{\partial x_m}
-n_m \frac{\partial}{\partial x_i} \right)
\left(f^\e_{m k i \ell} \frac{\partial^2 u_0}{\partial x_k \partial x_\ell}\right)
-\e^2 n_i f_{mki\ell}^\e \frac{\partial^3 u_0}{\partial x_k \partial x_\ell \partial x_m}.
$$
As a result, we obtain
\begin{equation}\label{w-3-estimate}
\aligned
\| w_\e^{(3)}\|_{L^2(\Omega)}
&\le C \|\nabla w_\e^{(3)} \|_{L^2(\Omega)}
\le C \, \e^2 \left\{ \|\nabla^3 u_0\|_{L^2(\Omega)}
+\| f^\e \nabla^2 u_0\|_{H^{\frac12}(\partial\Omega)}\right\}\\
&\le C \e^{\frac32} \| u_0\|_{W^{3, \infty}(\Omega)}.
\endaligned
\end{equation}

Next, we observe that $w_\e^{(2)}$ may be dealt with by the classical  homogenization results for
$\mathcal{L}_\e$. Indeed, let $v_0^{(2)}$ be the solution of
\begin{equation}\label{v-0-1}
\left\{
\aligned
& \mathcal{L}_0 (v^{(2)}_0)=-\overline{c}_{ki\ell} \frac{\partial^3 u_0}{\partial x_i \partial x_k \partial x_\ell}
&\quad & \text{ in } \Omega,\\
&
\frac{\partial v_0^{(2)}}{\partial \nu_0}
=n_i \overline{c}_{k i \ell} \frac{\partial^2 u_0}{\partial x_k \partial x_\ell} 
&\quad & \text{ on } \partial\Omega,
\endaligned
\right.
\end{equation}
with $\int_\Omega v_0^{(2)}=0$. It is well known that
\begin{equation}\label{v-1-estimate}
\| w_\e^{(2)} -\e v_0^{(2)}\|_{L^2(\Omega)} \le C\, \e^2 \| u_0\|_{W^{3, \infty}(\Omega)}.
\end{equation}

It remains to estimate the solution $w_\e^{(1)}$, which will be handled by using
Theorem \ref{main-theorem-1}.
Observe that  by the skew-symmetry of $\phi_{kij}$ in $k$ and $i$,
the Neumann data  of $ w_\e^{(1)}$ may be written as
\begin{equation}\label{data-9}
-\frac{\e}{2} \Big(T_{ik}\cdot \nabla\Big)
\left(\phi_{kij}^\e \frac{\partial u_0}{\partial x_j}\right),
\end{equation}
where $T_{ik}=n_ie_k -n_k e_i$.
This allows us to apply Theorem \ref{main-theorem-1} to deduce that
\begin{equation}\label{w-1-estimate}
\| w_\e -\e v_0^{(1)}\|_{L^2(\Omega)}
\le C_\sigma \, \e^{\frac32 -\sigma} \| u_0\|_{W^{2, \infty}(\Omega)},
\end{equation}
for any $\sigma\in (0, 1/2)$,
where $v_0^{(1)}$ is a solution of the Neumann problem
\begin{equation}\label{N-v-1}
\left\{
\aligned
&\mathcal{L}_0 (v_0^{(1)}) =0  &\quad  & \text{ in } \Omega,\\
&\frac{\partial}{\partial \nu_0} \big( v_0^{(1)}\big)
=(T_{ij}\cdot \nabla \big)(\overline{g}_{ij}) & \quad  &\text{ in } \partial\Omega,
\endaligned
\right.
\end{equation}
and $\overline{g}_{ij}\in W^{1, q}(\partial\Omega)$  for any $1<q<d-1$.
We remark that the explicit dependence on the $W^{2, \infty}(\Omega)$ norm of
$u_0$ in the RHS of (\ref{w-1-estimate})
follows from the proof of Theorem \ref{main-theorem-1}.
The key observation is that the fast variable $y$ in the Neumann data (\ref{data-9})
is separated from the slow variable $x$.

Let $v^{bl}= v_0^{(1)} + v_0^{(2)}$.
In view of (\ref{w-4-estimate}), (\ref{w-3-estimate}), (\ref{v-1-estimate}) and
(\ref{w-1-estimate}), we have proved that 
\begin{equation}\label{9-100}
\| w_\e -\average_{\Omega} w_\e -\e v^{bl}\|_{L^2(\Omega)} \le C _\sigma \, \e^{\frac32-\sigma} 
\| u_0\|_{W^{3, \infty}(\Omega)}.
\end{equation}
Finally, we note that since $\int_\Omega u_\e =\int_\Omega u_0=0$,
$$
\aligned
\Big|\average_{\Omega} w_\e\Big|
 &\le  C \, \e\, \Big|\int_{\Omega} \chi_k(x/\e) \frac{\partial u_0}{\partial x_k}\, dx  \Big|
+C  \e^2 \| \nabla^2 u_0\|_\infty\\
 &\le C\, \e^{2}  \|u_0\|_{W^{2, \infty}(\Omega)},
 \endaligned
 $$
 where the last step follows from the fact that $\chi_k$ is periodic with mean value zero. 
 This, together with (\ref{9-100}), yields  the estimate (\ref{9.1})
and thus completes the proof of Theorem \ref{theorem-9.1}.

\bibliography{Shen-Zhuge-2.bbl}

\medskip

\begin{flushleft}
Zhongwei Shen,
 Department of Mathematics,
University of Kentucky,
Lexington, Kentucky 40506,
USA. \ \ 
% Fax: 1-859-257-4078.
E-mail: zshen2@uky.edu
\end{flushleft}

\begin{flushleft}
Jinping Zhuge, 
Department of Mathematics,
University of Kentucky,
Lexington, Kentucky 40506,
USA. \ \ 
E-mail: jinping.zhuge@uky.edu

\end{flushleft}
\medskip

\noindent \today

\end{document}